\documentclass[a4paper, 12pt]{amsart}
\usepackage{amscd,amssymb,epsfig,amsxtra,amsfonts}
\usepackage{amsmath}
\usepackage{stmaryrd, mathrsfs}
\usepackage{ upgreek }
\usepackage{verbatim}
\usepackage{chngcntr}
\usepackage{epsfig} 
\usepackage{diagrams}
\counterwithin{equation}{section}
\newcommand{\ms}[1]{\mbox{\tiny$#1$}}
\newcommand{\epsh}[2]
          {\begin{array}{c} \hspace{-1.3mm}
         \raisebox{-4pt}{\epsfig{figure=#1,height=#2}}
         \hspace{-1.9mm}\end{array}}

\newcommand{\coev}{\stackrel{\longrightarrow}{\operatorname{coev}}}
\newcommand{\ev}{\stackrel{\longrightarrow}{\operatorname{ev}}}
\newcommand{\tev}{\stackrel{\longleftarrow}{\operatorname{ev}}}
\newcommand{\tcoev}{\stackrel{\longleftarrow}{\operatorname{coev}}}

\newcommand{\Id}{\operatorname{Id}}

\newcommand{\End}{\operatorname{End}}
\newcommand{\Hom}{\operatorname{Hom}}

\newcommand{\tr}{\operatorname{tr}}

\newcommand{\Vect}{\operatorname{Vect}}

\newcommand{\mtr}{\operatorname{\mathsf{t}}}
\newtheorem{theo}{Theorem}[section]
\newtheorem{Def}[theo]{Definition}
\newtheorem{The}[theo]{Theorem}
\newtheorem{Lem}[theo]{Lemma}
\newtheorem{Exem}[theo]{Example}
\newtheorem{Pro}[theo]{Proposition}
\newtheorem{HQ}[theo]{Corollary}

\begin{document}
\title{Modified trace from pivotal Hopf $G$-coalgebra}     


\author{Ngoc Phu HA}        
\address{Universit\'e de Bretagne Sud $\&$ Laboratoire de Math\'ematiques de Bretagne Atlantique, UMR CNRS 6205, Centre de Recherche, Campus de Tohannic, BP 573     
F-56017 Vannes, France }
\email{ngoc-phu.ha@univ-ubs.fr}
\address{Hung Vuong university, Department of mathematics-informatics, Nong Trang, Viet Tri, Phu Tho, Viet Nam}
\email{ngocphu.ha@hvu.edu.vn}

\maketitle

\begin{abstract}
In a recent paper the authors Beliakova, Blanchet and Gainutdinov have shown that the modified trace on the category $H$-pmod of the projective modules corresponds to the symmetrised integral on the finite dimensional pivotal Hopf algebra $H$. We generalize this fact to the context of $G$-graded categories and Hopf $G$-coalgebra studied by Turaev-Virelizier. We show that the symmetrised $G$-integral on a finite type pivotal Hopf $G$-coalgebra induces a modified trace in the associated $G$-graded category.

\end{abstract}
\vspace{25pt}



\section{Introduction}
The notion of a modified trace was introduced by N. Geer, J. Kujawa and B. Patureau-Mirand in the article \cite{NgJkBp13}. This is one of the topological tools which can be used first to renormalize the Reshetikhin-Turaev invariant of links. Later F. Costantino, N. Geer and B. Patureau-Mirand used the modified trace to construct a class of invariants of $3$-manifolds (CGP invariant) via link surgery presentations (see \cite{FcNgBp14}). The modified trace is also used to construct invariants of $3$-manifolds of type Reshetikhin-Turaev from quantum group associated to the super Lie algebra $\mathfrak{sl}(2|1)$ (see \cite{Ha16}) and for constructing the logarithmic invariant of type Hennings (see \cite{BBGe17}). In order to construct invariant of $3$-manifolds, Hennings proposed a method based on the theory of integral for a finite dimensional Hopf algebra (see \cite{Henning96}). The notion of integral was introduced by Larson and Sweedler in \cite{LaSweedler69} and is studied in the book \cite{Ra12} of Radford. It is known that under some assumption, both the space of modified trace and that of integral are one dimensional (see \cite{GR17, Ra12}). A close relation between the modified trace and the integral has been established recently in \cite{BBG17}. The authors proved that a symmetrised integral for a finite dimensional pivotal Hopf algebra gives a modified trace $t$ on $H$-pmod with an explicit formula. We would like to adapt these results to the unrestricted quantum groups at root of unity. They are infinite dimensional Hopf algebra but can be understood as a Hopf $G$-coalgebra organized into a bundles of algebra over a Lie group. For a finite type Hopf $G$-coalgebra $H=\left(H_\alpha\right)_{\alpha\in G}$ there exists a family of linear forms on $H_{\alpha}$, called $G$-integral (see \cite{Vire02}). The aim of this article is to establish a correspondence between the $G$-integral for the finite type unimodular pivotal Hopf $G$-coalgebra $H$ and the modified trace in the associated $G$-graded category $H$-mod.
We introduce now these two notions.
\subsection*{$G$-integral} Let $H=(\{H_{\alpha}, m_{\alpha}, 1_{\alpha}\}, \Delta, \varepsilon, S)$ be a Hopf $G$-coalgebra over a field $\Bbbk$ (see in Section \ref{Section pivot}). A {\em right $G$-integral} for the Hopf $G$-coalgebra $H$ is a family of $\Bbbk$-linear forms $\mu=(\mu_{\alpha}:\ H_{\alpha} \rightarrow \Bbbk)_{\alpha \in G}$ satisfying
\begin{equation}\label{right int}
(\mu_{\alpha} \otimes \Id_{H_{\beta}})\Delta_{\alpha, \beta}(x)=\mu_{\alpha\beta}(x)1_{\beta}\ \text{for any}\ x \in H_{\alpha\beta}.
\end{equation}
Similarly, a {\em left $G$-integral} $\mu_{\alpha}^{l}\in \prod_{\alpha\in G} H_{\alpha}^{*}$ satisfies 
\begin{equation*}
(\Id_{H_{\alpha}} \otimes \mu_{\beta}^{l})\Delta_{\alpha, \beta}(x)=\mu_{\alpha\beta}^{l}(x)1_{\alpha}\ \text{for any}\ x \in H_{\alpha\beta}.
\end{equation*}
The linear form $\mu_1$ is an usual right integral for the Hopf algebra $H_1$ (see e.g \cite{Ra12}). If $H$ is a finite type Hopf $G$-coalgebra, i.e. a Hopf $G$-coalgebra in which $\dim(H_{\alpha})< + \infty $ for any $\alpha\in G$, the space of right (resp. left) $G$-integral is known to be 1-dimensional (see e.g \cite{Vire02}).\\
A {\em pivotal Hopf $G$-coalgebra} is a pair $(H, g)$, where the pivot is the family $g=(g_{\alpha})_{\alpha \in G}\in \prod_{\alpha\in G}H_{\alpha}$ satisfying
$\Delta_{\alpha, \beta}(g_{\alpha \beta})=g_{\alpha} \otimes g_{\beta}$ for any $\alpha, \beta \in G,\  \varepsilon(g_1)=1_{\Bbbk}$, and $
S_{\alpha^{-1}}S_{\alpha}(x)=g_{\alpha}xg_{\alpha}^{-1}$ for any $ x \in H_{\alpha}$.
Note that $g^{-1}=(S_{\alpha^{-1}}(g_{\alpha^{-1}}))_{\alpha \in G}$, i.e. $g_{\alpha}^{-1}=S_{\alpha^{-1}}(g_{\alpha^{-1}})$ (see e.g \cite{Vire02}). In particular, $g_1$ is a pivotal element for $H_1$ and $g_1$ is invertible with $g_1^{-1}=S_1(g_1),\ \varepsilon(g_1)=0$ (see e.g \cite{ChKa95}).\\
The {\em symmetrised right $G$-integral} on $(H, g)$ associated with $\mu$ is the family $\widetilde{\mu}=(\widetilde{\mu}_{\alpha})_{\alpha\in G}\in \prod_{\alpha\in G}H_{\alpha}^*$ defined by
\begin{equation*}
\widetilde{\mu}_{\alpha}(x):=\mu_{\alpha}(g_{\alpha}x)\ \text{for any}\  x\in H_{\alpha}.
\end{equation*}
Similarly, a {\em symmetrised left $G$-integral} on $(H, g)$ is
\begin{equation} \label{left int}
\widetilde{\mu}_{\alpha}^{l}(x):=\mu_{\alpha}^{l}(g_{\alpha}^{-1}x)\ \text{for any}\ x\in H_{\alpha}.
\end{equation}
In the case $(H, g)$ is unimodular, i.e. $H_1$ is unimodular, we show that the symmetrised $G$-integrals are symmetric linear forms on $H$ and they are non-degenerate (see Proposition \ref{sr int sym}).
\subsection*{Modified trace} 
Let $\mathcal{C}$ be a {\em pivotal} $\Bbbk$-linear category \cite{BePM12}. Let $\textbf{Proj}(\mathcal{C})$ be the tensor ideal of projective objects of $\mathcal{C}$. 
A {\em modified trace} on ideal $\textbf{Proj}(\mathcal{C})$ is a family of $\Bbbk$-linear forms
$\mtr=\{\mtr_{P}:\ \End_{\mathcal{C}}(P) \rightarrow \Bbbk\}_{P \in \textbf{Proj}(\mathcal{C})}$ satisfying the cyclicity property and the partial trace property (see in Section \ref{Section modified trace}).

\subsection*{Main results} Let $(H,g)=(\{H_{\alpha}, m_{\alpha}, 1_{\alpha}\}, \Delta, \varepsilon, S, g)$ be a finite type unimodular pivotal Hopf $G$-coalgebra. If $\mtr$ is a right (resp. left) modified trace on $H$-pmod, it defines a family of linear forms $\lambda^{\mtr}=(\lambda_{\alpha}^{\mtr})_{\alpha\in G}\in \prod_{\alpha\in G}H_{\alpha}^{*}$ by $\lambda_{\alpha}^{\mtr}(h)=\mtr_{H_{\alpha}}(R_h)$ for $h\in H_{\alpha}$, $H_{\alpha}$ is a projective object of $H$-mod and $R_h$ is the right multiplication of $H_{\alpha}$.
\begin{The} \label{main theorem}
The application $\mtr\mapsto \lambda^{\mtr}$ defined above gives a bijection between the space of right (resp. left) modified traces and the space of symmetrised right (resp. left) $G$-integrals.\\
Furthermore, $(H, g)$ is $G$-unibalanced if and only if the right modified trace is also left.
\end{The}

The paper contains five section. In section 2 we recall some definition and results for a Hopf $G$-coalgebra, we also define a pivotal Hopf $G$-coalgebra, a symmetrised $G$-integral for a pivotal Hopf $G$-coalgebra $H$ and prove that the symmetrised $G$-integrals are symmetric non-degenerate forms on $H$. Section 3 recall some results about modified traces and the proof of Reduction Lemma in the context of $G$-graded categories. In section 4 we present the decomposition of tensor product $H_{\alpha}\otimes H_{\beta}$ and the proof of the main theorem. In section 5 we give an application of the main theorem in the case associated to a quantization of the Lie algebra $\mathfrak{sl}(2)$.
\subsection*{Acknowledgments}
I would like to thank B. Patureau-Mirand, my thesis advisor, who helped me with this work, and who gave me the motivation to study mathematics. The author would also like to thank the professors and friends in the laboratory LMBA of the Universit\'e de Bretagne Sud and the Centre Henri Lebesgue ANR-11-LABX-0020-01 for creating an attractive mathematical environment.

\section{Pivotal Hopf $G$-coalgebra}\label{Section pivot}
In this section, we recall some facts about Hopf $G$-coalgebra. For details see \cite{Turaev10,Vire02}. We then define a pivotal Hopf $G$-coalgebra, a symmetrised $G$-integral and give some of its properties.

\subsection{Pivotal Hopf $G$-coalgebra}
\subsubsection{Hopf $G$-coalgebra}
\begin{Def}
Let $G$ be a multiplicative group. A $G$-coalgebra over a field  $\Bbbk$ is a family $C=\{C_{\alpha}\}_{\alpha \in G}$ of $\Bbbk$-spaces endowed with a family $\Delta=\{\Delta_{\alpha, \beta}:\ C_{\alpha\beta}\rightarrow C_{\alpha}\otimes C_{\beta}\}_{\alpha, \beta \in G}$ of $\Bbbk$-linear maps (the coproduct) and a $\Bbbk$-linear map $\varepsilon:\ C_1 \rightarrow \Bbbk$ (the counit) such that
\begin{itemize}
\item $\Delta$ is coassociative, i.e. for any $\alpha, \beta, \gamma \in G$, $$(\Delta_{\alpha, \beta}\otimes \Id_{C_{\gamma}})\Delta_{\alpha\beta, \gamma}=(\Id_{C_{\alpha}}\otimes \Delta_{\beta, \gamma})\Delta_{\alpha, \beta\gamma},$$
\item for all $\alpha \in G, \ (\Id_{C_{\alpha}}\otimes \varepsilon)\Delta_{\alpha,1}=\Id_{C_{\alpha}}=(\varepsilon \otimes \Id_{C_{\alpha}})\Delta_{1,\alpha}$.
\end{itemize}
A Hopf $G$-coalgebra is a $G$-coalgebra $H=(\{H_{\alpha}\}_{\alpha \in G}, \Delta, \varepsilon)$ endowed with a family $S=\{S_{\alpha}:\ H_{\alpha}\rightarrow H_{\alpha^{-1}}\}_{\alpha\in G}$ of $\Bbbk$-linear maps (the antipode) such that
\begin{itemize}
\item each $H_{\alpha}$ is an algebra with product $m_{\alpha}$ and unit element $1_{\alpha}\in H_{\alpha}$,
\item $\varepsilon:\ H_{1} \rightarrow \Bbbk$ and $\Delta_{\alpha, \beta}:\ H_{\alpha\beta}\rightarrow H_{\alpha}\otimes H_{\beta}$ are algebra homomorphisms for all $\alpha, \beta \in G$,
\item for any $\alpha \in G$
\begin{equation*}
m_{\alpha}(S_{\alpha^{-1}}\otimes \Id_{H_{\alpha}})\Delta_{\alpha^{-1}, \alpha}=\varepsilon 1_{\alpha}=m_{\alpha}(\Id_{H_{\alpha}}\otimes S_{\alpha^{-1}})\Delta_{\alpha, \alpha^{-1}}.
\end{equation*}
\end{itemize}
\end{Def}
The antipode automatically satisfies additional property:
\begin{Lem} \label{tc antip}
Given a Hopf $G$-coalgebra $H=(\{H_{\alpha}\}_{\alpha \in G}, \Delta, \varepsilon, S)$, then
\begin{enumerate}
\item $S_{\alpha}(xy)=S_{\alpha}(y)S_{\alpha}(x)$ for any $x, y \in H_{\alpha}$,
\item $S_{\alpha}(1_{\alpha})=1_{\alpha^{-1}}$,
\item $\Delta_{\beta^{-1}, \alpha^{-1}}S_{\alpha\beta}=\tau(S_{\alpha}\otimes S_{\beta})\Delta_{\alpha, \beta}$ where $\tau:\ H_{\alpha^{-1}} \otimes  H_{\beta^{-1}} \rightarrow H_{\beta^{-1}} \otimes H_{\alpha^{-1}}$ is the flip switching the two factors of $H_{\alpha^{-1}} \times  H_{\beta^{-1}}$,
\item $\varepsilon S_{1}=\varepsilon$.
\end{enumerate}
\end{Lem}
\subsubsection{Graphical axioms for Hopf $G$-coalgebras}\label{the graphical representation}
We will use the diagrams for the structural maps and the identities corresponding to the Hopf $G$-coalgebra $H=(H_{\alpha})_{\alpha\in G}$. For simplicity we write $\alpha$ instead of $H_{\alpha}$ in the diagrams. Figure \ref{Fig cac toan tu cau truc} presents the structural maps of the Hopf $G$-coalgebra which are the product, coproduct, unity, counity and the antipode, respectively.
\begin{figure} 
$$m_{\alpha}= \epsh{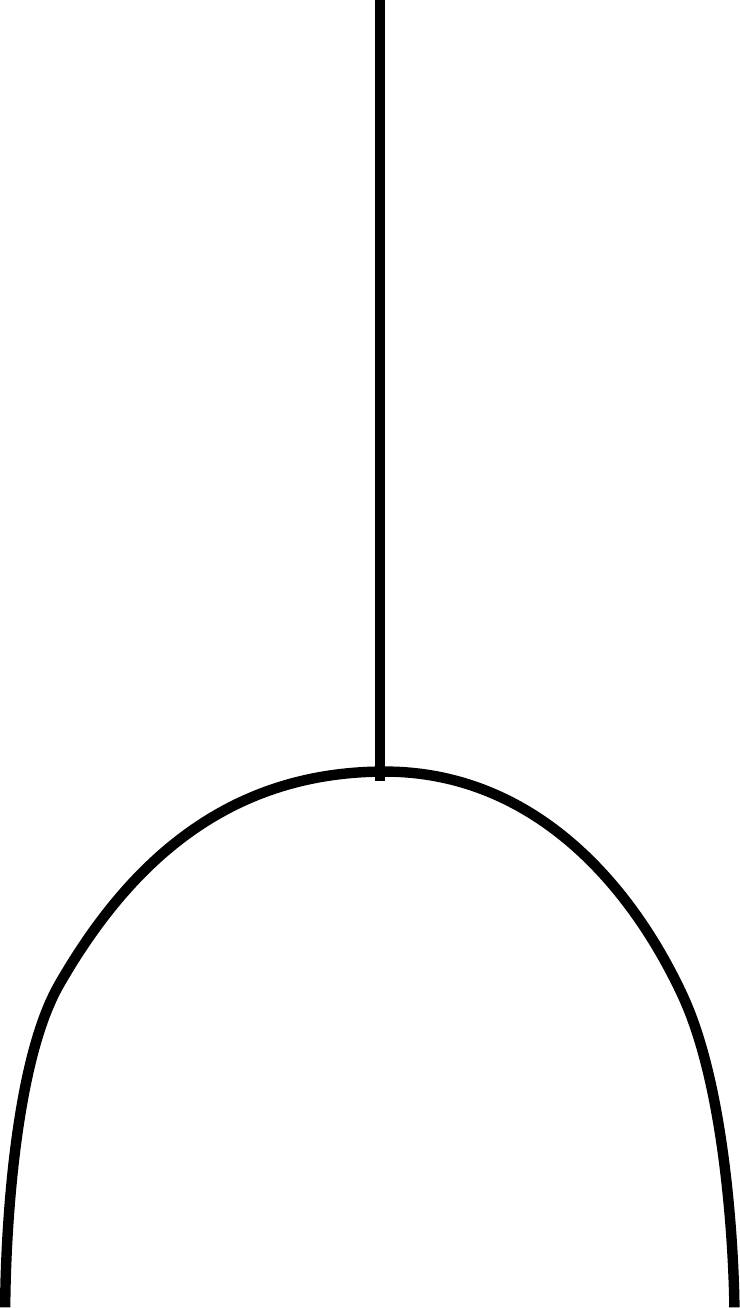}{9ex}
      \put(-2,-27){\ms{\alpha}}  \put(-29,-27){\ms{\alpha}}
			\put(-15,28){\ms{\alpha}}, \ 
		\Delta_{\alpha, \beta}=\epsh{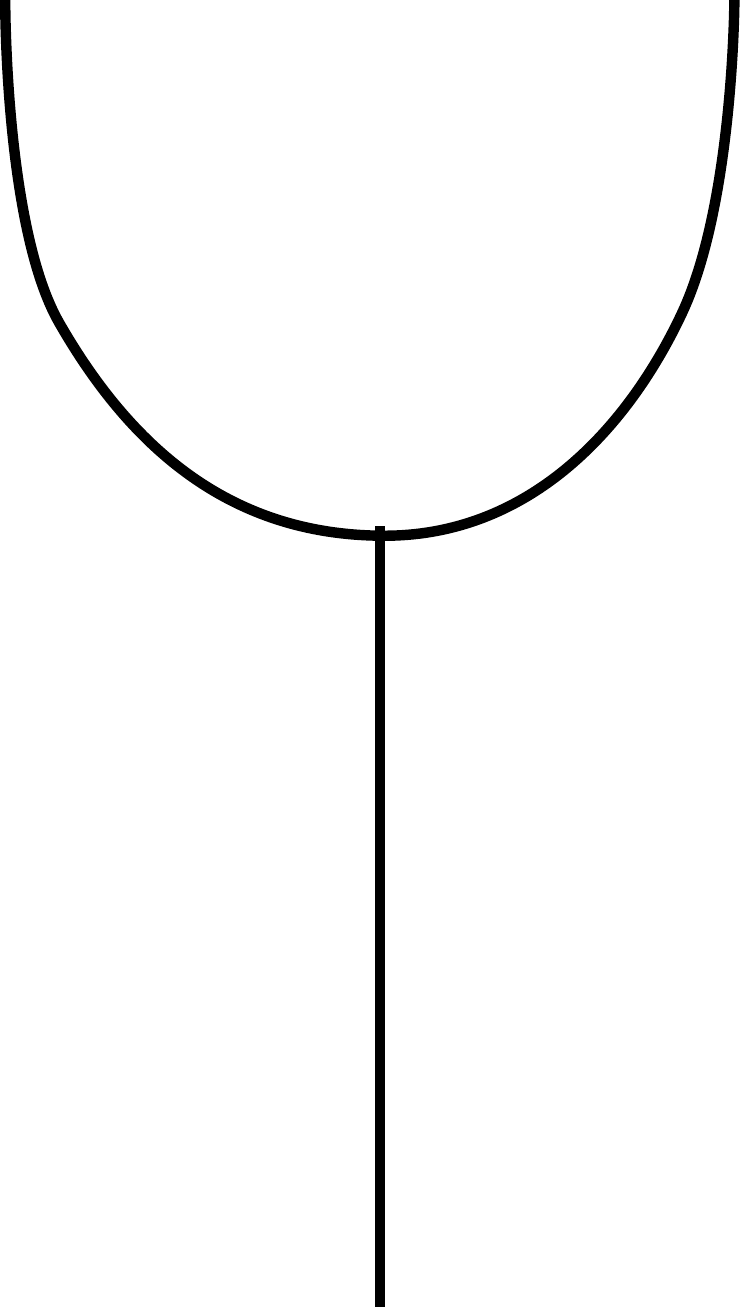}{9ex} 
			\put(-4,29){\ms{\beta}}  \put(-28,29){\ms{\alpha}}
			\put(-18,-28){\ms{\alpha\beta}}, \ 
		\eta_{\alpha}=\epsh{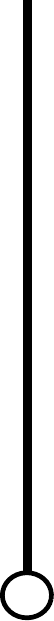}{5ex}\put(-4,20){\ms{\alpha}}\ , 	\ 
		\varepsilon=\epsh{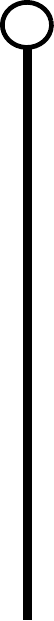}{5ex}\put(-2,-20){\ms{1}}\ ,  \ 
		S_{\alpha} =\epsh{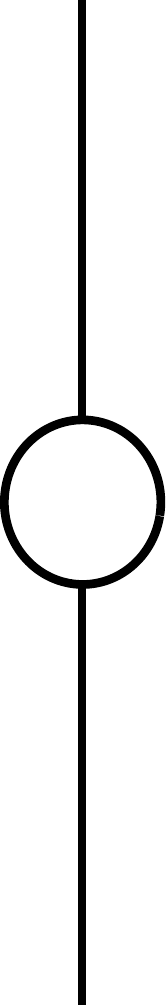}{9ex}\put(-6, -27){\ms{\alpha}}\put(-6, 29){\ms{\alpha^{-1}}}\ .$$
\caption{The structural maps}\label{Fig cac toan tu cau truc}
\end{figure}
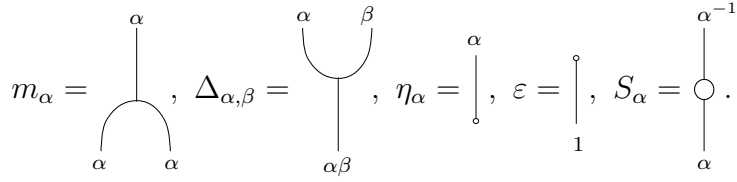
Note that these maps are in the category $\Vect_{\Bbbk}$ of finite dimensional vector spaces over a field $\Bbbk$.

The identity of the coassociativity and the algebra homomorphism $\Delta_{\alpha,\beta}$ are defined as in Figure \ref{Fig coassociativity and algebra homomorphism}.  
\begin{figure}
$$\epsh{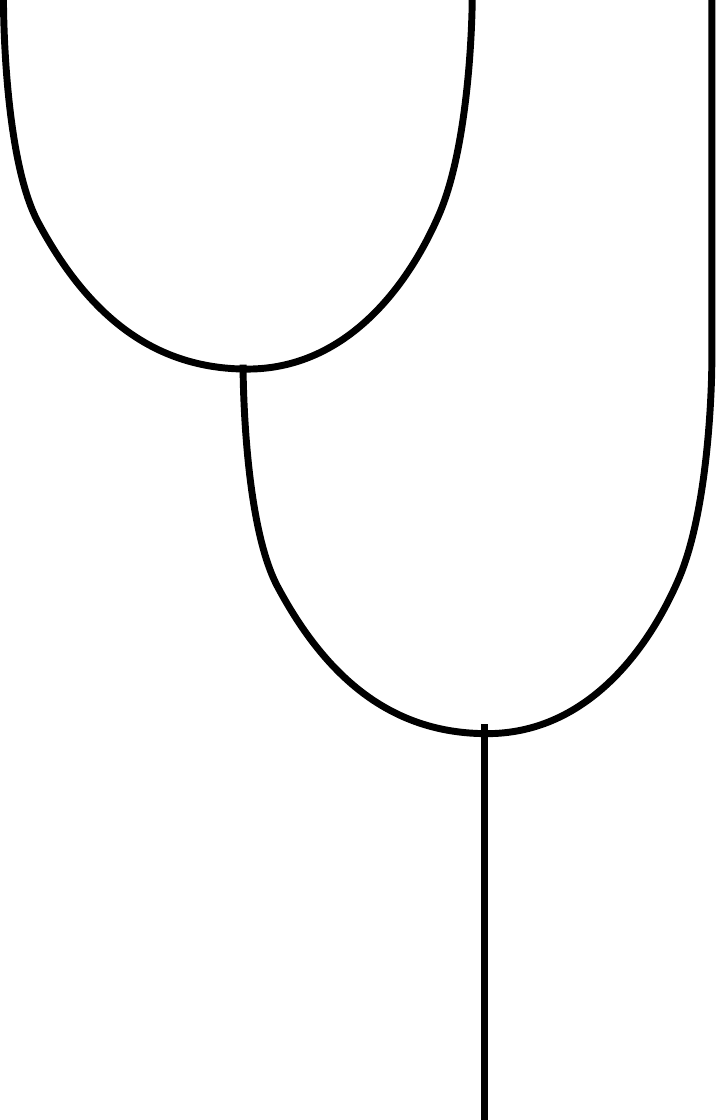}{10ex}\put(-18,-31){\ms{\alpha\beta\gamma}} \put(-3,33){\ms{\gamma}}\put(-35,32){\ms{\alpha}}\put(-14,32){\ms{\beta}} \quad =\quad\epsh{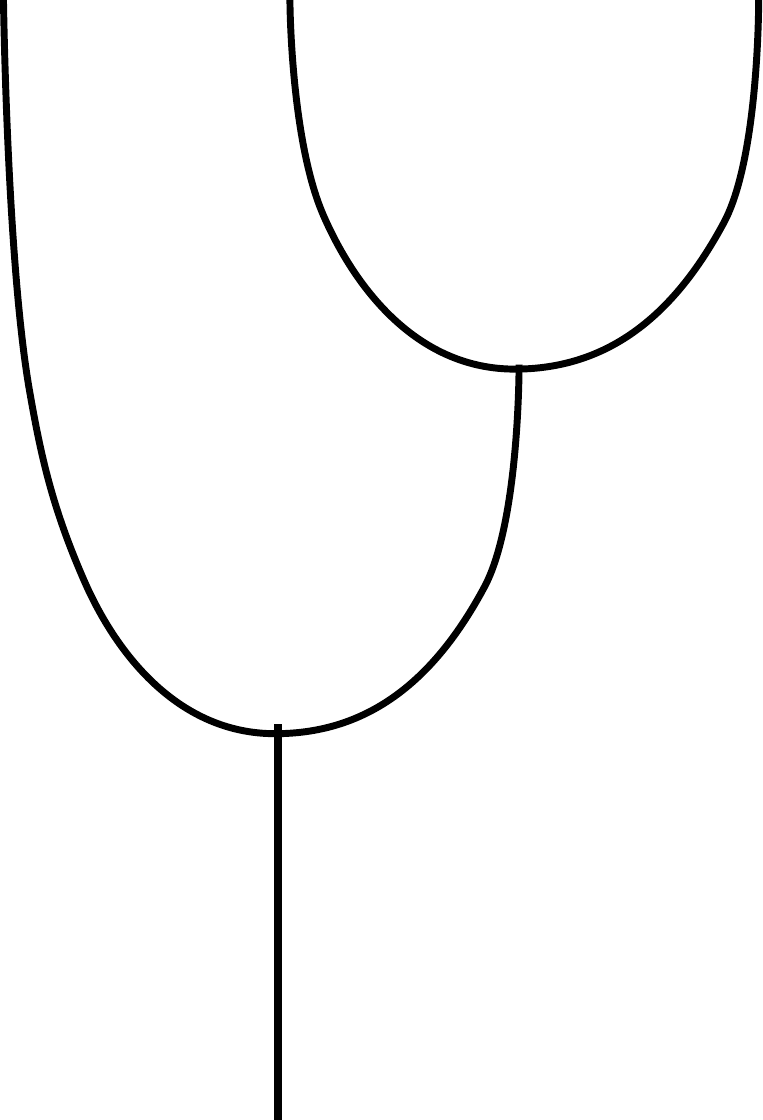}{10ex}\put(-28,-31){\ms{\alpha\beta\gamma}}
\put(-38,32){\ms{\alpha}} \put(-24,32){\ms{\beta}}\put(-3,33){\ms{\gamma}} \quad ,
 \quad \epsh{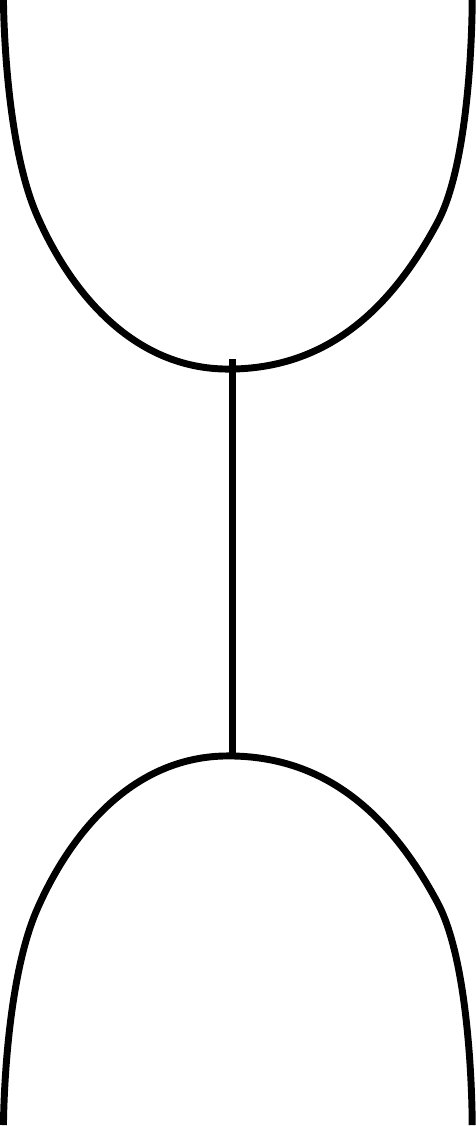}{10ex} \put(-28,-31){\ms{\alpha\beta}} \put(-2,-31){\ms{\alpha\beta}}
 \put(-24,31){\ms{\alpha}} \put(-2,31){\ms{\beta}}
 \quad=\quad\epsh{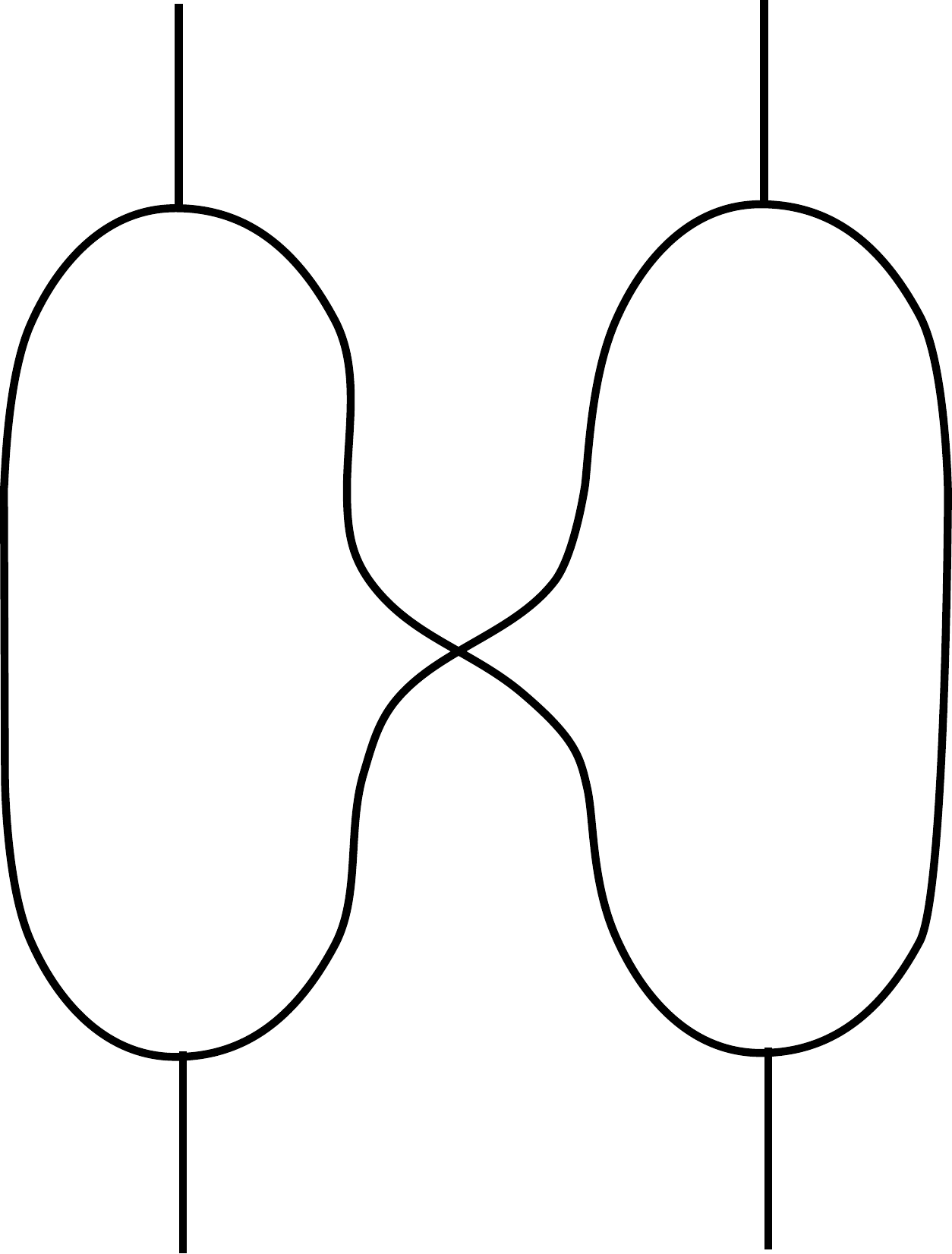}{10ex}  
 \put(-36,-31){\ms{\alpha\beta}} \put(-10,-31){\ms{\alpha\beta}}
 \put(-34,31){\ms{\alpha}} \put(-10,31){\ms{\beta}} \put(-45,-12){\ms{\alpha}} \put(-25,-14){\ms{\beta}} \put(-12,-10){\ms{\alpha}} \put(1,-10){\ms{\beta}}  \quad .$$
\caption{The coassociativity and algebra homomorphism $\Delta_{\alpha,\beta}$}\label{Fig coassociativity and algebra homomorphism}
\end{figure}
The antipode properties are shown in Figure \ref{Fig antipode axioms}.
\begin{figure}
$$\epsh{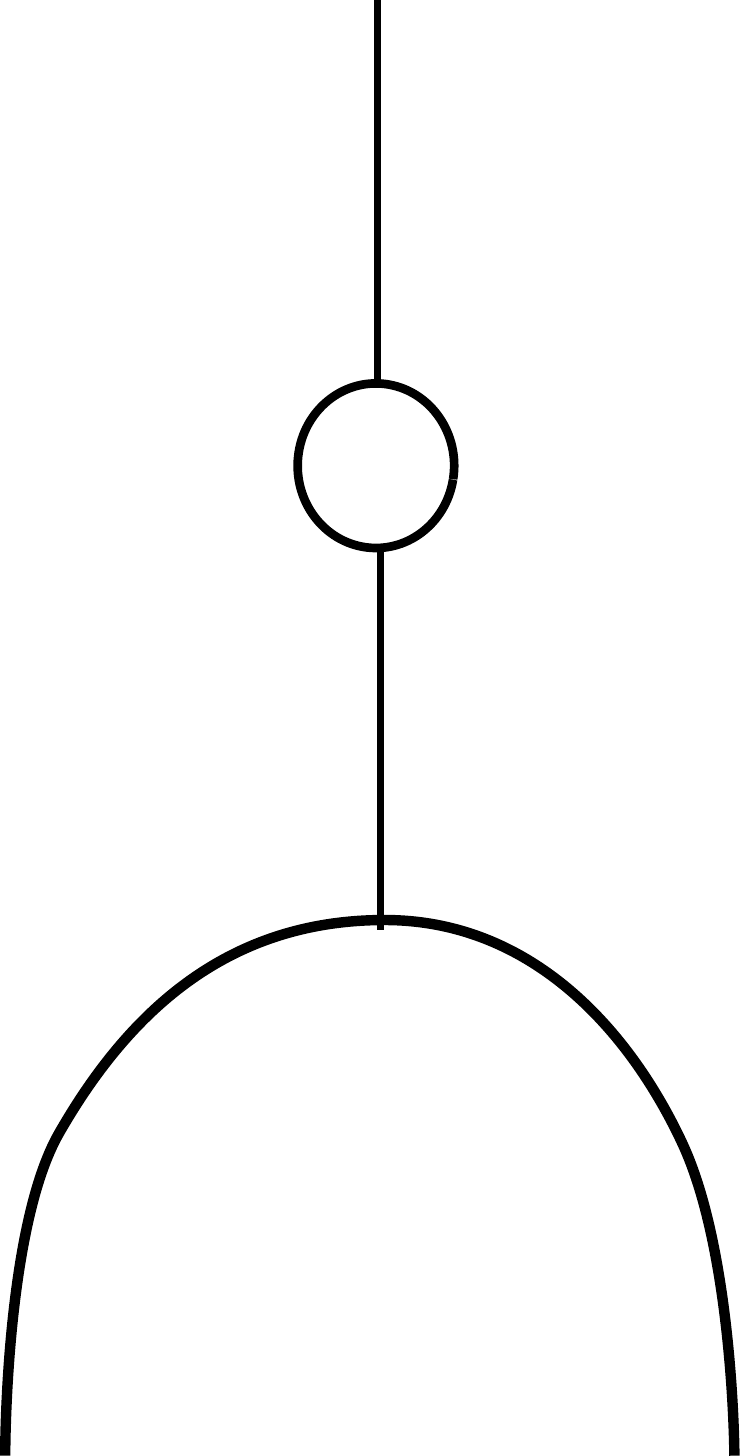}{10ex}\put(-29,-29){\ms{\alpha}} \put(-2,-29){\ms{\alpha}} \put(-15,31){\ms{\alpha^{-1}}}
=\epsh{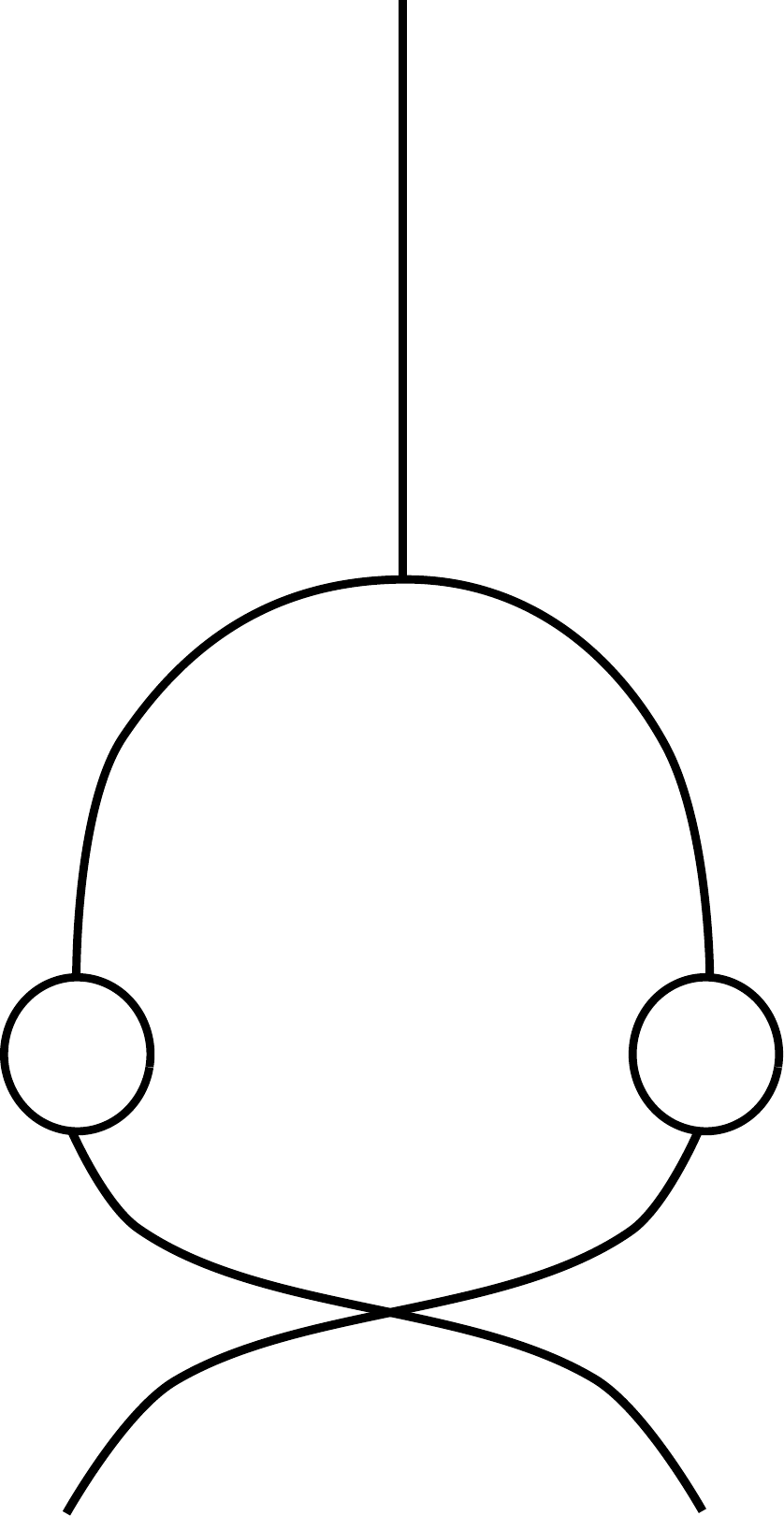}{10ex} \put(-29,-29){\ms{\alpha}} \put(-2,-29){\ms{\alpha}} \put(-15,31){\ms{\alpha^{-1}}} \ , \quad 
\epsh{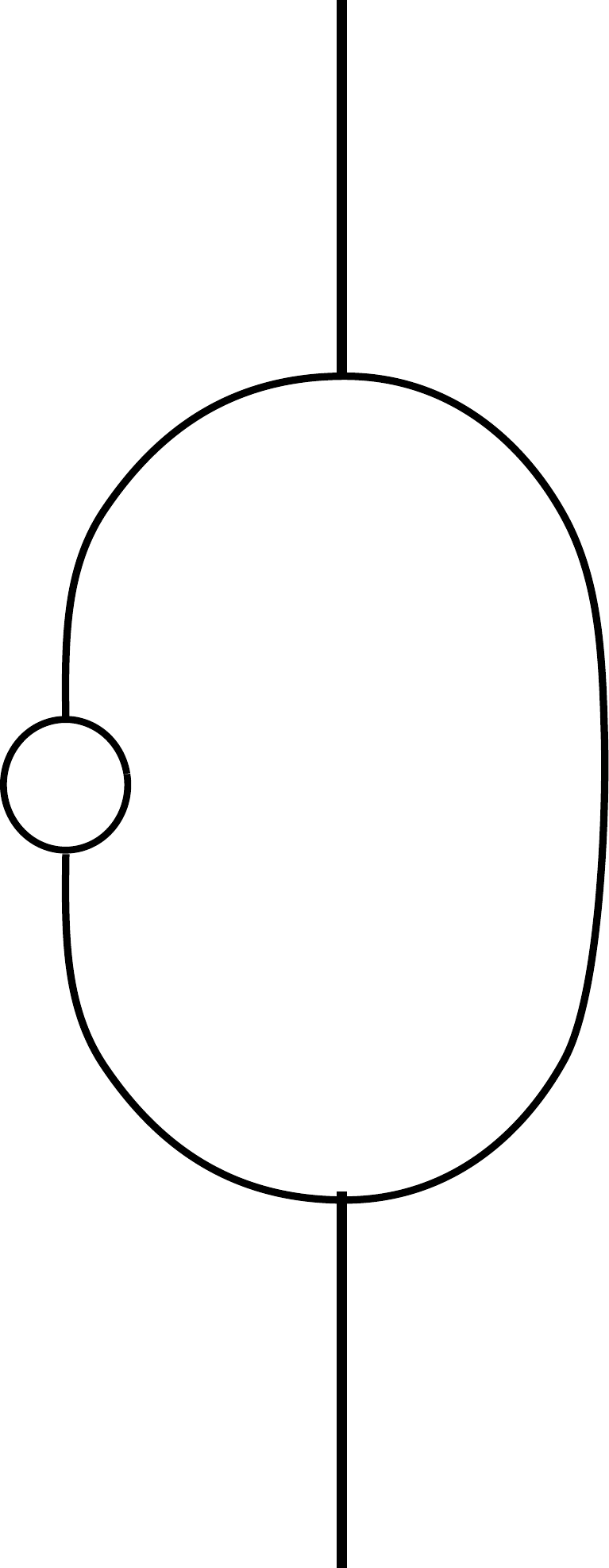}{10ex}\put(-12,31){\ms{\alpha}} \put(-11,-31){\ms{1}}\
 =\epsh{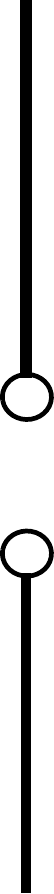}{10ex}\put(-5,31){\ms{\alpha}} \put(-4,-31){\ms{1}}\ 
 =\epsh{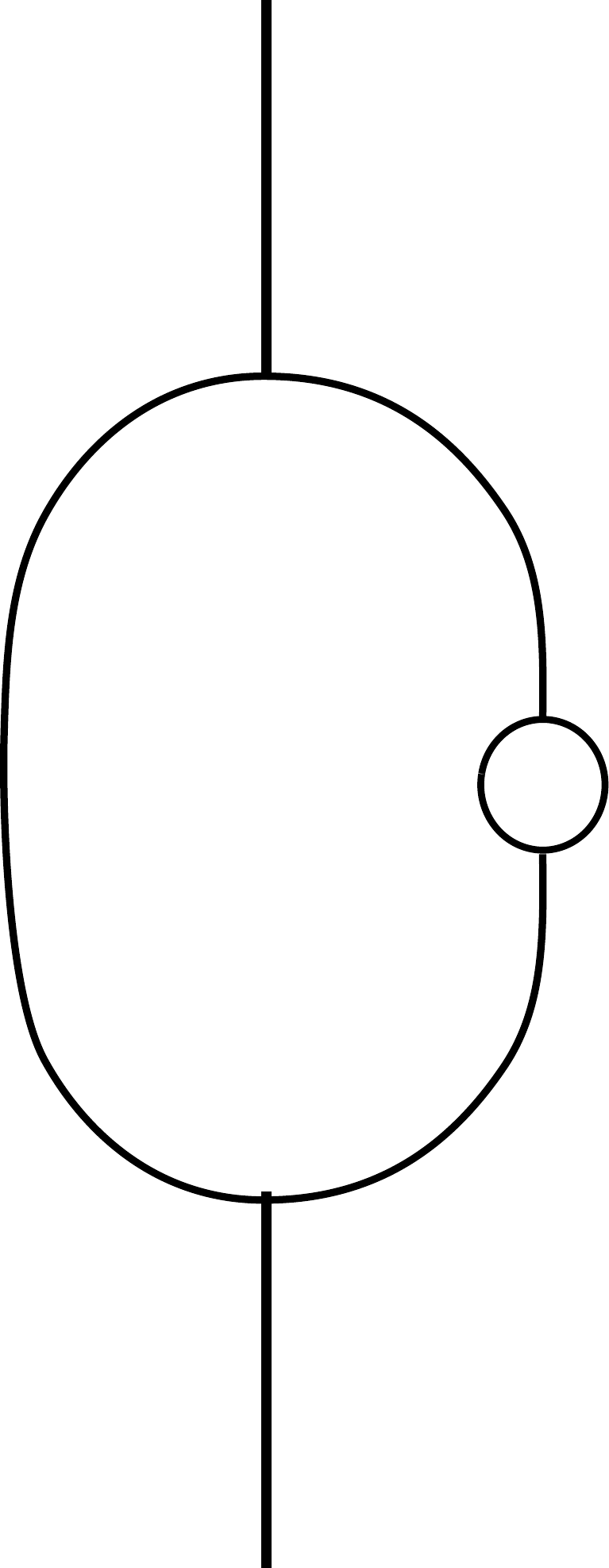}{10ex} \put(-14,31){\ms{\alpha}} \put(-13,-31){\ms{1}}\ , \quad
\epsh{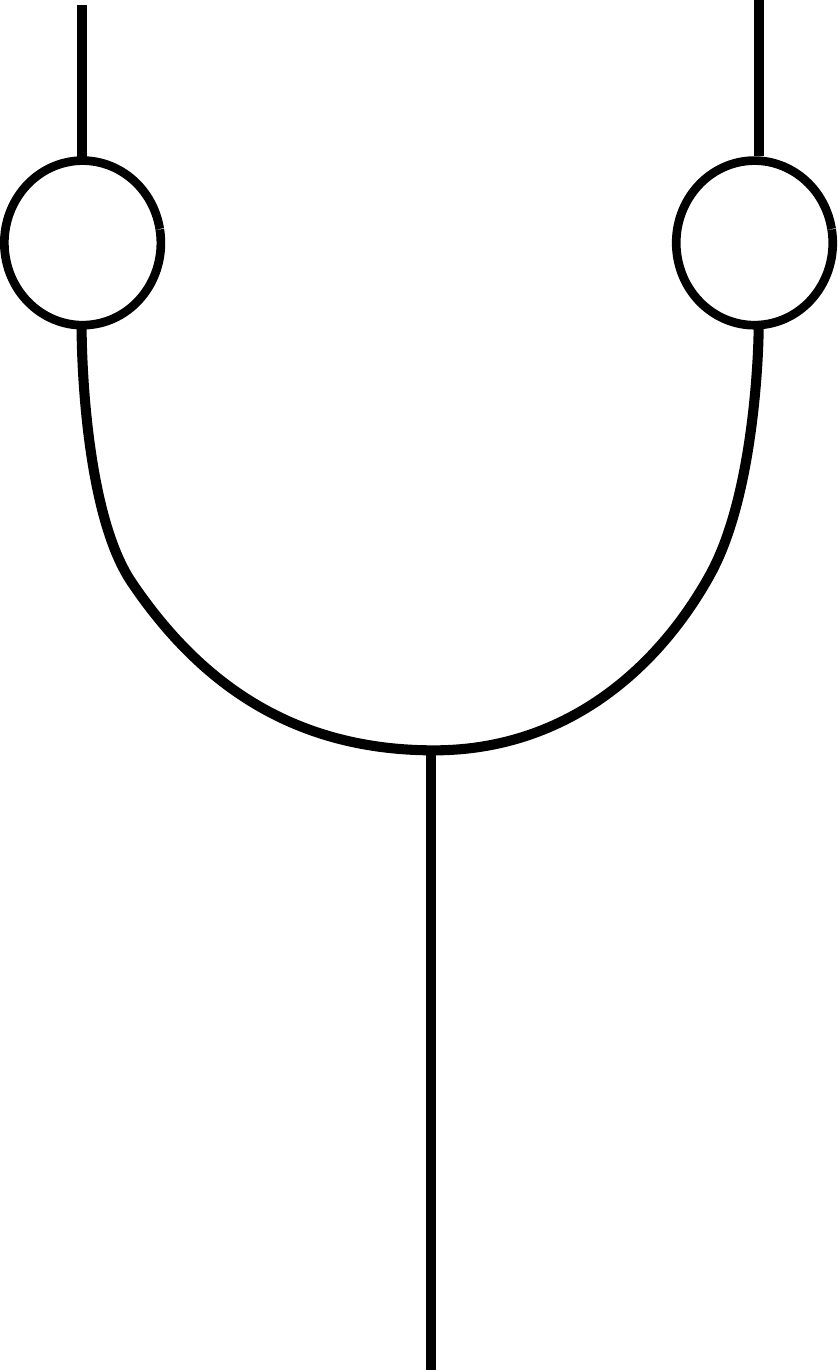}{10ex} \put(-31,31){\ms{\alpha^{-1}}} \put(-6,31){\ms{\beta^{-1}}} 
			\put(-20,-31){\ms{\alpha\beta}}
\ =\ \epsh{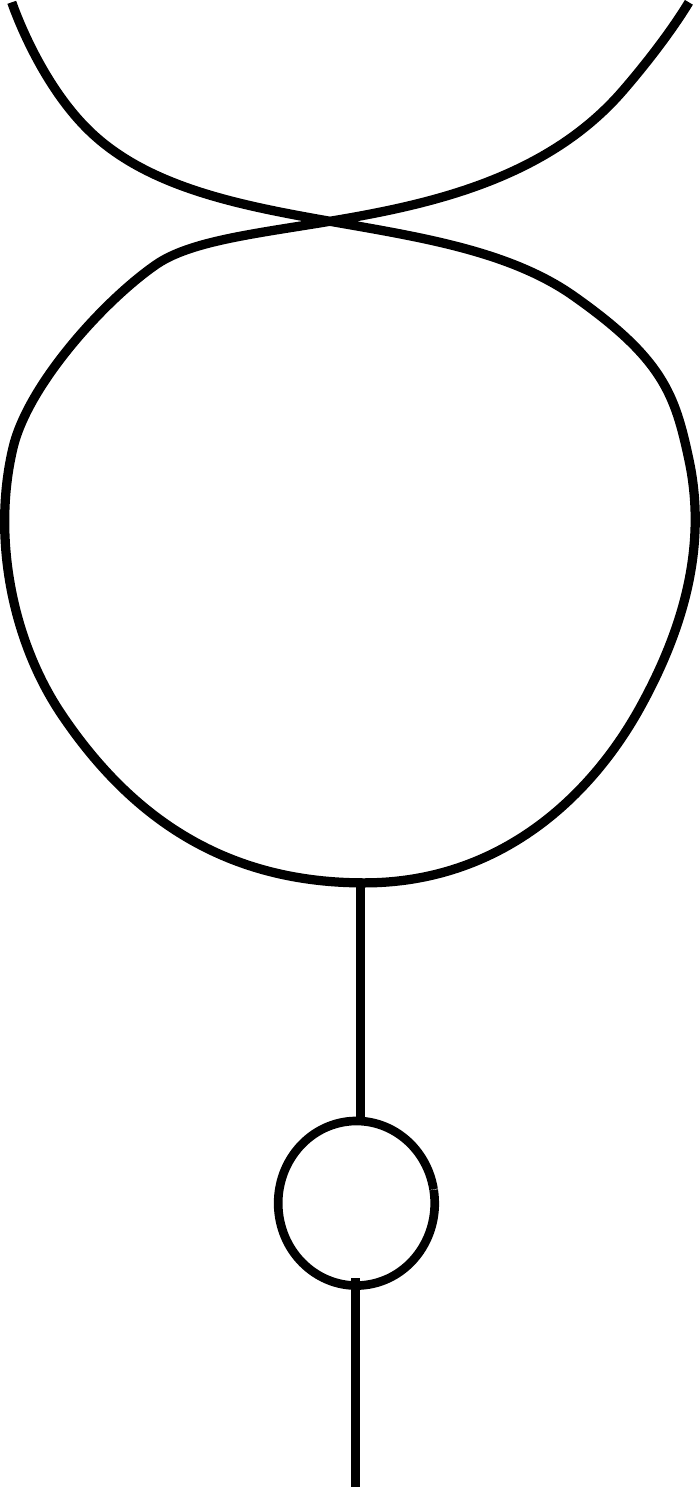}{10ex} \put(-18,-31){\ms{\alpha\beta}} \put(-26,31){\ms{\alpha^{-1}}} \put(-1,31){\ms{\beta^{-1}}}\ .$$
\caption{The antipode properties}\label{Fig antipode axioms}
\end{figure}
\begin{figure}
$$\epsh{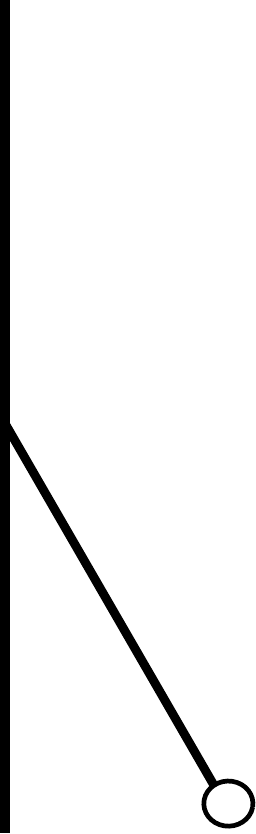}{8ex}\put(-14,27){\ms{\alpha}} \put(-14,-25){\ms{\alpha}} \put(-1,-12){\ms{\alpha}}\quad=\quad\epsh{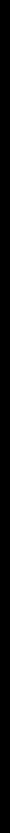}{8ex}\put(-3,27){\ms{\alpha}} \put(-3,-25){\ms{\alpha}}\quad=\quad\epsh{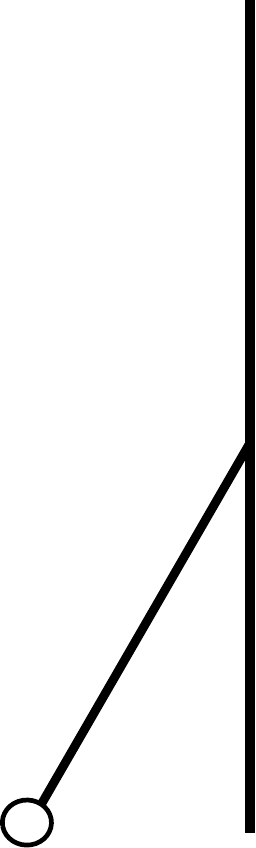}{8ex}\put(-3,27){\ms{\alpha}} \put(-3,-25){\ms{\alpha}} \put(-16,-12){\ms{\alpha}}\quad , \quad \epsh{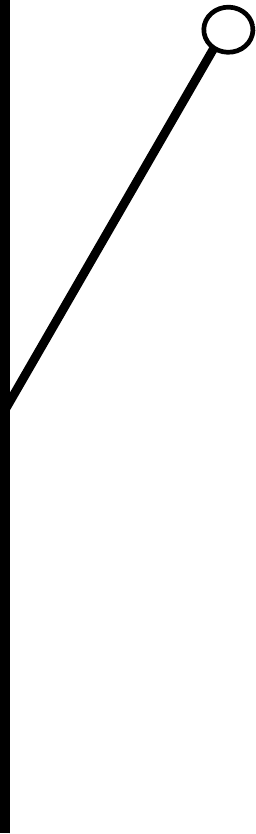}{8ex}\put(-14,27){\ms{\alpha}} \put(-14,-25){\ms{\alpha}} \quad=\quad\epsh{fig412}{8ex}\put(-3,27){\ms{\alpha}} \put(-3,-25){\ms{\alpha}}\quad=\quad\epsh{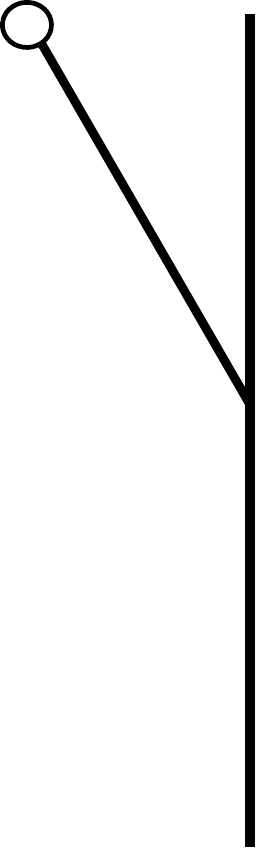}{8ex}\put(-3,27){\ms{\alpha}} \put(-3,-25){\ms{\alpha}}\ ,$$
$$\epsh{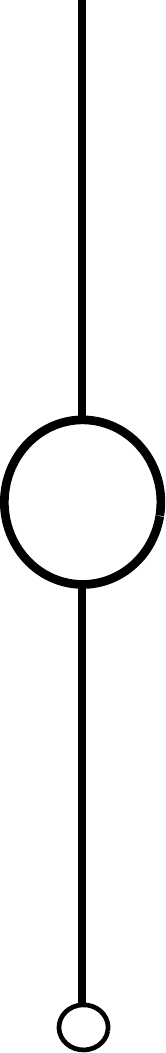}{8ex}\ =\ \epsh{fig43}{8ex} \put(-29,-12){\ms{\alpha^{-1}}} \put(-33,26){\ms{\alpha}} \put(-4,26){\ms{\alpha}}\quad , \quad \epsh{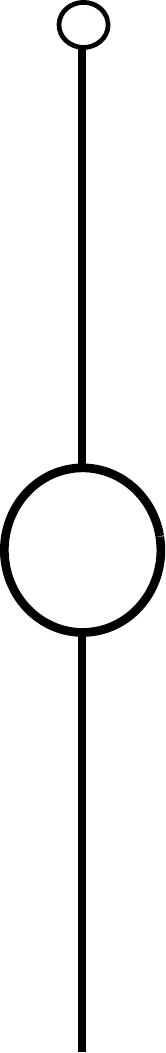}{8ex} \put(-5, -26){\ms{1}}\ =\  \epsh{fig44}{8ex} \put(-3, -26){\ms{1}}\ .$$
\caption{The unit and counity}\label{Fig counity axioms}
\end{figure}
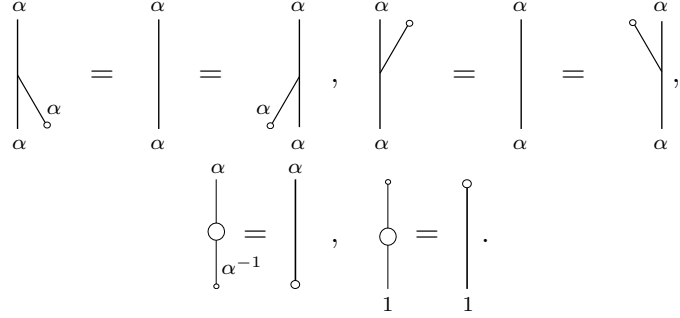
Finally, the compatibility between the antipode and the unity, counity are illustrated in Figure \ref{Fig counity axioms}.


\begin{Exem}\label{example 1} Let $H$ be a possibly infinite dimensional pivotal Hopf algebra with the pivot $\phi$. Suppose there is a commutative Hopf subalgebra $C$ contained in the center of $H$ (for example $H$ can be the unrestrited quantum group in \cite{KaDe92}; an other example will be detailed in Section 6). Let $G=\Hom_{Alg}(C, \Bbbk)$ be the group of characters on $C$ with multiplication given by $gh= (g\otimes h)\circ \Delta$ for $g, h \in G$ and $g^{-1}=g\circ S|_{C}$. For $g\in G$ we define $H_{g}=H\otimes_{g: C\rightarrow \Bbbk}\Bbbk = H/I_{g}$ where $I_{g}$ is the ideal generated by elements $z-g(z)$ for $z\in C$. Assume  $g=g_1 g_2$ for $g_1, g_2 \in G$, then
\begin{align*}
\Delta(z-g(z))&=\Delta(z)-(g_1\otimes g_2)(\Delta(z))\\
&=z_{(1)}\otimes z_{(2)}-g_1(z_{(1)})\otimes g_2(z_{(2)})\\
&=\left(z_{(1)}-g_1(z_{(1)})\right) \otimes z_{(2)} + g_1(z_{(1)})\otimes \left(z_{(2)}-g_2(z_{(2)})\right)
\end{align*}
where we used the Sweedler's notation $\Delta(z)=z_{(1)}\otimes z_{(2)}$.
This implies that $\Delta(I_{g})\subset I_{g_1}\otimes H+H\otimes I_{g_2}$. We thus have that a well defined coproduct $\Delta_{g_1, g_2}$ given by the commutative diagram below
\begin{diagram}
H &\rTo^{\Delta} &H \otimes H\\
\dTo_{p_{g_1 g_2}} & &\dTo_{p_{g_1}\otimes p_{g_2}}\\
H_{g_1 g_2} &\rTo^{\Delta_{g_1, g_2}} &H_{g_1} \otimes H_{g_2}
\end{diagram}
where $p_g:\ H \rightarrow H_g$ is the projective morphism. The family $\{H_g\}_{g\in G}$ with coproduct $\Delta_{g,h}$ is a $G$-coalgebra. It is also a Hopf $G$-coalgebra with the family of antipode given by the commutative diagram
\begin{diagram}
H &\rTo^{S} &H \\
\dTo_{p_{g}} & &\dTo_{p_{g^{-1}}}\\
H_{g} &\rTo^{S_{g}} &H_{g^{-1}} .
\end{diagram}
The family $S_g$ for $g\in G$ is well defined since $S(z-g(z))=S(z)-g(z)=S(z)-g^{-1}(S(z)) \in I_{g^{-1}}$. 
\end{Exem}
We say a Hopf $G$-coalgebra $H$ is of {\em finite type} if $H_{\alpha}$ is finite dimensional over $\Bbbk$ for all $\alpha\in G$. 
\subsubsection{Pivotal structure}
We recall that a $G$-grouplike element of a Hopf $G$-coalgebra $H$ is a family $g=(g_{\alpha})_{\alpha \in G}\in \prod_{\alpha}H_{\alpha}$ such that $\Delta_{\alpha, \beta}(g_{\alpha \beta})=g_{\alpha} \otimes g_{\beta}$ for any $\alpha, \beta \in G$ and $\varepsilon(g_1)=1_{\Bbbk}$. Note that $g_1$ is a grouplike element of the Hopf algebra $H_1$. It follows \cite{Vire02} that the set of the $G$-grouplike elements of $H$ is a group and if $g=(g_{\alpha})_{\alpha \in G}$, then $g^{-1}=(S_{\alpha^{-1}}(g_{\alpha^{-1}}))_{\alpha \in G}$.
\begin{Def}
A $G$-grouplike element $g\in H$ is called a pivot if 
$S_{\alpha^{-1}}S_{\alpha}(x)=g_{\alpha}xg_{\alpha}^{-1}$ for all $x \in H_{\alpha}$.
The pair $(H, g)$ of a Hopf $G$-coalgebra $H$ and a pivot $g$ is called a pivotal Hopf $G$-coalgebra.
\end{Def}
Remark that for a pivotal Hopf $G$-coalgebra $H=(\{H_{\alpha}\}_{\alpha \in G}, \Delta, \varepsilon, S, g)$, $H_1$ is a pivotal Hopf algebra.
\begin{Exem}
Let $H$ be a Hopf $G$-coalgebra as in Example \ref{example 1}. Let $\phi_{g}$ be the image of $\phi$ in the quotient $H_g$. Then $H$ is a pivotal Hopf $G$-coalgebra.
\end{Exem}
\subsection{Symmetrised right and left $G$-integrals}
Let $H=(\{H_{\alpha}\}_{\alpha \in G}, \Delta, \varepsilon, S)$ be a finite type pivotal Hopf $G$-coalgebra with right $G$-integral $\mu$. The symmetrised right $G$-integral associated with $\mu$ is a family $\widetilde{\mu}=(\widetilde{\mu}_{\alpha})_{\alpha\in G}\in \prod_{\alpha\in G}H_{\alpha}^*$ defined by
$\widetilde{\mu}_{\alpha}(x):=\mu_{\alpha}(g_{\alpha}x)$ for any $x\in H_{\alpha}$.\\
Using the definition of the right $G$-integral, see Equation \eqref{right int} and replacing $x\in H_{\alpha\beta}$ by $g_{\alpha\beta}x $ we get:
\begin{equation}\label{s right int}
(\widetilde{\mu}_{\alpha} \otimes g_{\beta})\Delta_{\alpha, \beta}(x)= \widetilde{\mu}_{\alpha\beta}(x)1_{\beta}.
\end{equation}
Similarly, the symmetrised left $G$-integral is defined by $\widetilde{\mu}_{\alpha}^{l}(x):=\mu_{\alpha}^{l}(g_{\alpha}^{-1}x)$ for any $x \in H_{\alpha}$. Applying (\ref{left int}) for $g_{\alpha\beta}^{-1}x, \ x\in H_{\alpha\beta}$ we get the defining relation for the symmetrised left $G$-integral:
\begin{equation}\label{s left int}
(g_{\alpha}^{-1} \otimes \widetilde{\mu}_{\beta}^{l})\Delta_{\alpha, \beta}(x)= \widetilde{\mu}_{\alpha\beta}^{l}(x)1_{\alpha}.
\end{equation}
The graphical representation for Equality \eqref{s right int} is given in Figure \ref{presentation of right sym G-int}.
\begin{figure}
$$\epsh{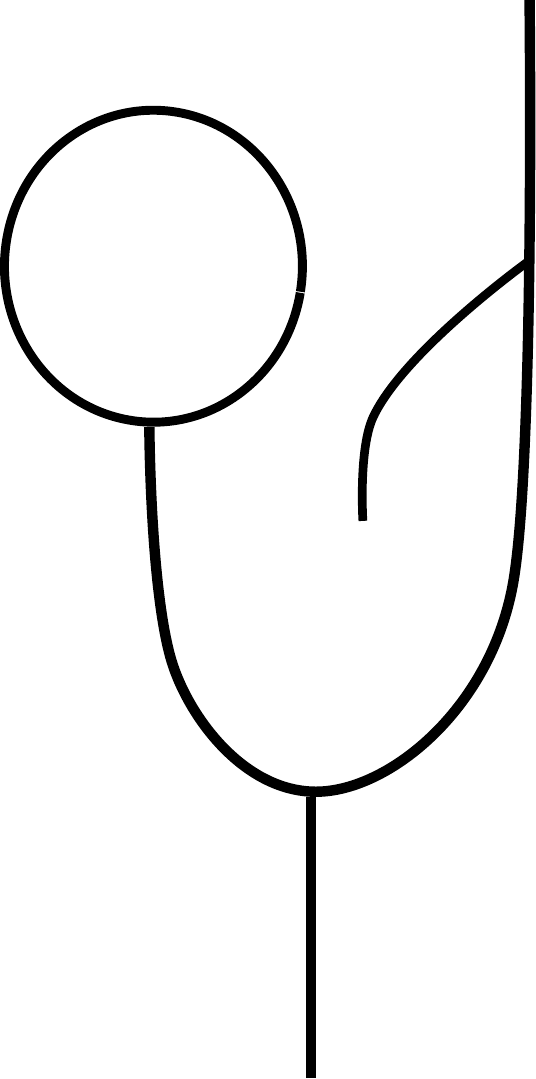}{14ex}\put(-31,20){\ms{\widetilde{\mu}_{\alpha}}} \put(-21,-42){\ms{\alpha\beta}} \put(-16,-2){\ms{g_{\beta}}} \put(-2,42){\ms{\beta}}\ 
=\ \epsh{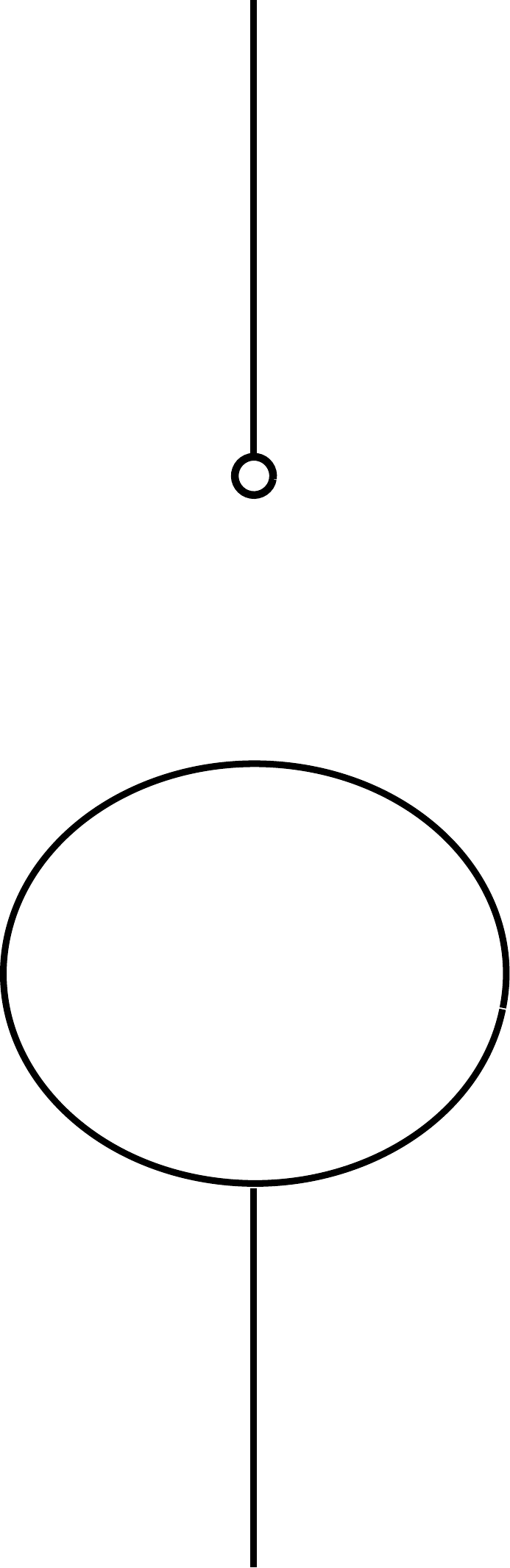}{14ex}\put(-19, -6){\ms{\widetilde{\mu}_{\alpha\beta}}} \put(-17,-42){\ms{\alpha\beta}} \put(-15,42){\ms{\beta}}\ .$$
\caption{The graphical representation of the relation for the right symmetrised $G$-integral}\label{presentation of right sym G-int}
\end{figure} 
The graphical representation of the relation for the left symmetrised $G$-integral is similar.

Since the pivot is invertible Equation \eqref{s right int} for $\widetilde{\mu}$ is equivalent to Equation \eqref{right int} for $\mu$. As the space of right $G$-integrals is one-dimensional, relation \eqref{s right int} defines $\widetilde{\mu}$ uniquely (up to a scalar). Similarly the symmetrised left $G$-integral $\widetilde{\mu}^{l}$ defined by \eqref{s left int} is unique. Note also that the symmetrised $G$-integral for $H_1$ is the one in the sense of \cite{BBG17}.

A Hopf $G$-coalgebra $H$ is {\em unimodular} if the Hopf algebra $H_1$ is unimodular, this means that the spaces of left and right cointegrals for $H_1$ coincide.

We say that a family of linear forms $\varphi_{\alpha}\in H_{\alpha}^{*}$ for $\alpha\in G$ is {\em symmetric} non-degenerate if for any $\alpha\in G$ the associated bilinear forms $(x,y) \mapsto \varphi_{\alpha}(xy),\ x, y \in H_{\alpha}$ is. 
\begin{Pro}\label{sr int sym}
Assure $(H, g)$ is unimodular, then the symmetrised right (resp. left) $G$-integral for $(H, g)$ is symmetric and non-degenerate.
\end{Pro} 
\begin{proof}
For any $\alpha \in G,\ x, y \in H_{\alpha}$, by \cite[Lemma 7.1]{Vire02} we have $$\widetilde{\mu}_{\alpha}(xy)=\mu_{\alpha}(g_{\alpha}xy)=\mu_{\alpha}(S_{\alpha^{-1}}S_{\alpha}(y)g_{\alpha}x)=\mu_{\alpha}(g_{\alpha}yx)=\widetilde{\mu}_{\alpha}(yx)$$
and by \cite[Corollary 3.7]{Vire02} $H_{\alpha}^*$ is free left module rank one over $H_\alpha$ with basis $\{\mu_{\alpha}\}$ when the action is defined by
\begin{equation*}
(h\rightharpoonup \mu_{\alpha})(x):=\mu_{\alpha}(xh) \ \text{for} \ h, x \in H_{\alpha}.
\end{equation*}
If $\widetilde{\mu}_{\alpha}(xy)=\mu_{\alpha}(g_{\alpha}yx)=(g_{\alpha}y\rightharpoonup \mu_{\alpha})(x)=0$ for all $x\in H_{\alpha}$, then $g_{\alpha}y\rightharpoonup \mu_{\alpha}=0$. It follows thus $y=0$.\\
For the symmetrised left $G$-integral the proof is similar.
\end{proof}
 
Also note that the spaces of left and right $G$-integrals are not equal in general. We have a lemma.
\begin{Lem}\label{choice of left int}
The left $G$-integral for $H$ can be chosen as $\mu_{\alpha}^{l}(x)=\mu_{\alpha^{-1}}(S_{\alpha}(x))$ for any $x\in H_{\alpha}$.
\end{Lem}
\begin{proof}
By (\ref{right int}) we have 
\begin{equation*}
(\mu_{\alpha^{-1}}\otimes \Id_{H_{\beta^{-1}}})\Delta_{\alpha^{-1}, \beta^{-1}}(S_{\beta\alpha}(x))=\mu_{{(\beta\alpha})^{-1}}(S_{\beta\alpha}(x))1_{ \beta^{-1}}\ \text{for any}\ x\in H_{\beta\alpha}.
\end{equation*}
Using Lemma \ref{tc antip} (3) $\Delta_{\alpha^{-1}, \beta^{-1}}(S_{\beta\alpha}(x))=(S_{\alpha}\otimes S_{\beta})\Delta_{\beta, \alpha}^{op}(x)$ we get
$$(\mu_{\alpha^{-1}}\circ S_{\alpha}\otimes S_{\beta})\Delta_{\beta, \alpha}^{op}(x)=(S_{\beta} \otimes \mu_{\alpha^{-1}}\circ S_{\alpha})\Delta_{\beta, \alpha}(x)=\mu_{{(\beta\alpha})^{-1}}(S_{\beta\alpha}(x))1_{ \beta^{-1}}.$$
Applying $S_{\beta}^{-1}$ to both sides of the last equality and $S_{\beta}^{-1}(1_{ \beta^{-1}})=1_{ \beta}$, we obtain that $(\Id_{H_{\beta}}\otimes \mu_{\alpha^{-1}}\circ S_{\alpha})\Delta_{\beta, \alpha}(x)=(\mu_{{(\beta\alpha})^{-1}} \circ S_{\beta\alpha})(x)1_{\beta}$, i.e. $\mu_{{\alpha}^{-1}} \circ S_{\alpha}$ satisfies the definition of the left $G$-integral.
\end{proof}
\subsection{$G$-unibalanced Hopf algebras}
Let $H=(\{H_{\alpha}\}_{\alpha \in G}, \Delta, \varepsilon, S)$ be a finite type Hopf $G$-coalgebra with right $G$-integral $\mu$. We call a {\em distinguished $G$-grouplike} of $H$ (see e.g \cite{Vire02}) or {\em $G$-comodulus} of $H$ a $G$-grouplike element $\textbf{a}=(a_{\alpha})_{\alpha\in G}\in \prod_{\alpha\in G}H_{\alpha}$ satisfying 
\begin{equation}\label{G-comodulus}
(\Id_{H_{\alpha}}\otimes \mu_{\beta})\Delta_{\alpha, \beta}(x)=\mu_{\alpha\beta}(x)a_{\alpha} \ \text{for any}\ x \in H_{\alpha\beta}.
\end{equation} 
Note that $a_1$ is the comodulus element of the Hopf algebra  $H_1$ (see \cite{BBG17}). By multiplying \eqref{G-comodulus} with $\textbf{a}^{-1}$ and replacing $x$ by $a_{\alpha\beta}x$ we have
\begin{equation}\label{tc comodulus}
(\Id_{H_{\alpha}}\otimes \mu_{\beta}(a_{\beta}?))\Delta_{\alpha, \beta}(x)=\mu_{\alpha\beta}(a_{\alpha\beta}x)1_{\alpha}
\end{equation}
where denote by $\mu_{\beta}(a_{\beta}?)$ the linear map $x\mapsto \mu_{\beta}(a_{\beta}x)$ for $x\in H_{\beta}$.
This equality implies that $\mu_{\beta}(a_{\beta}?)$ is a left $G$-integral for $H$, i.e.
\begin{equation}\label{another choice left int}
\mu_{\beta}^{l}(x)=\mu_{\beta}(a_{\beta}x).
\end{equation}
This is another choice for left $G$-integral from right $G$-integral. This choice of the left $G$-integral is the same with the one in Lemma \ref{choice of left int} by following proposition.
\begin{Pro}
We have the relation $\mu_{\alpha^{-1}}(S_{\alpha}(x))=\mu_{\alpha}(a_{\alpha}x)$ for any $x\in H_{\alpha}$.
\end{Pro}
\begin{proof}
By \eqref{tc comodulus} we get 
\begin{equation*}
\left(\Id_{H_{\alpha}}\otimes \mu_1(a_1 ?)\right)\Delta_{\alpha,1}(x)=\mu_{\alpha}(a_\alpha x)1_{\alpha}\ \text{for}\ x\in H_{\alpha}.
\end{equation*}
By Lemma \ref{choice of left int} we get
\begin{equation*}
\left(\Id_{H_{\alpha}}\otimes \mu_1\circ S_1\right)\Delta_{\alpha,1}(x)=(\mu_{\alpha^{-1}}\circ S_{\alpha})(x)1_\alpha \ \text{for}\ x\in H_{\alpha}.
\end{equation*}
Furthermore, Proposition 4.7 \cite{BBG17} gives $\mu_{1}(S_{1}(x))=\mu_{1}(a_{1}x)$ for $x\in H_1$. This implies that $\mu_{\alpha}(a_\alpha x)1_{\alpha}=(\mu_{\alpha^{-1}}\circ S_{\alpha})(x)1_\alpha$ for all $x\in H_{\alpha}$, i.e. $\mu_{\alpha^{-1}}(S_{\alpha}(x))=\mu_{\alpha}(a_{\alpha}x)$ for any $x\in H_{\alpha}$.
\end{proof}
A finite type pivotal Hopf $G$-coalgebra $(H, g)$ is called {\em $G$-unibalanced} if its symmetrised right $G$-integral is also left, i.e. $\widetilde{\mu}_{\alpha}=\widetilde{\mu}_{\alpha}^{l}$ for any $\alpha\in G$. 
\begin{Lem}
Assume $(H, g)$ is an unimodular pivotal Hopf $G$-coalgebra. Then $(H, g)$ is
$G$-unibalanced if and only if $a_{\alpha}=g_{\alpha}^{2}$ for any $\alpha\in G$.
\end{Lem}
\begin{proof}
First, we assume that $a_{\alpha}=g_{\alpha}^{2}$. Applying \eqref{G-comodulus} on $g_{\alpha\beta}x$ we have
\begin{equation*}
(g_{\alpha}^{-1}\otimes \widetilde{\mu}_{\beta})\Delta_{\alpha, \beta}(x)=\widetilde{\mu}_{\alpha\beta}(x)1_{\alpha}.
\end{equation*}
This equality states that $\widetilde{\mu}_{\beta}$ is a symmetrised left $G$-integral, i.e. $\widetilde{\mu}_{\beta}=\widetilde{\mu}_{\beta}^{l}$.
Second, we assume that $(H, g)$ is $G$-unibalanced. By applying the equality \eqref{another choice left int} on $g_{\alpha}^{-1}x$ and the $G$-unibalanced condition one gets
\begin{equation*}
\mu_{\alpha}^{l}(g_{\alpha}^{-1}x)=\widetilde{\mu}_{\alpha}^{l}(x)=\widetilde{\mu}_{\alpha}(x)=\mu_{\alpha}(g_{\alpha}x)=\mu_{\alpha}(a_{\alpha}g_{\alpha}^{-1}x)
\end{equation*}
for any $x\in H_{\alpha}$. The last equality gives 
\begin{equation*}
\mu_{\alpha}\left((a_{\alpha}g_{\alpha}^{-1}-g_{\alpha})x\right)=0 \ \text{for any}\ x\in H_{\alpha}.
\end{equation*}
By Proposition \ref{sr int sym}, $\mu_{\alpha}$ is non-degenerate. Therefore, the above equality holds if and only if $a_{\alpha}=g_{\alpha}^{2}$.
\end{proof}

\section{Traces on finite $G$-graded categories}\label{section trace}
In this section we recall some notions and results from \cite{BBG17}. Let $(H,g)$ be a finite type unimodular pivotal Hopf $G$-coalgebra. We determine the pivotal structure in pivotal $G$-graded category $H$-mod. We also prove the Reduction Lemma in the context of $G$-graded categories and recall the close relation between a modified trace on $H_1$-pmod and a symmetrised integral for $H_1$ \cite{BBG17}.
\subsection{Cyclic traces}
Let $\mathcal{C}$ be a $\Bbbk$-linear category.
We call {\em cyclic trace} on $\mathcal{C}$ a family of $\Bbbk$-linear maps 
\begin{equation} \label{tc trace}
t=\{t_{P}:\ \End_{\mathcal{C}}(P) \rightarrow \Bbbk\}_{P \in \mathcal{C}}
\end{equation}
satisfying cyclicity property, i.e. $t_{V}(gh)=t_{U}(hg)$ for  $g\in \Hom_{\mathcal{C}}(U, V)$ and $h\in \Hom_{\mathcal{C}}(V, U)$ with $U, V\in \mathcal{C}$.
We say that a cyclic trace $t$ is {\em non-degenerate} if the pairings
\begin{equation}\label{trace non-dege}
\Hom_{\mathcal{C}}(M,P) \times \Hom_{\mathcal{C}}(P,M)\rightarrow \Bbbk, \ (f,g)\mapsto t_{P}(f g)
\end{equation}
are non-degenerate for all $P, M \in \mathcal{C}$.

For a finite dimensional algebra $A$, let $A$-pmod be the category of projective $A$-modules.
There is a bijection from the space of cyclic traces on $A$-pmod to the space of symmetric linear forms on $A$:
\begin{Lem}\label{ds-pt}
There is an isomorphism of algebras 
$$R:\ A^{op} \rightarrow \End_{A}(A) $$ 
given by 
$$R(h)=R_h, \qquad R^{-1}(f)=f(1) $$
where $R_h$ denotes the right multiplication with $h$, i.e. $R_h(x)=xh$ for any $ x \in A$.
\end{Lem}
Lemma \ref{ds-pt} implies that if $t$ is a cyclic trace on $A$-pmod then 
\begin{equation}\label{Eq xd linear form}
\lambda(h)=t_{A}(R_h)
\end{equation}
defines a symmetric linear form on $A$.
\begin{Pro}\cite[Proposition 2.4]{BBG17}\label{pro xd cyclic trace tu sym linear}
A symmetric linear form $\lambda$ on a finite dimensional algebra $A$ extends uniquely to a family of cyclic traces $\{t_{P}:\ \End_{A}(P) \rightarrow \Bbbk\}_{P\in A\text{-pmod}}$ which satisfies Equality \eqref{Eq xd linear form}.
\end{Pro}
If $f\in \End_{A}(P)$, one can find $a_i\in \Hom(A, P),\ b_i\in \Hom(P,A)\ i \in I$ for some finite set $I$ such that $f=\sum_{i\in I}a_i b_i$ (see \cite{BBG17}). Then the cyclicity property of $t$ implies that 
\begin{equation}\label{bt xd trace}
t_{P}(f)=\sum_{i\in I}t_A(b_i a_i)=\sum_{i\in I}\lambda\left(b_i a_i(1)\right).
\end{equation}

Furthermore, the non-degeneracy of the form linear $\lambda$ is equivalent to the one of the pairings \eqref{trace non-dege} determined by $\left(t_P\right)_{P\in A\text{-pmod}}$ in \eqref{bt xd trace} (see \cite{BBG17}, Theorem 2.6 where a stronger non-degeneracy condition for traces is considered).
\subsection{Modified trace in pivotal category}\label{Section modified trace}
Let $\mathcal{C}$ be a {\em pivotal} $\Bbbk$-linear category \cite{BePM12}. 
Then $\mathcal{C}$ is a strict monoidal $\Bbbk$-linear category, with a unit object $\mathbb{I}$, equipped with the data for each object $V\in \mathcal{C}$ of its dual object $V^{*}\in \mathcal{C}$ and of four morphisms
\begin{align*}
&\ev_{V}:\ V^{*} \otimes V \rightarrow \mathbb{I}, &\coev_{V}:\ 
\mathbb{I} \rightarrow V \otimes V^{*},\\
&\tev_{V}:\ V \otimes V^{*} \rightarrow \mathbb{I}, &\tcoev_{V}:\ \mathbb{I} \rightarrow V^{*} \otimes V
\end{align*} 
such that $(\ev_{V}, \coev_{V})$ and $(\tev_{V}, \tcoev_{V})$ are dualities which induce the same functor duality which is monoidal. In the category $\mathcal{C}$ there is a family of isomorphisms
\begin{equation*}
\Phi = \{\Phi_{V}=(\tev_{V} \otimes \Id_{V^{**}})  (\Id_{V} \otimes \coev_{V^{*}}):\ V \rightarrow V^{**}\}_{V \in \mathcal{C}}
\end{equation*}
which is a monoidal natural isomorphism called the pivotal structure.\\
We recall the notion of a {\em modified trace} on ideal in a pivotal category $\mathcal{C}$ which be introduced in \cite{NgJkBp13, GPV13}. Given $U, V, W \in \mathcal{C}$ and $f\in \End_{\mathcal{C}}(W \otimes V)$. 
The {\em left partial trace} (with respect to $W$) is the map 
\begin{equation*}
\tr_{W}^{l}:\ \Hom_{\mathcal{C}}(W\otimes U, W\otimes V) \rightarrow \Hom_{\mathcal{C}}(U, V)
\end{equation*}
defined for $f\in \Hom_{\mathcal{C}}(W\otimes U, W\otimes V)$ by
\begin{equation*}
\tr_{W}^{l}(f)=(\ev_{W}\otimes \Id_{V}) (\Id_{W^{*}}\otimes f) (\tcoev_{W}\otimes \Id_{V})=\epsh{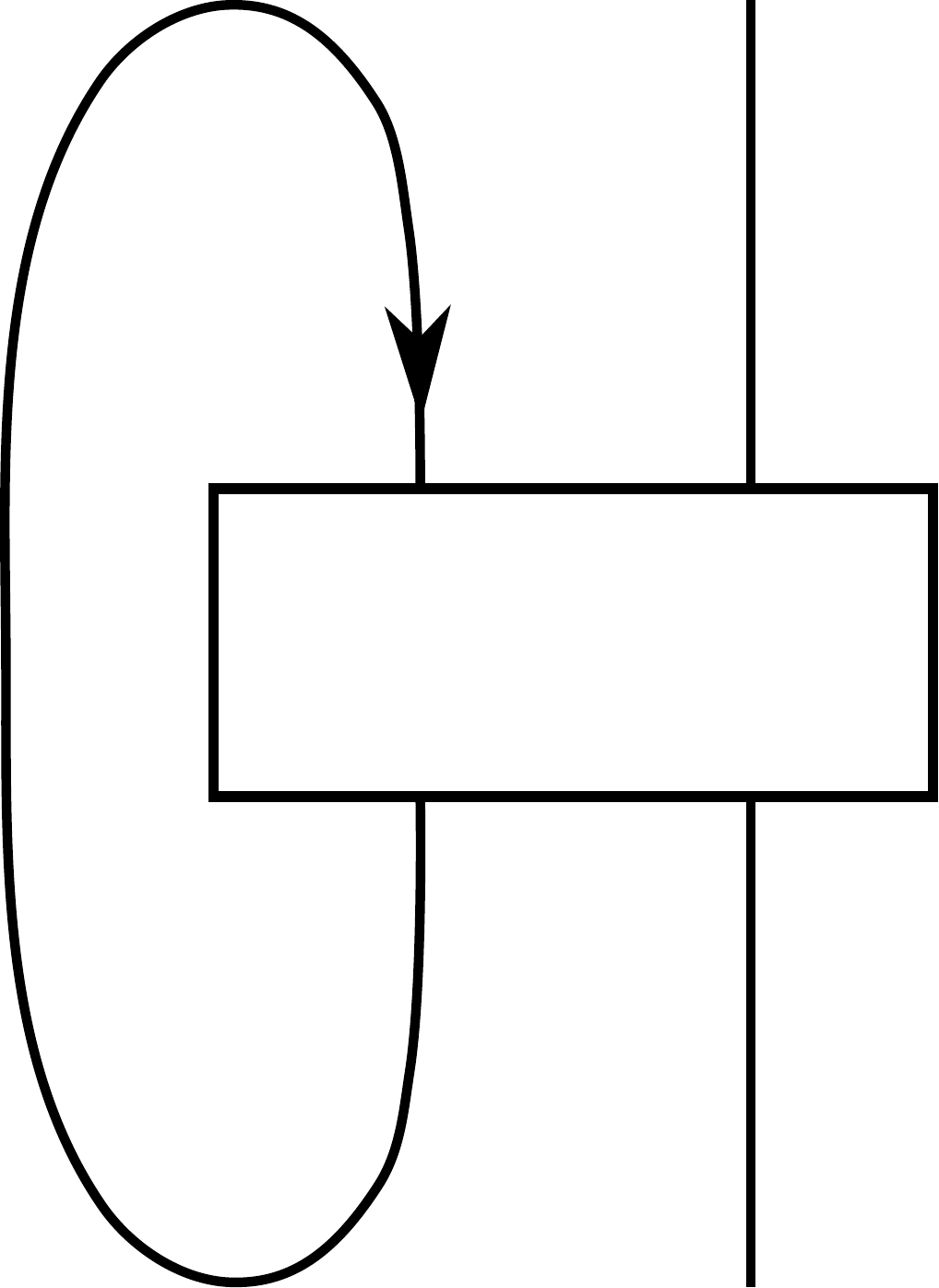}{10ex}\put(-16,1){\ms{f}} \put(-10,32){\ms{V}} \put(-42,25){\ms{W}} \put(-10,-32){\ms{U}} \in \Hom_{\mathcal{C}}(U, V).
\end{equation*}

The {\em right partial trace} (with respect to $W$) is the map 
\begin{equation*}
\tr_{W}^{r}:\ \Hom_{\mathcal{C}}(U\otimes W, V\otimes W) \rightarrow \Hom_{\mathcal{C}}(U, V)
\end{equation*}
defined for $f\in \Hom_{\mathcal{C}}(U\otimes W, V\otimes W)$ by
\begin{equation}\label{dn r trace}
\tr_{W}^{r}(f)=(\Id_{V}\otimes \tev_{W}) (f \otimes \Id_{W^{*}}) (\Id_{U}\otimes \coev_{W})= \epsh{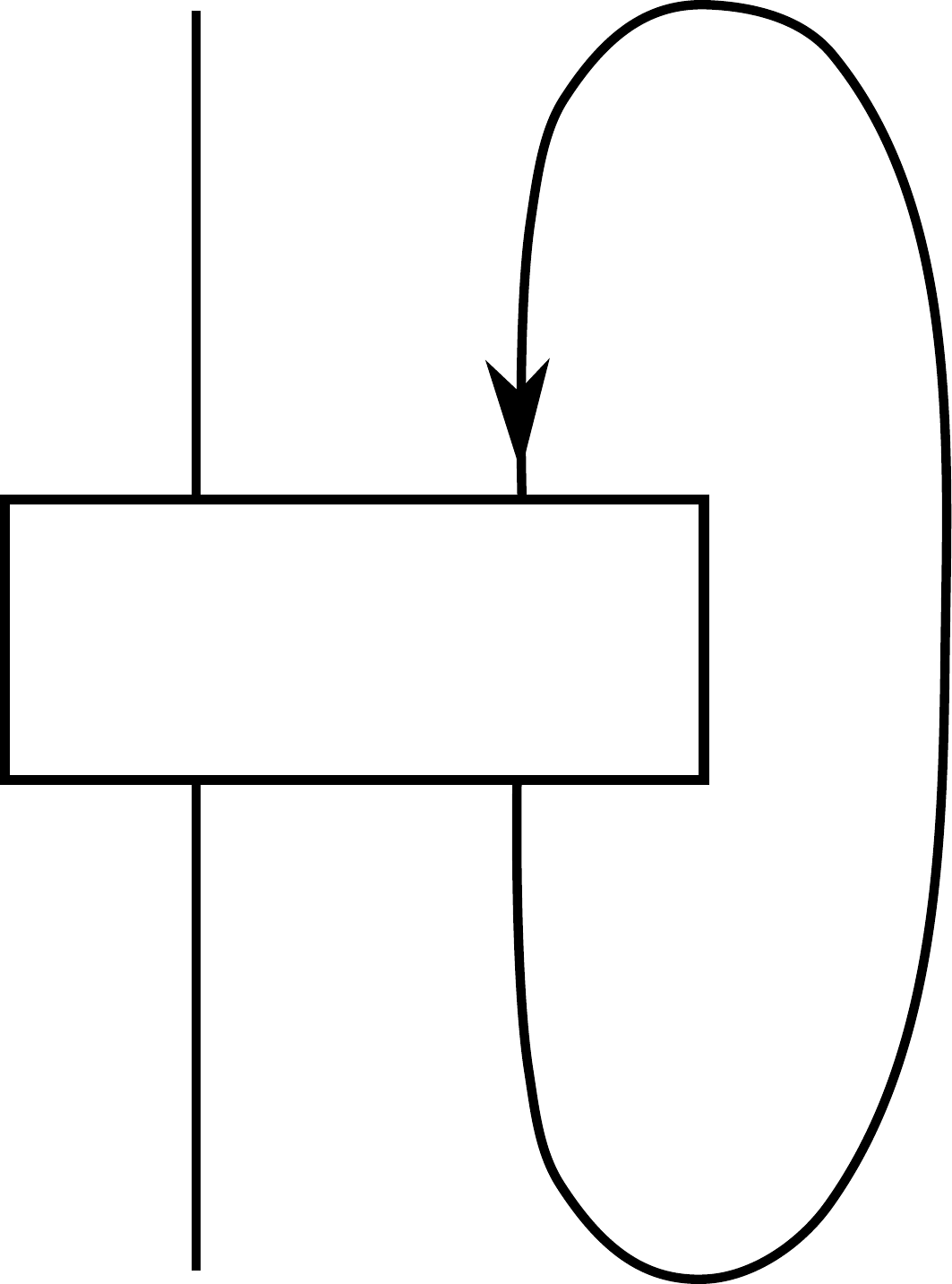}{10ex}\put(-24,1){\ms{f}} \put(-1,25){\ms{W}} \put(-33,32){\ms{V}} \put(-33,-32){\ms{U}} \in \Hom_{\mathcal{C}}(U, V).
\end{equation}
Let $\textbf{Proj}(\mathcal{C})$ be the tensor ideal of projective objects of $\mathcal{C}$.
A {\em left modified trace} on $\textbf{Proj}(\mathcal{C})$ is a cyclic trace $\mtr$ on $\textbf{Proj}(\mathcal{C})$ satisfying 
\begin{equation*}
\mtr_{W\otimes P}(f)=\mtr_{P}(\tr_{W}^{l}(f))
\end{equation*}
for any $f\in \End_{\mathcal{C}}(W\otimes P)$ with $P\in \textbf{Proj}(\mathcal{C})$ and $W\in \mathcal{C}$.

A {\em right modified trace} on $\textbf{Proj}(\mathcal{C})$ is a cyclic trace $\mtr$ on $\textbf{Proj}(\mathcal{C})$ satisfying
\begin{equation*}
\mtr_{P\otimes W}(f)=\mtr_{P}(\tr_{W}^{r}(f))
\end{equation*}
for any $f\in \End_{\mathcal{C}}(P\otimes W)$ with $P\in \textbf{Proj}(\mathcal{C})$ and $W\in \mathcal{C}$.


A {\em modified trace} on ideal $\textbf{Proj}(\mathcal{C})$ is a cyclic trace $\mtr$ on $\textbf{Proj}(\mathcal{C})$
which is both a left and right trace on $\textbf{Proj}(\mathcal{C})$.



Next we define the category of $H$-mod which is a pivotal $G$-graded category.
\subsection{Pivotal structure on $H$-mod}
\subsubsection{$G$-graded category} Given a multiplicative group $G$, we call the category $\mathcal{C}$ pivotal $G$-graded $\Bbbk$-linear if there exists a family of full subcategories $(\mathcal{C}_{\alpha})_{\alpha\in G}$ of $\mathcal{C}$ such that
\begin{itemize}
\item $\mathbb{I}\in \mathcal{C}_1$.
\item $\forall (\alpha, \beta)\in G^2, \ \forall (V, W) \in \mathcal{C}_\alpha \times \mathcal{C}_\beta, \ \Hom_{\mathcal{C}}(V, W)\neq \{0\} \Rightarrow \alpha=\beta$.
\item $\forall V \in \mathcal{C}, \ \exists n\in \mathbb{N},\ \exists (\alpha_1, ..., \alpha_n)\in G^n, \ \exists  V_i \in \mathcal{C}_{\alpha_i}$ for $i=1, ..., n$ such that $V\simeq V_1\oplus ... \oplus V_n$.
\item $\forall (V, W) \in \mathcal{C}_\alpha \times \mathcal{C}_\beta,\ V\otimes W\in \mathcal{C}_{\alpha\beta}$.
\item $\forall \alpha\in G,\ \mathcal{C}_\alpha$ does not reduce to null object.
\end{itemize}
\subsubsection{Pivotal structure on $H$-mod}
Let $(H,g)=(\{H_{\alpha}\}_{\alpha \in G}, \Delta, \varepsilon, S, g)$ be a finite type pivotal Hopf $G$-coalgebra, let $\mathcal{C}$ be the $\Bbbk$-linear category $\bigoplus_{\alpha\in G}\mathcal{C}_{\alpha}$ in which $\mathcal{C}_{\alpha}$ is $H_{\alpha}$-mod the category of finite dimensional $H_{\alpha}$-modules. An object $V$ of $\mathcal{C}$ is a finite direct sum $V_{\alpha_1}\oplus ... \oplus V_{\alpha_n}$ where $V_{\alpha_i}\in \mathcal{C}_{\alpha_i}$.
Each object $V$ in $H_{\alpha}$-mod has a dual $V^{*}=\Hom_{\Bbbk}(V, \Bbbk)$ in $H_{\alpha^{-1}}$-mod with the $H_{\alpha^{-1}}$ action defined by $(hf)(x)=f(S_{\alpha^{-1}}(h)x)$ for $h\in H_{\alpha^{-1}},\ f \in V^{*}$ and $x \in V$.
The category $\mathcal{C}$ is a $G$-graded tensor category, i.e. for $V_\alpha \in \mathcal{C}_{\alpha},\  V_\beta \in \mathcal{C}_{\beta} \ V_\alpha \otimes V_\beta \in \mathcal{C}_{\alpha\beta}$ and for $\alpha \neq \beta \ \Hom_{\mathcal{C}}(V_\alpha, V_\beta)=0$.\\
Then $\mathcal{C}$ is a pivotal category with pivotal structure given by the left and right duality morphisms as follows. Assume that $\{v_j \ |\ j\in J\}$ is a basis of $V\in H_{\alpha}$-mod and $\{v^j \ |\ j\in J\}$ is the dual basis of $V^{*}$, then
\begin{align}
&\ev_V:\ V^* \otimes V \rightarrow \Bbbk,\qquad  f\otimes v \mapsto f(v),\label{left dual} \\
&\coev_V:\ \Bbbk \rightarrow V \otimes V^*,\qquad  1 \mapsto \sum_{j\in J}v_j\otimes v^j, \nonumber\\
&\tev_V:\ V\otimes V^* \rightarrow \Bbbk, \qquad v\otimes f \mapsto f(g_{\alpha}v),\label{right dual}\\
&\tcoev_V:\ \Bbbk \rightarrow V^*\otimes V, \qquad 1 \mapsto \sum_{i\in J}v^i \otimes g_{\alpha}^{-1}v_i. \nonumber
\end{align}
We call $H$-pmod or $\textbf{Proj}(\mathcal{C})$ the ideal of projective $H$-modules. As $\mathcal{C}=\bigoplus_{\alpha\in G}\mathcal{C}_{\alpha}$, the projective modules of $\mathcal{C}_{\alpha}$ are in $H\text{-pmod}\cap \mathcal{C}_{\alpha}=H_{\alpha}\text{-pmod}$.
\begin{Lem}\label{lemma partial trace of cyclic trace}
Let $(H, g)$ be a finite type pivotal Hopf $G$-coalgebra. 
Let $t$ be a cyclic trace on $H$-pmod. Let $V\in H\text{-pmod}$ and $_\varepsilon  W\in H_{1}\text{-mod}$ be endowed with the trivial action $\rho_{_\varepsilon W}=\varepsilon \Id_{_\varepsilon W}$. Then 
\begin{equation}\label{lemma partial trace of cyclic trace right}
\forall f\in \End_{H\text{-mod}}(V\otimes\ _\varepsilon W),\
t_{V\otimes\ _\varepsilon W}(f)=t_{V}(\tr_{_\varepsilon W}^{r}(f))
\end{equation}
and
\begin{equation}\label{lemma partial trace of cyclic trace left}
\forall f\in \End_{H\text{-mod}}(_\varepsilon W\otimes V ),\
t_{_\varepsilon W\otimes V}(f)=t_{V}(\tr_{_\varepsilon W}^{l}(f)).
\end{equation}
\end{Lem}
\begin{proof}
Consider a decomposition of $\Id_{_\varepsilon W}$
\begin{equation}\label{pt id W}
\Id_{_\varepsilon W}=\sum_{i\in I}e_i\varphi_i \ \text{where}\ \varphi_i:\ _\varepsilon W \rightarrow \Bbbk,\ e_i:\ \Bbbk \rightarrow\ _\varepsilon W,\ \varphi_i(e_j)=\delta_{ij}.
\end{equation}
By setting $\widetilde{e}_i=\Id_V\otimes e_i:\ V\rightarrow V\otimes\ _\varepsilon W$ and $\widetilde{\varphi}_i=\Id_V\otimes \varphi_i:\  V\otimes\ _\varepsilon W\rightarrow V$ one gets 
\begin{equation}\label{pt id VxW}
\Id_{V\otimes\ _\varepsilon W}=\sum_{i\in I}\widetilde{e}_i\widetilde{\varphi}_i .
\end{equation}
For $f\in \End_{H\text{-mod}}(V\otimes\ _\varepsilon W)$, on the one hand we have
\begin{equation*}
t_{V\otimes\ _\varepsilon W}(f)=\sum_{i\in I}t_{V\otimes\ _\varepsilon W}(f\widetilde{e}_i\widetilde{\varphi}_i)
=\sum_{i\in I}t_{V}(\widetilde{\varphi}_i f\widetilde{e}_i) =\sum_{i\in I}t_{V}(f_{ii})
\end{equation*}
where $f_{ii}=\widetilde{\varphi}_i f\widetilde{e}_i\in \End_{H\text{-mod}}(V)$. In the above calculations, we use Equation \eqref{pt id VxW} in the first equality and the cyclicity property in the second equality. \\
On the other hand, each map $f\in \End_{H\text{-mod}}(V\otimes\ _\varepsilon W)$ is presented by graph below
\begin{equation*}
\epsh{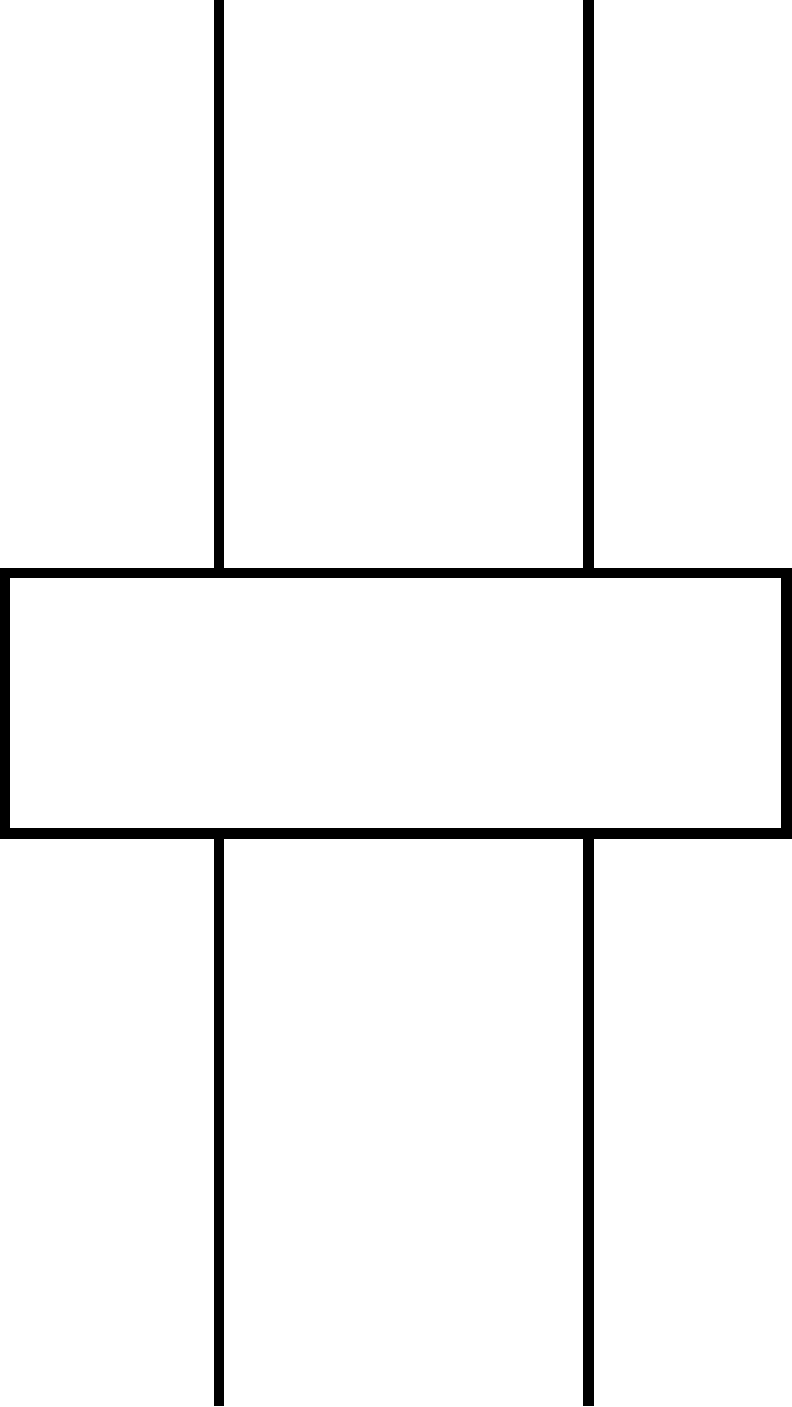}{14ex}\put(-22,1){\ms{f}}\put(-36,-42){\ms{V}} \put(-33,42){\ms{V}} \put(-13,42){\ms{_\varepsilon W}} \put(-13,-42){\ms{_\varepsilon W}}\
=\ \sum_{i, j\in I}\epsh{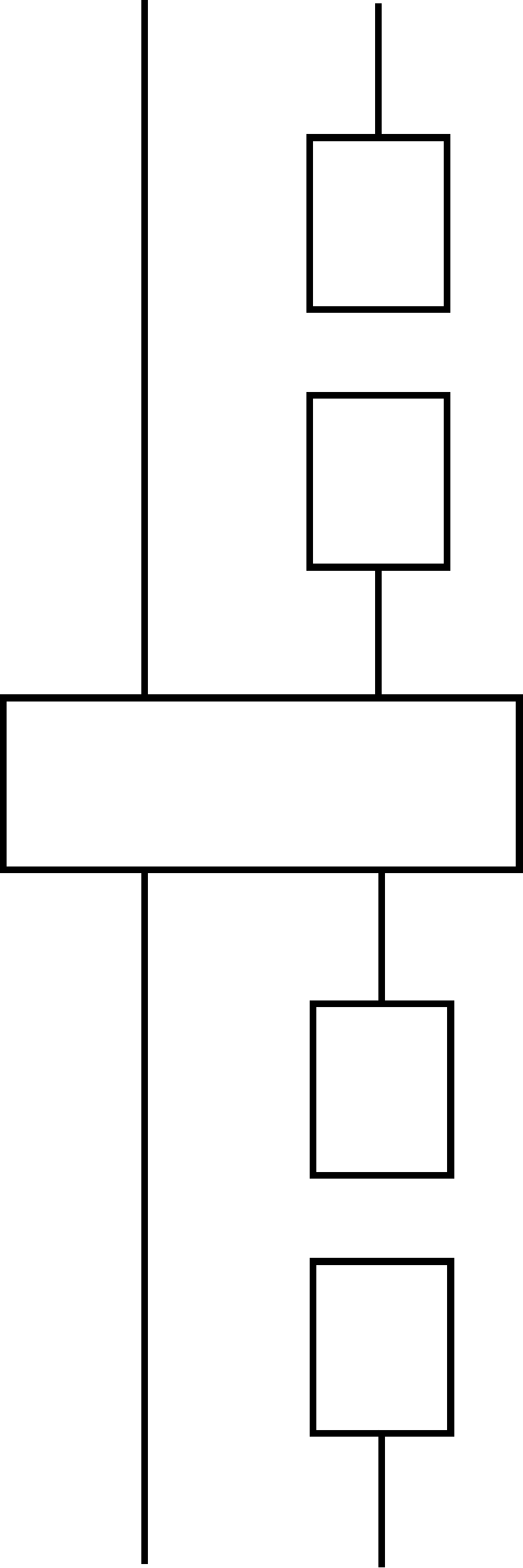}{22ex}\put(-20,1){\ms{f}} \put(-13, 41){\ms{e_i}} \put(-14, 23){\ms{\varphi_i}} 
\put(-14, -39){\ms{\varphi_j}} \put(-13, -20){\ms{e_j}}\ 
=\  \sum_{i, j\in I}\epsh{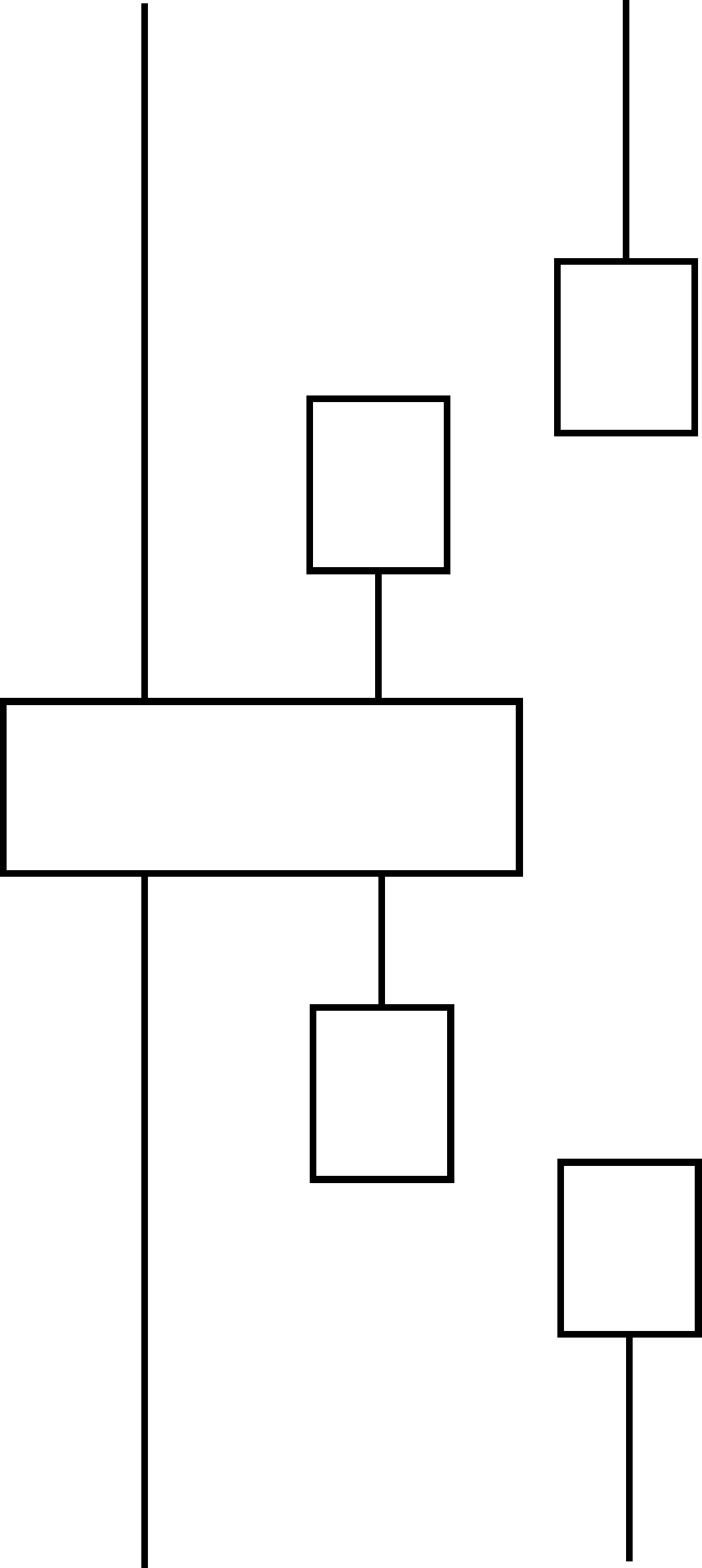}{22ex}\put(-33,1){\ms{f}} \put(-27, 23){\ms{\varphi_i}}  \put(-26, -20){\ms{e_j}} \put(-9, 33){\ms{e_i}} \put(-9, -32){\ms{\varphi_j}}\ 
=\  \sum_{i, j\in I}f_{ij}\otimes (e_i \varphi_j) \Id_{_\varepsilon W}
\end{equation*}
where $f_{ij}=\widetilde{\varphi}_i f\widetilde{e}_j\in \End_{H\text{-mod}}(V)$. From this graphical representation implies
\begin{equation*}
t_{V}(\tr_{_\varepsilon W}^{r}(f))=\sum_{i, j\in I}t_{V}\left(\epsh{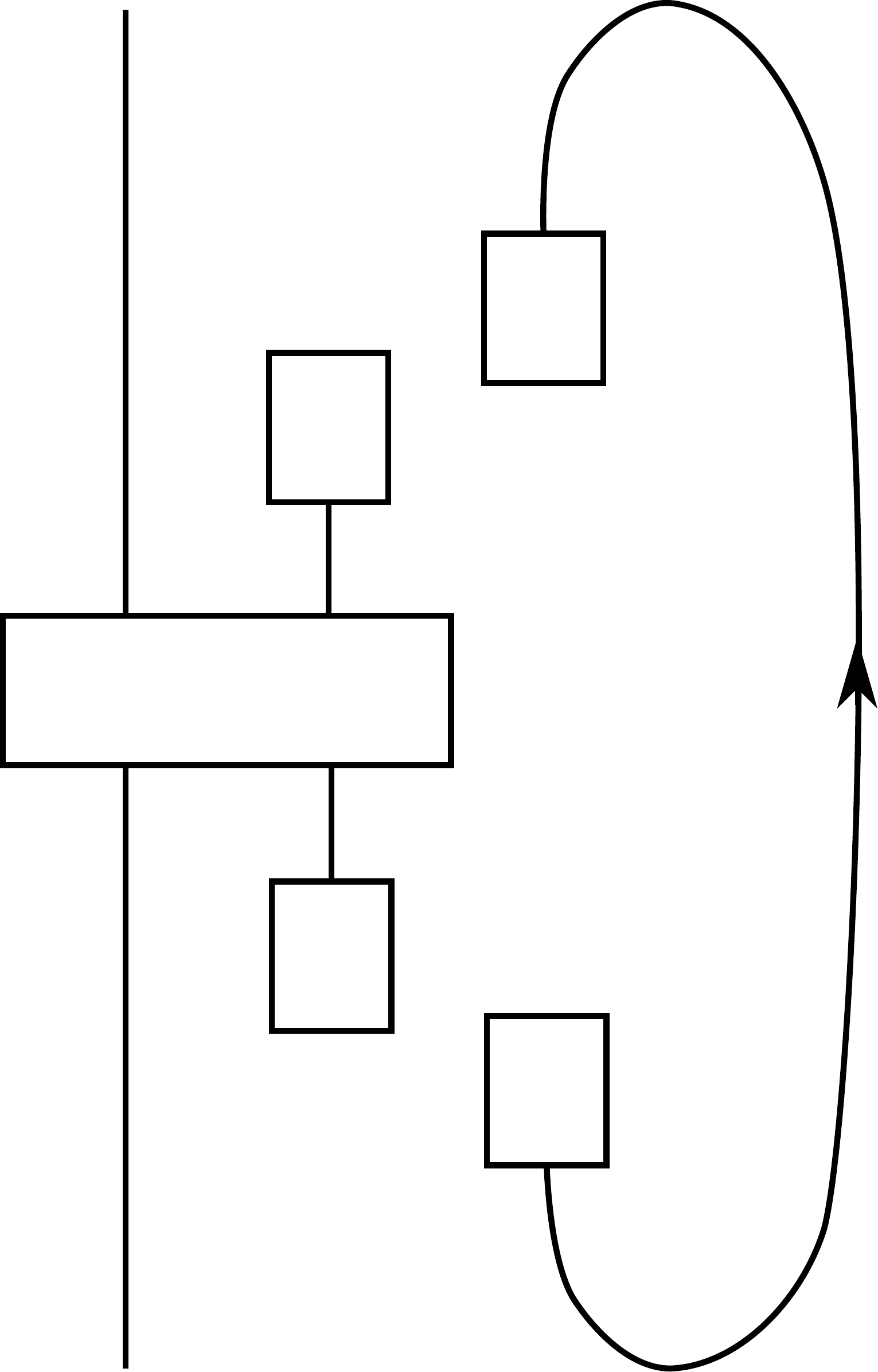}{22ex}\right)\put(-67,1){\ms{f}} \put(-61, 23){\ms{\varphi_i}}  \put(-61, -20){\ms{e_j}} \put(-43, 33){\ms{e_i}} \put(-43, -32){\ms{\varphi_j}} \put(-77,-62){\ms{V}}  \put(-32,-62){\ms{_\varepsilon W}}
=\sum_{i\in I}t_{V}(f_{ii}).
\end{equation*}
Therefore Equality \eqref{lemma partial trace of cyclic trace right} holds.\\
Remark that the pivotal element acts trivially on $_\varepsilon W$ so the evaluation $\tev_{_\varepsilon W}$ in $H$-mod is just the usual evaluation of $\Vect_{\Bbbk}$.

For Equality \eqref{lemma partial trace of cyclic trace left} the proof is similar.
\end{proof}
\subsubsection{Reduction Lemma} We have a graded version of Reduction Lemma \cite[Lemma 3.2]{BBG17}
\begin{Lem}\label{hq chinh}
Let $(H, g)$ be a finite type unimodular pivotal Hopf $G$-coalgebra and $\lambda=(\lambda^{\alpha})_{\alpha\in G}\in \prod_{\alpha\in G}H_{\alpha}^{*}$ be a family of symmetric linear forms and $\mtr=(\mtr^{\alpha})_{\alpha\in G}$ be the associated cyclic traces. Then $\mtr$ is  a right modified trace on $H$-pmod if and only if for all $\alpha, \beta\in G$ and for all $f\in \End_{H_{\alpha\beta}}(H_{\alpha}\otimes H_{\beta})$
\begin{equation}\label{tc right trace tren pt sinh}
\mtr_{H_{\alpha}\otimes H_{\beta}}^{\alpha\beta}(f)=\mtr_{H_{\alpha}}^{\alpha}(\tr_{H_{\beta}}^{r}(f)).
\end{equation}
Similarly, $\mtr$ is a left modified trace on $H$-pmod if and only if for all $f\in \End_{H_{\alpha\beta}}(H_{\alpha}\otimes H_{\beta})$
\begin{equation*}
\mtr_{H_{\alpha}\otimes H_{\beta}}^{\alpha\beta}(f)=\mtr_{H_{\beta}}^{\beta}(\tr_{H_{\alpha}}^{l}(f)).
\end{equation*}
\end{Lem}
\begin{proof}
The proof strictly follows the line of Reduction Lemma 3.2 \cite{BBG17}. 
The necessity is obvious. We now prove the sufficiency of the condition. By Proposition \ref{pro xd cyclic trace tu sym linear} for each $\alpha\in G$ the symmetric linear form $\lambda^{\alpha}$ induces the cyclic trace $\{\mtr_{P}^{\alpha}:\ \End_{H_{\alpha}}(P)\rightarrow \Bbbk\}_{P\in H_{\alpha}\text{-pmod}}$. We then prove that the cyclic trace $\mtr^{\alpha}$ satisfies the right partial trace property.\\
First, let $P \in H_{\alpha}\text{-pmod},\ P'\in H_{\beta}\text{-pmod}$ and $f\in \End_{H_{\alpha\beta}}(P\otimes P')$. Suppose that $\Id_{P}$ and $\Id_{P'}$ have the decomposition
\begin{equation}\label{pt idP}
\Id_{P}=\sum a_i\circ b_i, \qquad \Id_{P'}=\sum a_{i^{'}}\circ b_{i^{'}}
\end{equation}
where $a_i:\ H_{\alpha}\rightarrow P,\ b_i:\ P \rightarrow H_{\alpha}$ and $a_{i^{'}}:\ H_{\beta}\rightarrow P',\ b_{i^{'}}:\ P'\rightarrow  H_{\beta}$. The modified trace of $f$ is calculated as follows:
\begin{align}\label{ptrace voi hai projec module}
\mtr_{P\otimes P'}^{\alpha\beta}(f)
&=\mtr_{P\otimes P'}^{\alpha\beta}\left(\epsh{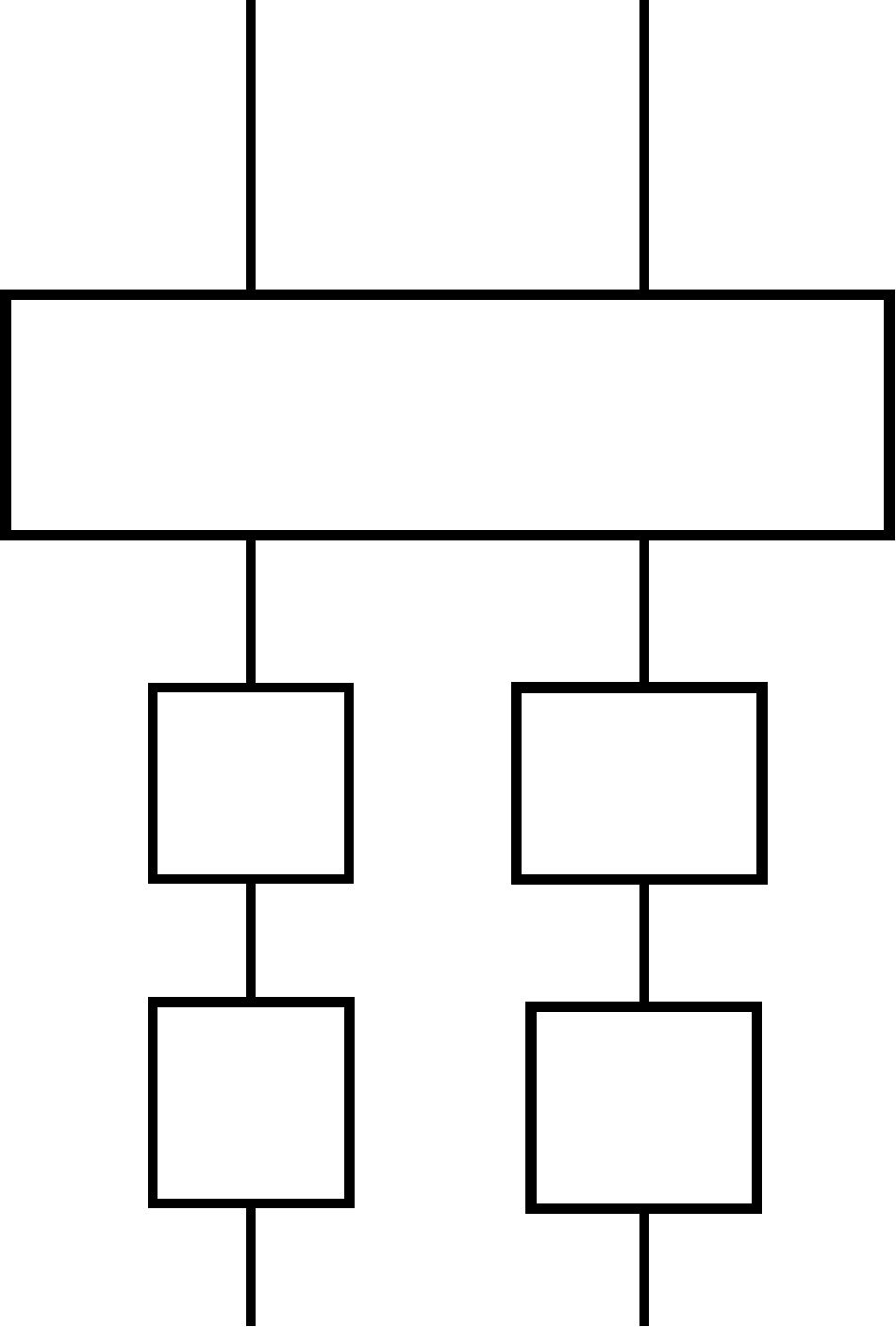}{12ex} \right)\put(-34,13){\ms{f}} \put(-46, -4){\ms{a_i}} \put(-45, -19){\ms{b_i}} \put(-29, -3){\ms{a_{i^{'}}}} \put(-28, -19){\ms{b_{i^{'}}}}
=\mtr_{H_{\alpha}\otimes H_{\beta}}^{\alpha\beta}\left(\epsh{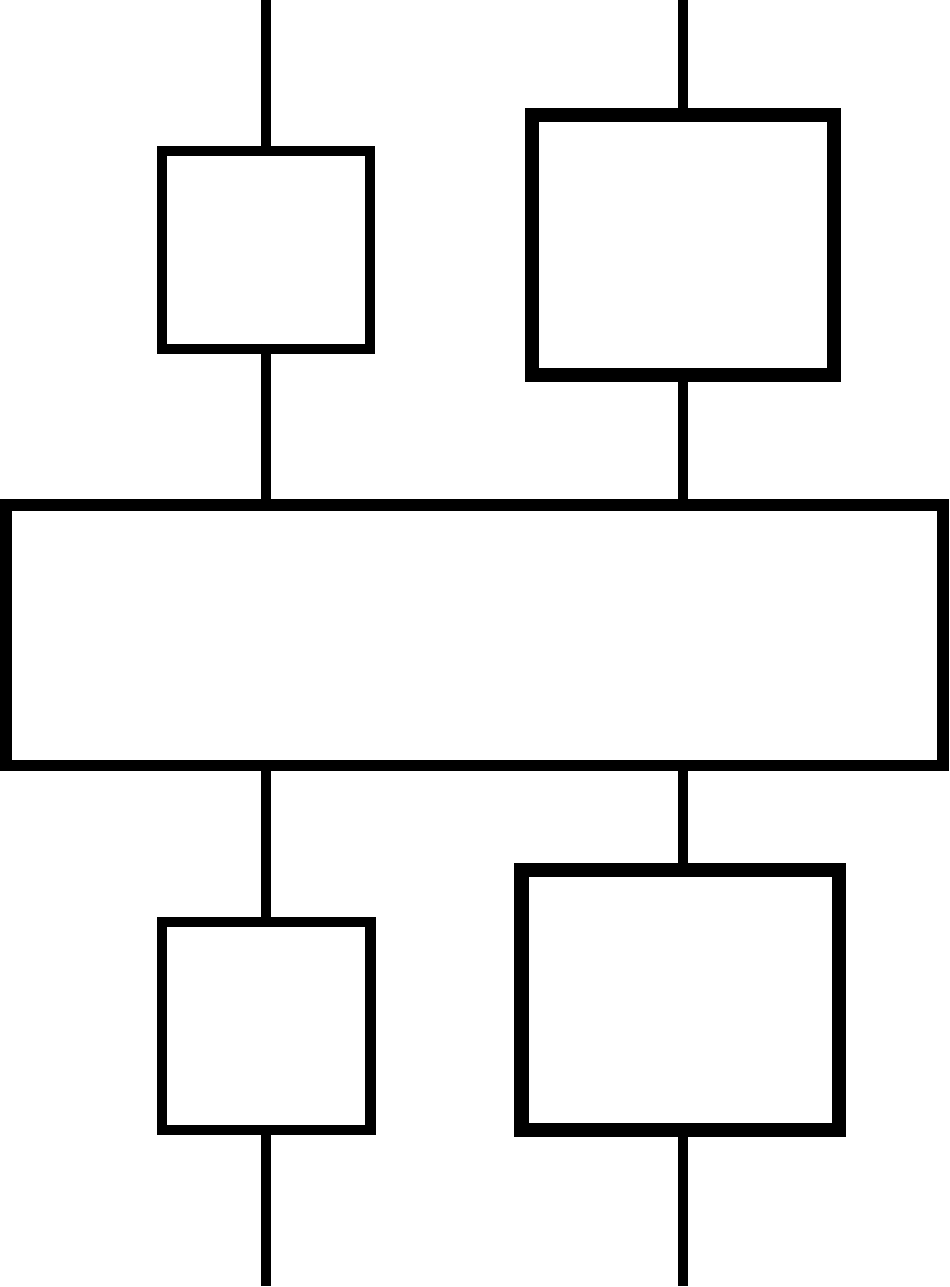}{12ex} \right) \put(-36,2){\ms{f}} \put(-48, -16){\ms{a_i}} \put(-29, -14){\ms{a_{i^{'}}}} \put(-48, 20){\ms{b_i}} \put(-29, 21){\ms{b_{i^{'}}}}\\
&=\mtr_{H_{\alpha}}^{\alpha}\left(\epsh{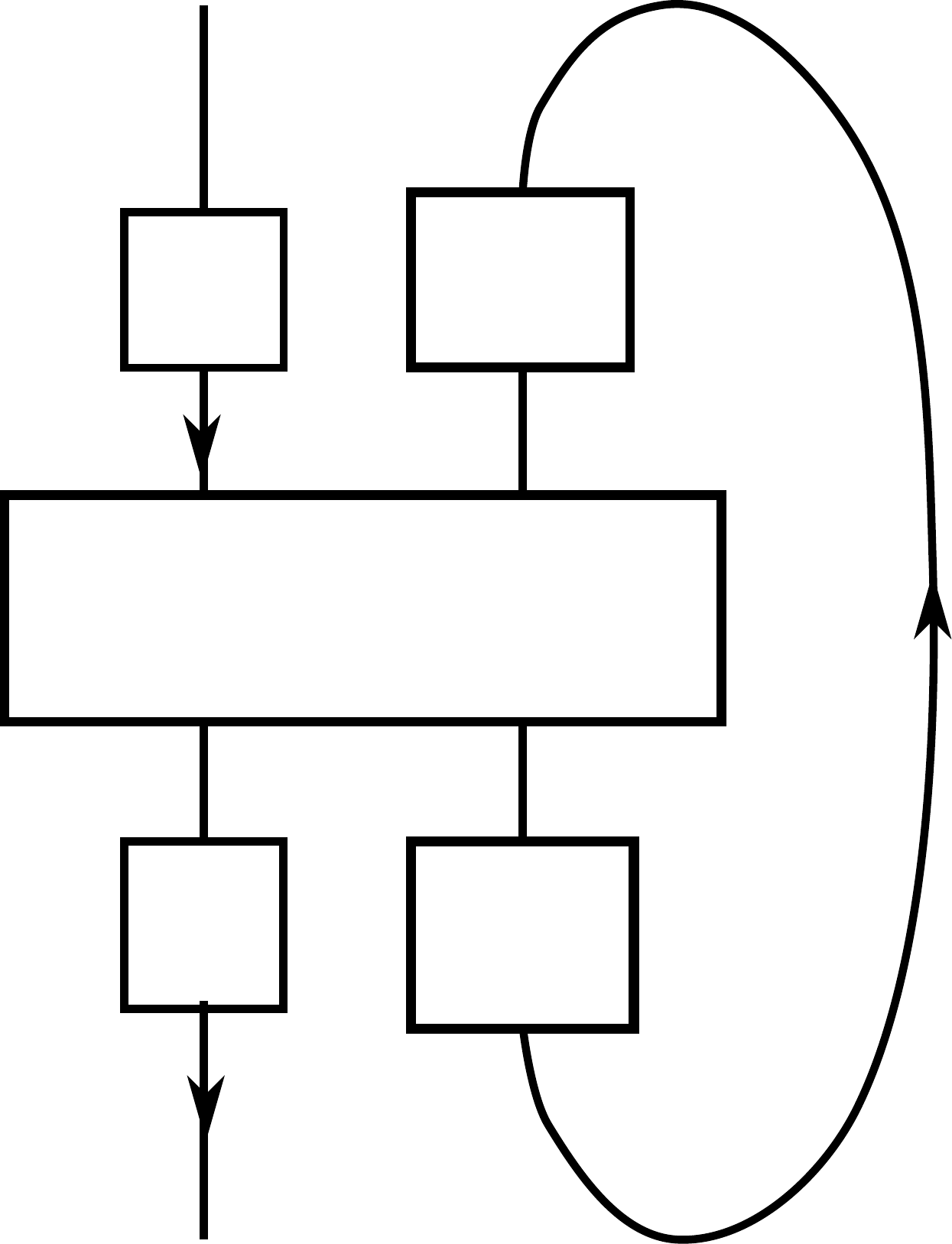}{14ex} \right) \put(-48,2){\ms{f}} \put(-59, -16){\ms{a_i}} \put(-41, -15){\ms{a_{i^{'}}}} \put(-59, 20){\ms{b_i}} \put(-41, 22){\ms{b_{i^{'}}}}
= \mtr_{P}^{\alpha}\left(\epsh{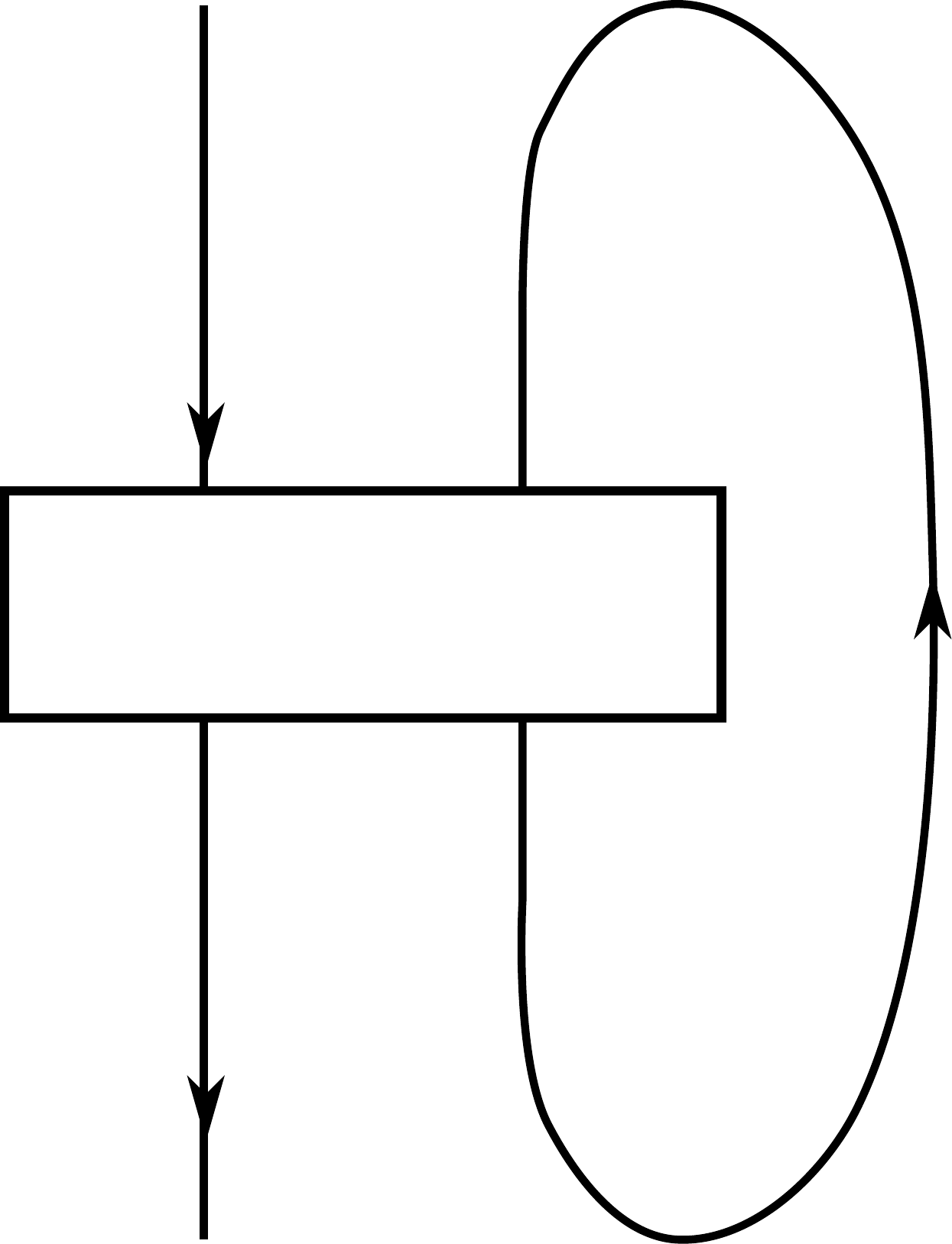}{14ex} \right)\put(-48,2){\ms{f}}\nonumber\\
&=\mtr_{P}^{\alpha}\left(\tr_{P'}^{r}(f) \right). \nonumber
\end{align}
In this calculation, one uses \eqref{pt idP} in the first equality, in the second equality one uses the cyclicity property of cyclic traces, the third equality thanks to \eqref{tc right trace tren pt sinh} and finally one uses the duality morphisms to move $b_{i'}$ around the loop then applying again \eqref{pt idP} and the cyclicity property.\\
Second, let $P \in H_{\alpha}\text{-pmod},\ V\in H_{\beta}\text{-mod}$ and $f\in \End_{H_{\alpha\beta}}(P\otimes V)$. Set $Q=P\otimes V$, note that $Q\in H_{\alpha\beta}\text{-pmod}$ and $P\otimes P^{*}, \ Q\otimes Q^{*}\in H_1\text{-pmod} $.
Consider two morphisms $A\in \Hom_{H_{1}\text{-mod}}(P \otimes P^*, Q\otimes Q^*)$ and $B\in \Hom_{H_{1}\text{-mod}}(Q\otimes Q^*, P \otimes P^*)$ are given by
\begin{equation*}
A=\epsh{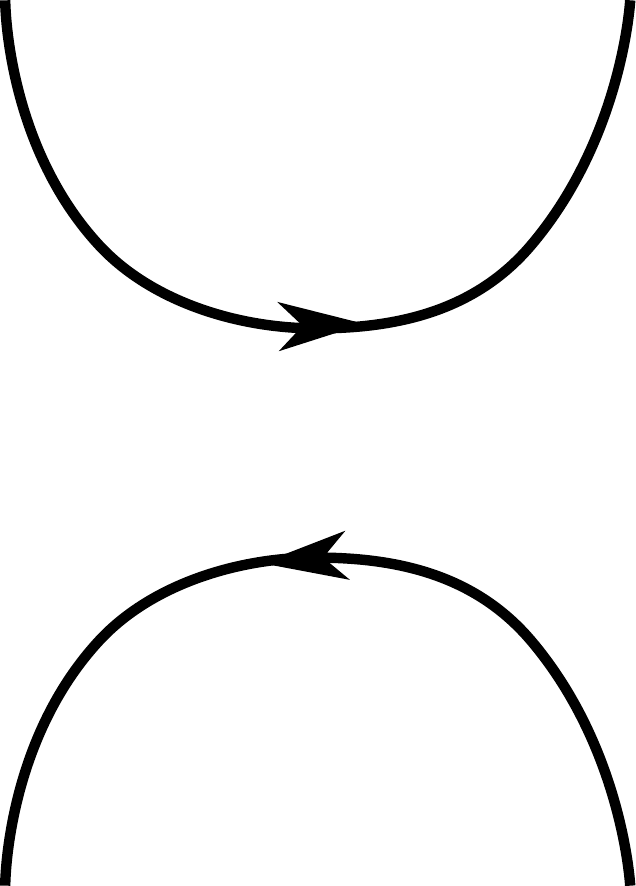}{8ex}\put(-34,-26){\ms{P}} \put(-2,-26){\ms{P}} \put(-34,28){\ms{Q}} \put(-2,28){\ms{Q}} \quad , \qquad B= \epsh{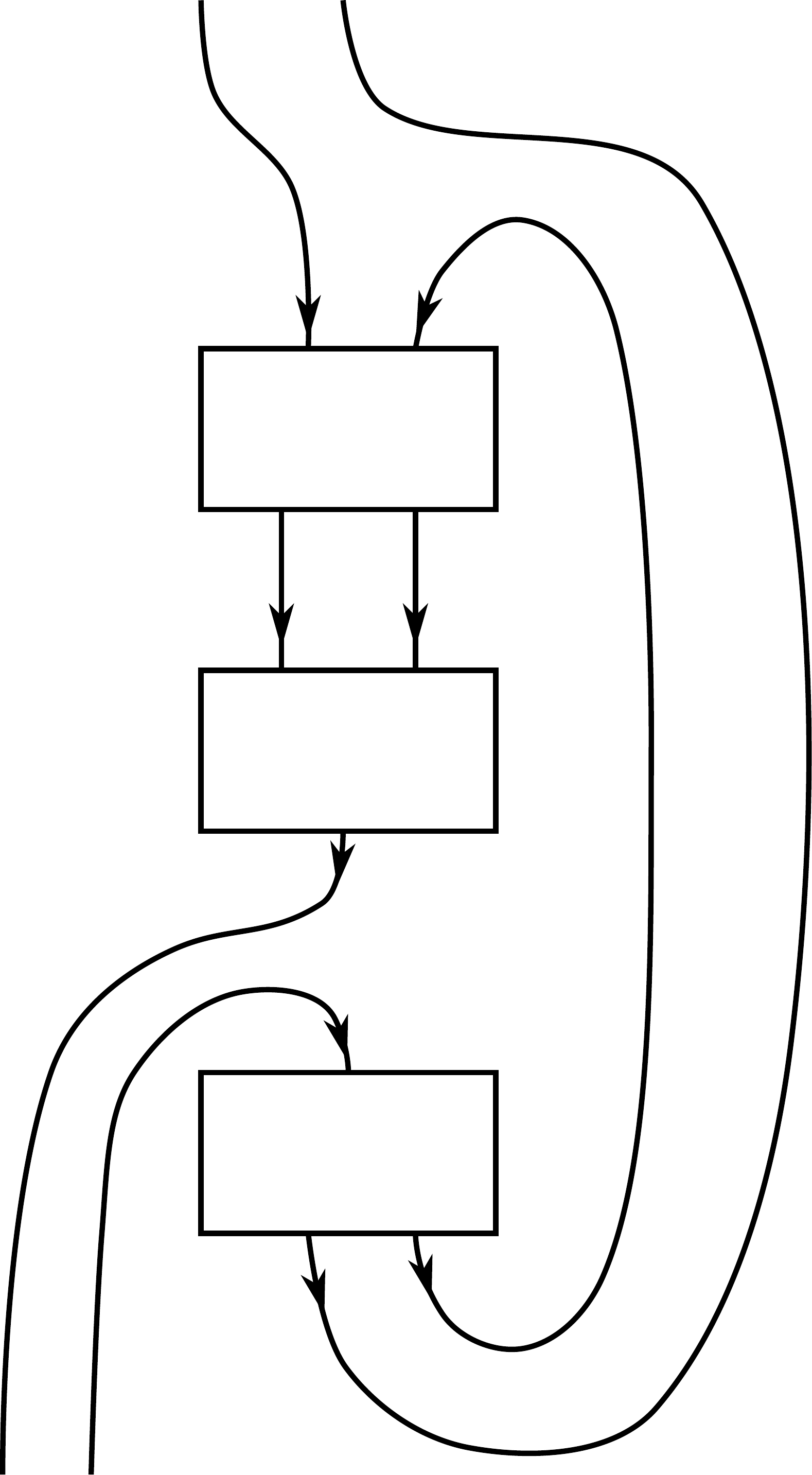}{20ex} \put(-45,58){\ms{P}} \put(-34,58){\ms{P}} \put(-60,-57){\ms{Q}} \put(-50,-57){\ms{Q}}\put(-34,23){\ms{f}} \put(-36,0){\ms{\Id}} \put(-36,-28){\ms{\Id}} \put(-18,0){\ms{V}}\ .
\end{equation*}
According to \eqref{ptrace voi hai projec module} one gets
\begin{align*}
\mtr_{P\otimes P^{*}}^{1}(B\circ A)
&=\mtr_{P}^{\alpha}\left(\tr_{P^{*}}^{r}(B\circ A)\right)\\
&=\mtr_{P}^{\alpha}\left(\epsh{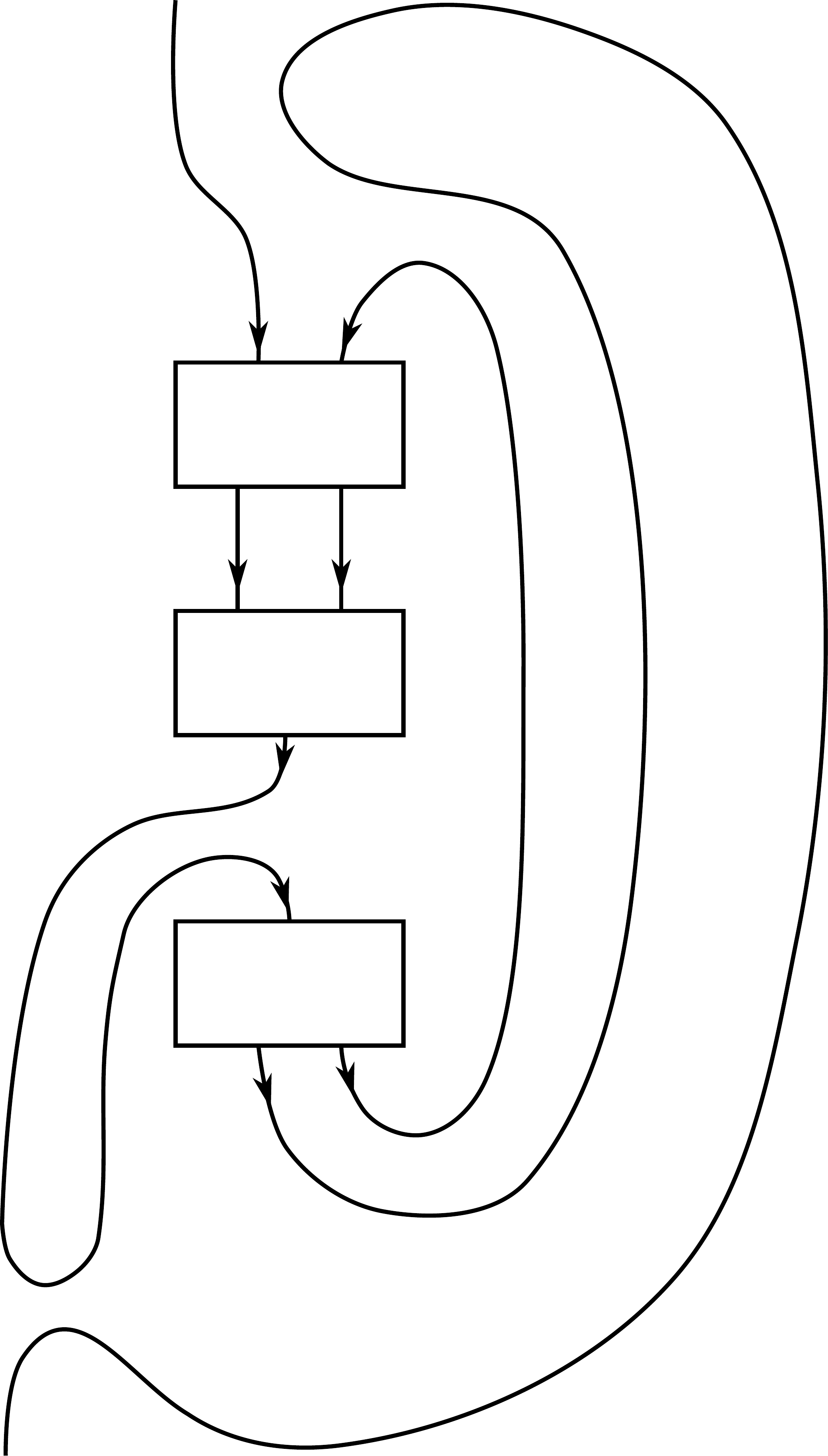}{24ex}\right) \put(-59,26){\ms{f}} \put(-61,5){\ms{\Id}} \put(-61,-22){\ms{\Id}} \put(-45,-4){\ms{V}} \put(-80,-4){\ms{Q}} \put(-76,54){\ms{P}} \put(-79,-60){\ms{P}}
=\mtr_{P}^{\alpha}\left(\epsh{fig457}{14ex}\right) \put(-47,3){\ms{f}} \put(-58,42){\ms{P}} \put(-58,-42){\ms{P}} \put(-34,43){\ms{V}}
=\mtr_{P}^{\alpha}\left(\tr_{V}^{r}(f)\right).
\end{align*}
In above calculation, one applies the definition of the partial trace in second equality, in the third equality one uses the properties of the pivotal structure. Similarly we also have 
\begin{align*}
\mtr_{Q\otimes Q^{*}}^{1}(A\circ B)
&=\mtr_{Q}^{\alpha\beta}\left(\tr_{Q^{*}}^{r}(A\circ B)\right)\\
&=\mtr_{Q}^{\alpha\beta}\left(\epsh{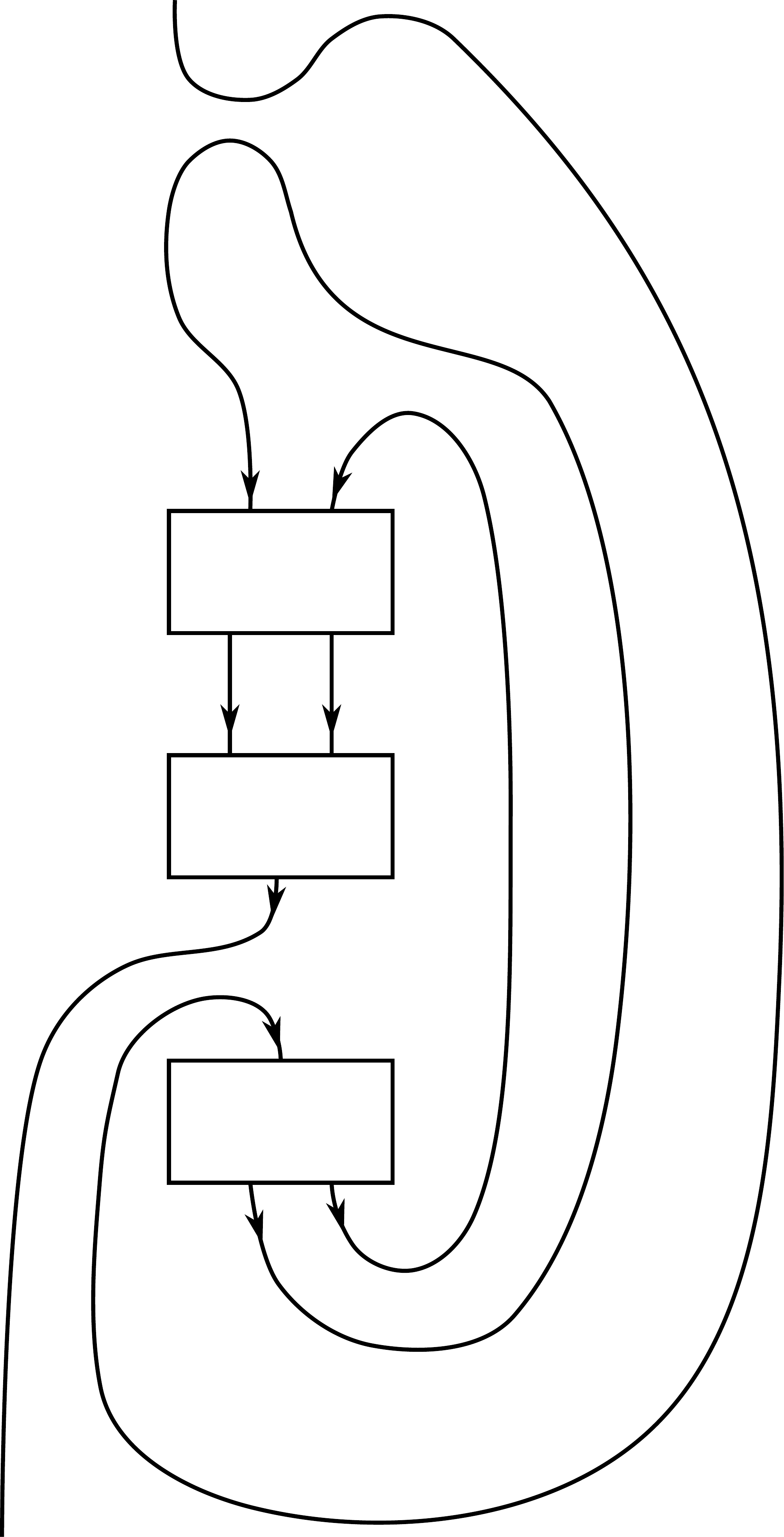}{24ex}\right) \put(-56,17){\ms{f}} \put(-57,-3){\ms{\Id}} \put(-57,-28){\ms{\Id}} \put(-43,-16){\ms{V}}  \put(-70,60){\ms{Q}} \put(-78,-66){\ms{Q}} \put(-70,40){\ms{P}}
=\mtr_{Q}^{\alpha\beta}\left(\epsh{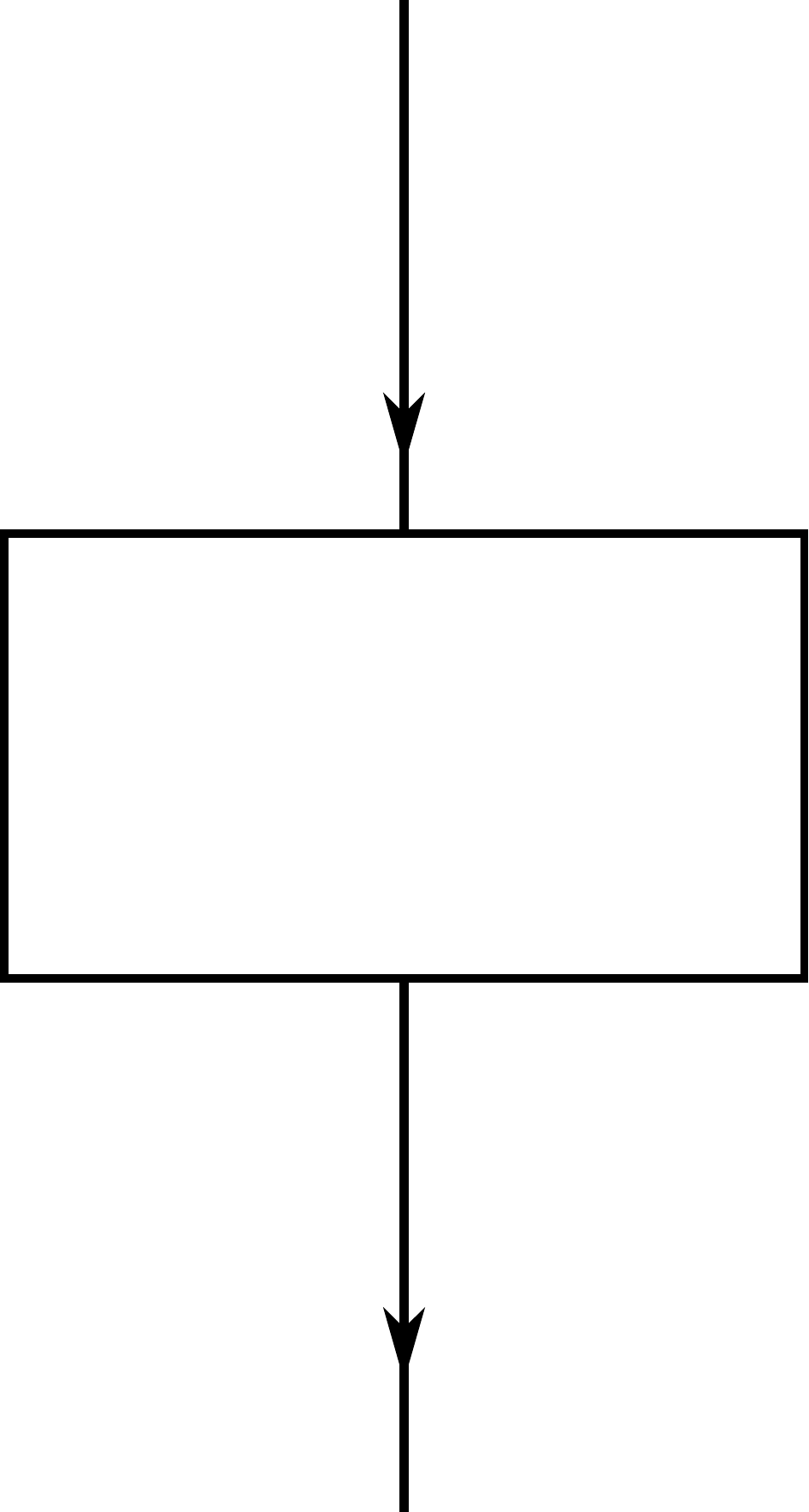}{14ex}\right)\put(-36,2){\ms{f}} \put(-34,42){\ms{Q}} \put(-34,-42){\ms{Q}}
=\mtr_{P\otimes V}^{\alpha\beta}(f).
\end{align*}
Since the cyclicity property $\mtr_{P\otimes P^{*}}^{1}(B\circ A)=\mtr_{Q\otimes Q^{*}}^{1}(A\circ B)$, it follows that $\mtr_{P\otimes V}^{\alpha\beta}(f)=\mtr_{P}^{\alpha}\left(\tr_{V}^{r}(f)\right)$.\\
The proof in the case of the left modified trace is similar.
\end{proof}
\subsection{Applications of Theorem \ref{main theorem}}
Theorem \ref{main theorem} has two immediate consequences when $G=\{1\}$ or $H$ is semi-simple. First, in degree $1$ the symmetrised $G$-integral is also the symmetrised integral of $H_1$ and Theorem \ref{main theorem} recovers the main theorem of \cite{BBG17} that we recall here:
\begin{The}[\cite{BBG17}]\label{Theorem of BBG}
Let $(H,g)$ be a finite dimensional unimodular pivotal Hopf algebra over a field $\Bbbk$. Then the space of right (left) modified traces on $H$-pmod is equal to the space of symmetrised right (left) integrals, and hence is 1-dimensional. Moreover, the right modified trace on $H$-pmod is non-degenerate and determined by the cyclicity property and by
\begin{equation*}
\mtr_{H}(f) = \mu(gf(1)) \ \text{
for any}\  f\in\End_{H}(H)\ .
\end{equation*}
Similarly, the left modified trace is non degenerate and determined by
\begin{equation*}
\mtr_{H}(f) = \mu^{l}(g^{-1}f(1)) \ \text{
for any}\  f\in\End_{H}(H)\ .
 \end{equation*}
In particular, $H$ is unibalanced if and only if the right modified trace is also left. 
\end{The}
Second, for a finite type unimodular pivotal Hopf $G$-coalgebra $(H,g)$, if $H$ is semi-simple, i.e. $H_{\alpha}$ is semi-simple for all $\alpha\in G$ then $H\text{-pmod}=\mathcal{C}$. Then the categorical trace generates the space of modified traces on $H\text{-pmod}$: for any $f\in \End_{\mathcal{C}}(V)$, the right and left categorical trace are
\begin{align*}
&\tr_{V}^{\mathcal{C}}(f):=\tev_{V} (f\otimes \Id_{V}) \coev_{V}\in \Bbbk,\\
&^{\mathcal{C}}\tr_{V}(f):=\ev_{V} (\Id_{V^{*}}\otimes f) \tcoev_{V}\in \Bbbk.
\end{align*}
As a corollary of Theorem \ref{main theorem} we then have the proposition.
\begin{Pro}
Let $(H,g)$ be a finite type unimodular pivotal Hopf $G$-coalgebra over a field $\Bbbk$. The right categorical trace $\tr_{H_{\alpha}}^{\mathcal{C}}$ and its left version $^{\mathcal{C}}\tr_{H_{\alpha}}$ are non-zero if and only if $H_{\alpha}$-mod is semi-simple and in this case coincide up to a scalar with the trace maps
\begin{equation*}
f\mapsto \widetilde{\mu}_{\alpha}(f(1_{\alpha}))\ \text{and}\ f \mapsto \widetilde{\mu}_{\alpha}^{l}(f(1_{\alpha}))
\end{equation*}
respectively, where $f\in \End_{H_{\alpha}}(H_{\alpha})$.
\end{Pro}

\section{Proof of the main theorem}
\subsection{Decomposition of tensor products of the regular representations}
We denote by $H_{\alpha}$ the left $H_{\alpha}$-module given by the left regular action.
Let us denote by $_{\varepsilon}H_{\beta}$ the vector space underlying $H_{\beta}$ equipped with the $H_{1}$-module structure given by
\begin{equation*}
 h.m=\varepsilon(h)m\ \text{for}\ m\in \  _{\varepsilon}H_{\beta},\ h\in H_{1}.
\end{equation*} 
We will use Sweedler's notation: $\Delta_{\alpha, \beta}(h)=h_{(1)}\otimes h_{(2)}$ for $h\in H_{\alpha\beta},\ h_{(1)} \in H_{\alpha}, h_{(2)} \in H_{\beta}$.
\begin{The}\label{dl phi xi}
Let $H=(H_{\alpha})_{\alpha\in G}$ be a finite type Hopf $G$-coalgebra. Then
\begin{enumerate}
\item[(1)] the map 
\begin{equation*}
\begin{array}{rcl}
\phi_{\alpha, \beta}:\ H_{\alpha\beta} \otimes\ _{\varepsilon}H_{\beta} &\rightarrow &H_{\alpha} \otimes H_{\beta}\\
h\otimes m &\mapsto &h_{(1)}\otimes h_{(2)}m
\end{array}
\end{equation*}
is an isomorphism of $H_{\alpha\beta}$-modules whose inverse is
\begin{equation*}
\begin{array}{rcl}
\psi_{\alpha, \beta}:\ H_{\alpha} \otimes H_{\beta} & \rightarrow &H_{\alpha\beta} \otimes\ _{\varepsilon}H_{\beta}\\
x \otimes y &\mapsto &x_{(1)}\otimes S_{\beta^{-1}}(x_{(2)})y.
\end{array}
\end{equation*}
\item[(2)] the map 
\begin{equation*}
\begin{array}{rcl}
\phi_{\alpha, \beta}^{l}:\ \ _{\varepsilon}H_{\alpha} \otimes  H_{\alpha\beta}  &\rightarrow &H_{\alpha} \otimes H_{\beta}\\
m\otimes h &\mapsto &h_{(1)}m\otimes h_{(2)}
\end{array}
\end{equation*}
is an isomorphism of $H_{\alpha\beta}$-modules whose inverse is
\begin{equation*}
\begin{array}{rcl}
\psi_{\alpha, \beta}^{l}:\ H_{\alpha} \otimes H_{\beta} & \rightarrow &\ _{\varepsilon}H_{\alpha} \otimes  H_{\alpha\beta}\\
x \otimes y &\mapsto &S_{\alpha^{-1}}^{-1}(y_{(1)})x \otimes y_{(2)}.
\end{array}
\end{equation*}
\end{enumerate}
\end{The}
We prove the theorem using graphical calculus with the graphical representations for Hopf $G$-coalgebras given in Section \ref{the graphical representation}. The maps $\phi_{\alpha, \beta}$ and $\psi_{\alpha, \beta}$ are presented in Figure \ref{Fig phi va psi}.
\begin{figure}
$$\epsh{fig425phi}{14ex}\put(-27,2){\ms{\phi_{\alpha,\beta}}}\put(-36,-42){\ms{\alpha\beta}} \put(-33,42){\ms{\alpha}} \put(-13,42){\ms{\beta}} \put(-13,-42){\ms{_\varepsilon\beta}}
=\epsh{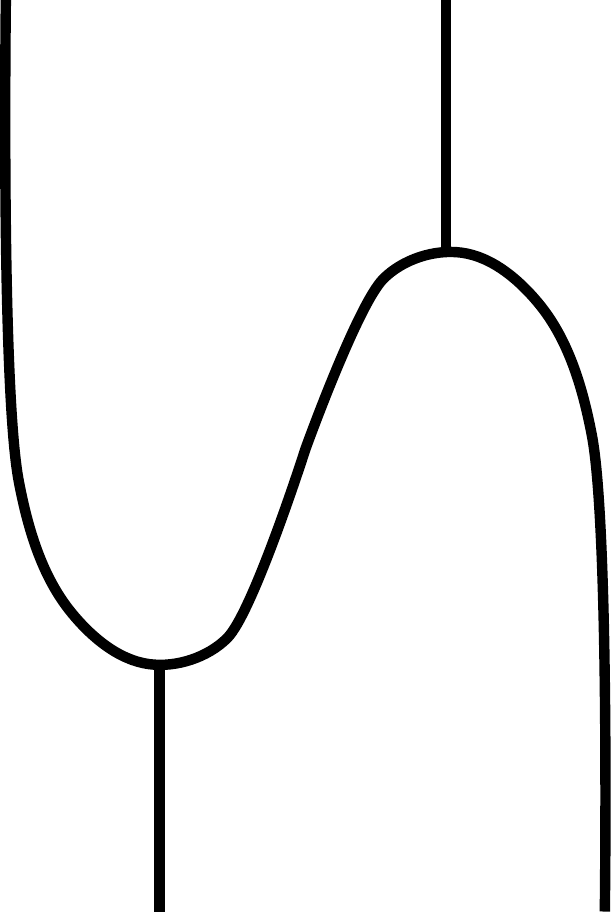}{14ex}\put(-42,-42){\ms{\alpha\beta}} \put(-49,42){\ms{\alpha}} \put(-16,42){\ms{\beta}} \put(-3,-42){\ms{_\varepsilon\beta}}
\ ,\qquad  \epsh{fig425phi}{14ex}\put(-27,2){\ms{\psi_{\alpha,\beta}}}\put(-34,-42){\ms{\alpha}} \put(-35,42){\ms{\alpha\beta}} \put(-13,42){\ms{_\varepsilon\beta}} \put(-13,-42){\ms{\beta}}=\epsh{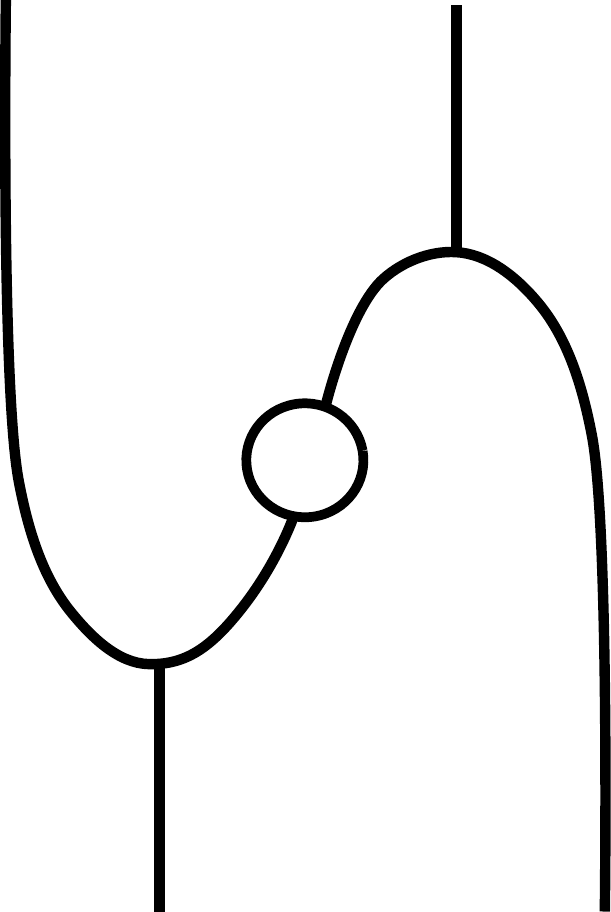}{14ex}\put(-39,-42){\ms{\alpha}} \put(-52,42){\ms{\alpha\beta}} \put(-16,42){\ms{_\varepsilon\beta}} \put(-3,-42){\ms{\beta}}$$
\caption{The graphical representations of $\phi_{\alpha, \beta}$ and $\psi_{\alpha, \beta}$}\label{Fig phi va psi}
\end{figure}
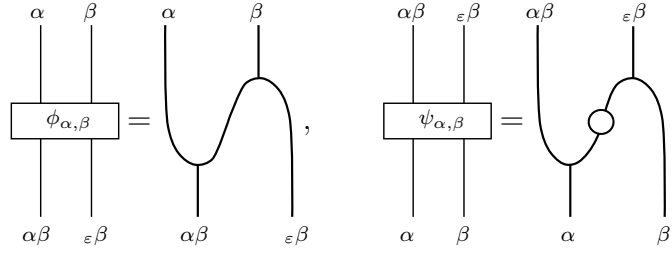
The graphical representations for $\phi_{\alpha, \beta}^{l}$ and $\psi_{\alpha, \beta}^{l}$ are similar.
\begin{proof}
In order to prove part (1), we first check that $\phi_{\alpha, \beta}$ is left inverse to $\psi_{\alpha, \beta}$, by computing the composition one gets
\begin{equation*}
\psi_{\alpha, \beta}\circ \phi_{\alpha, \beta}
=\epsh{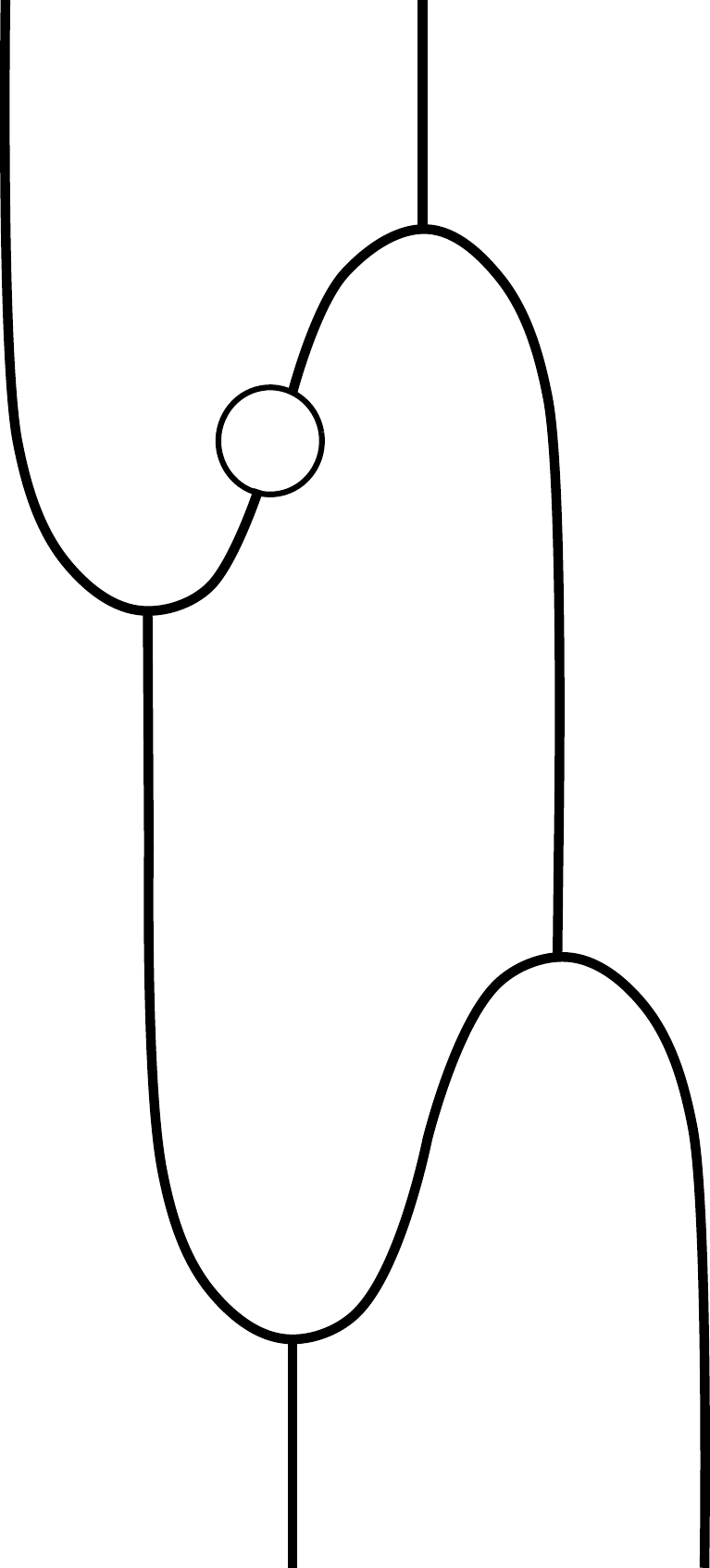}{16ex} \put(-28,-46){\ms{\alpha\beta}} \put(-6,-46){\ms{_\varepsilon\beta}} \put(-42,48){\ms{\alpha\beta}} \put(-18,48){\ms{_\varepsilon\beta}} 
\put(-36, -20){\ms{\alpha}} \put(-22, -20){\ms{\beta}} \ 
=\ \epsh{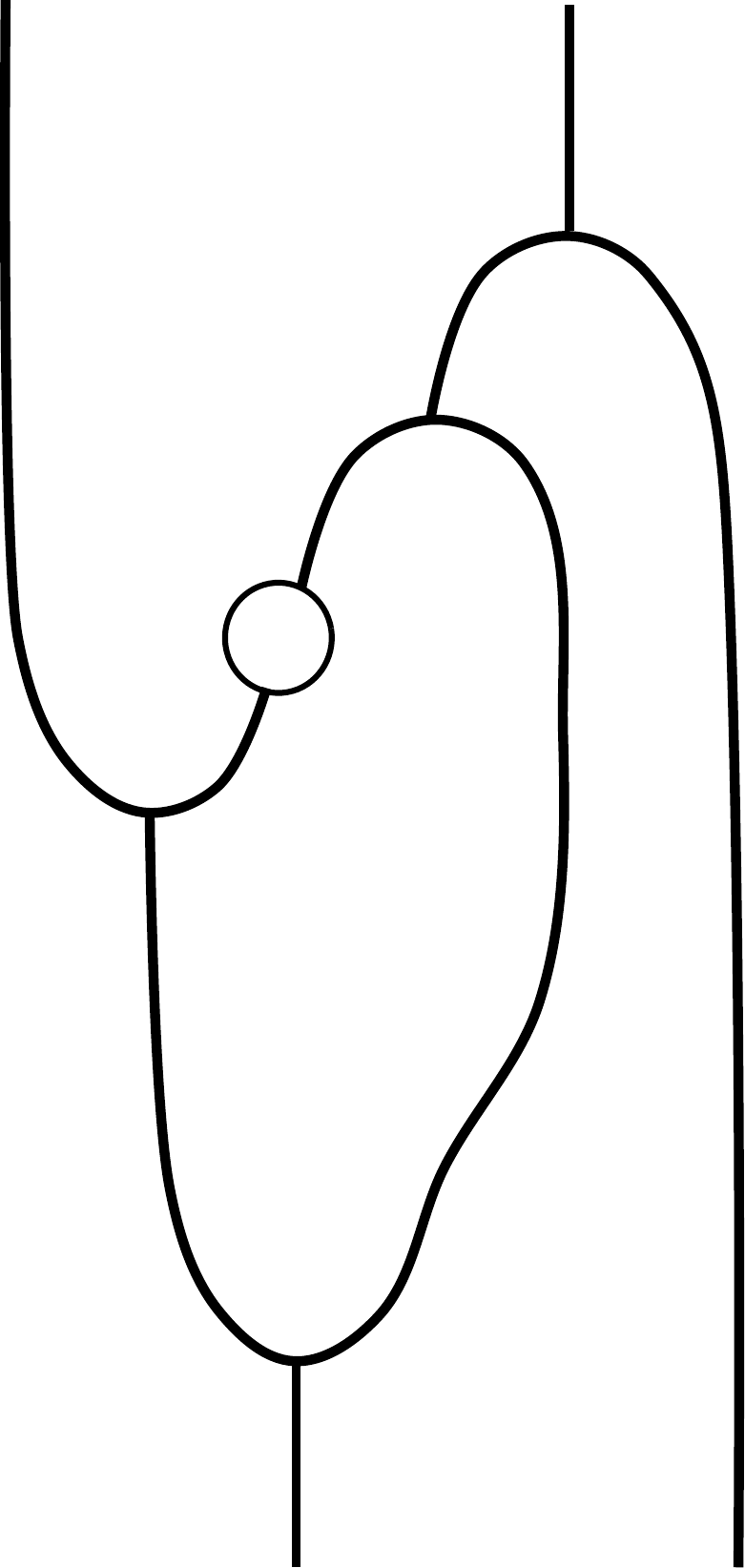}{16ex} \put(-28,-46){\ms{\alpha\beta}} \put(-6,-46){\ms{_\varepsilon\beta}} \put(-44,48){\ms{\alpha\beta}} \put(-14,48){\ms{_\varepsilon\beta}} \put(-37, -20){\ms{\alpha}} \put(-23, -20){\ms{\beta}} \ 
=\ \epsh{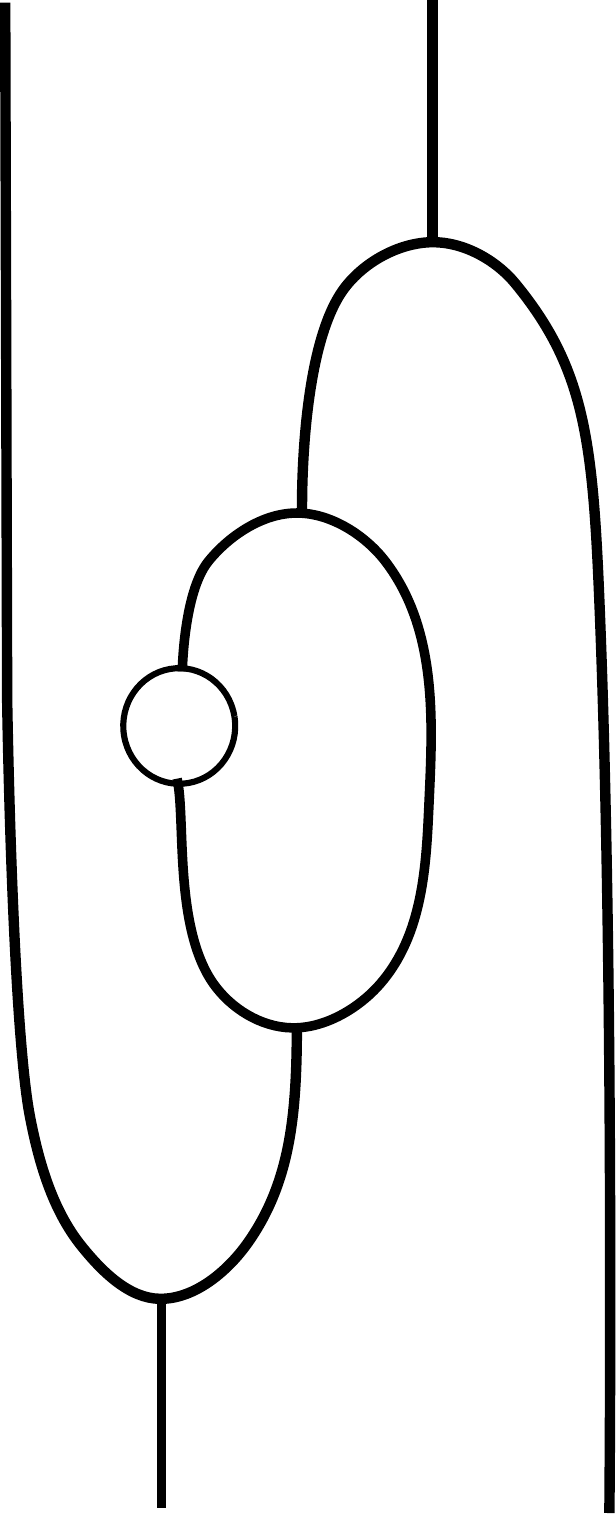}{16ex} \put(-28,-46){\ms{\alpha\beta}} \put(-6,-46){\ms{_\varepsilon\beta}} \put(-38,48){\ms{\alpha\beta}} \put(-15,48){\ms{_\varepsilon\beta}} \put(-44, -20){\ms{\alpha\beta}} \put(-23, -20){\ms{1}} \ 
=\ \epsh{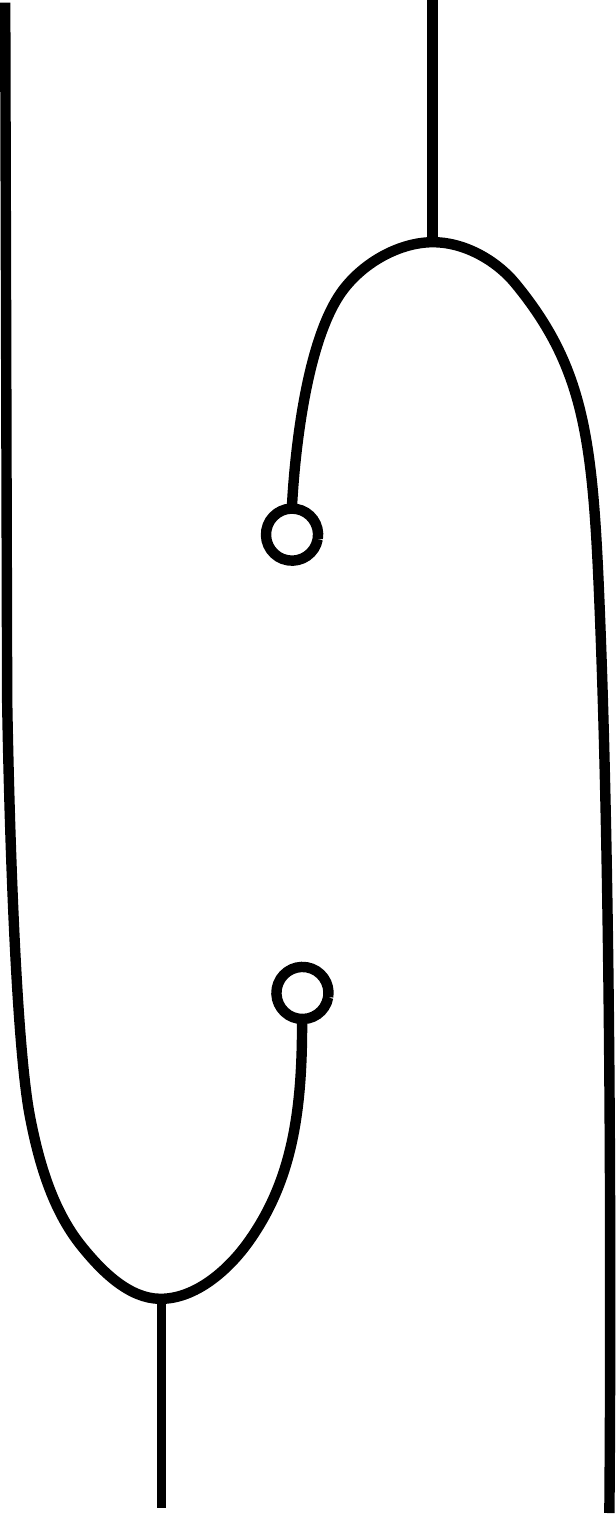}{16ex} \put(-28,-46){\ms{\alpha\beta}} \put(-6,-46){\ms{_\varepsilon\beta}} \put(-38,48){\ms{\alpha\beta}} \put(-15,48){\ms{_\varepsilon\beta}} \put(-23, -20){\ms{1}} \ 
= \Id_{H_{\alpha\beta}\otimes\ _\varepsilon H_{\beta}}
\end{equation*} 
where the associativity of the product $m_\beta$ is used in the first equality, then we use the coassociativity of the coproduct in the second equality, and finally we use the antipode properties in the last equality. Similarly we have $\phi_{\alpha, \beta}\circ \psi_{\alpha, \beta}=\Id_{H_{\alpha}\otimes H_{\beta}}$.\\
Next we prove the map $\phi_{\alpha, \beta}$ is $H_{\alpha\beta}$-linear by diagrammic calculus:
\begin{equation*}
\epsh{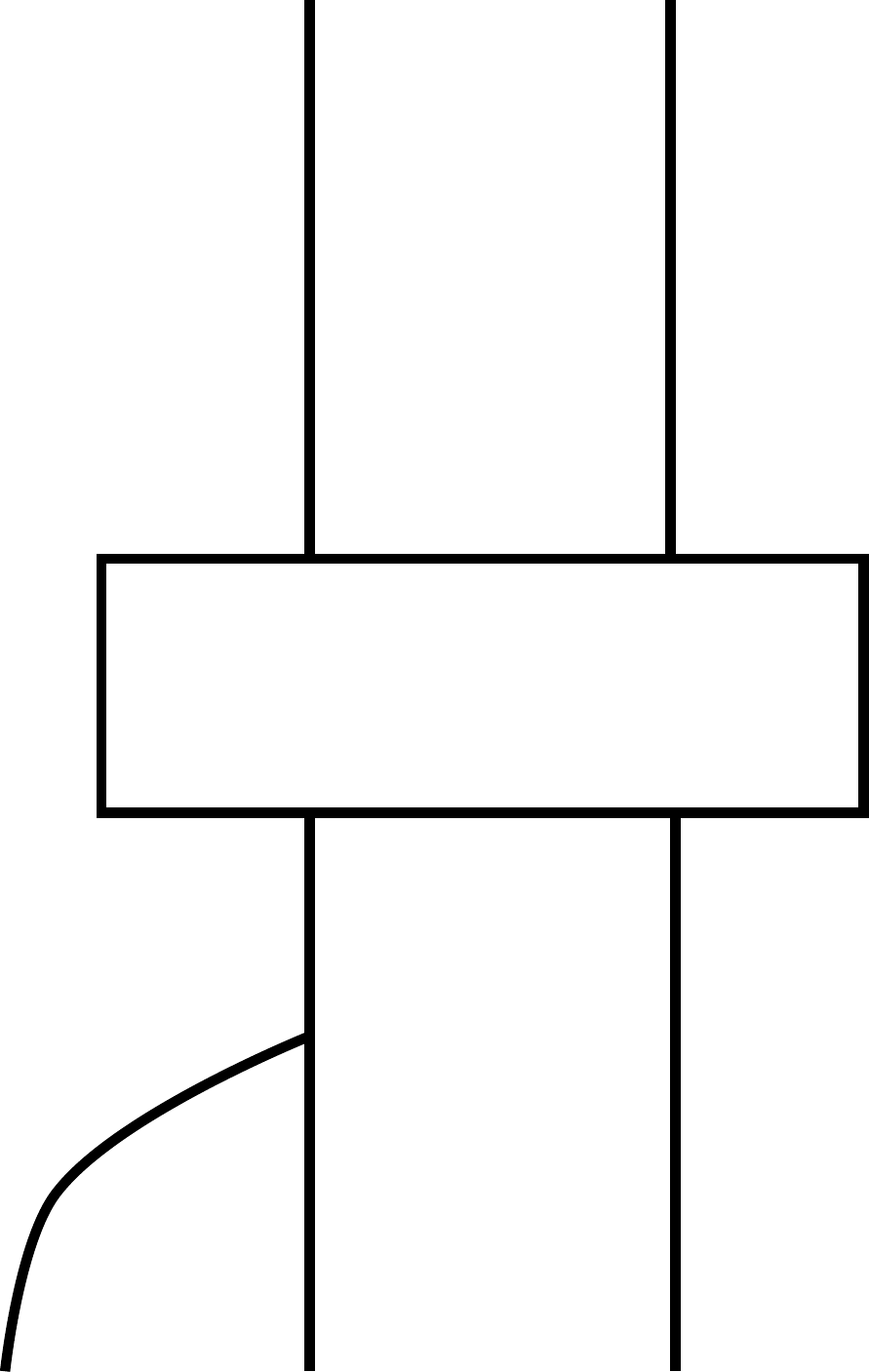}{14ex} \put(-27,2){\ms{\phi_{\alpha,\beta}}} \put(-52,-40){\ms{\alpha\beta}} \put(-34,-40){\ms{\alpha\beta}} \put(-14,-40){\ms{_\varepsilon\beta}} \put(-33,42){\ms{\alpha}} \put(-13,42){\ms{\beta}}\ 
=\ \epsh{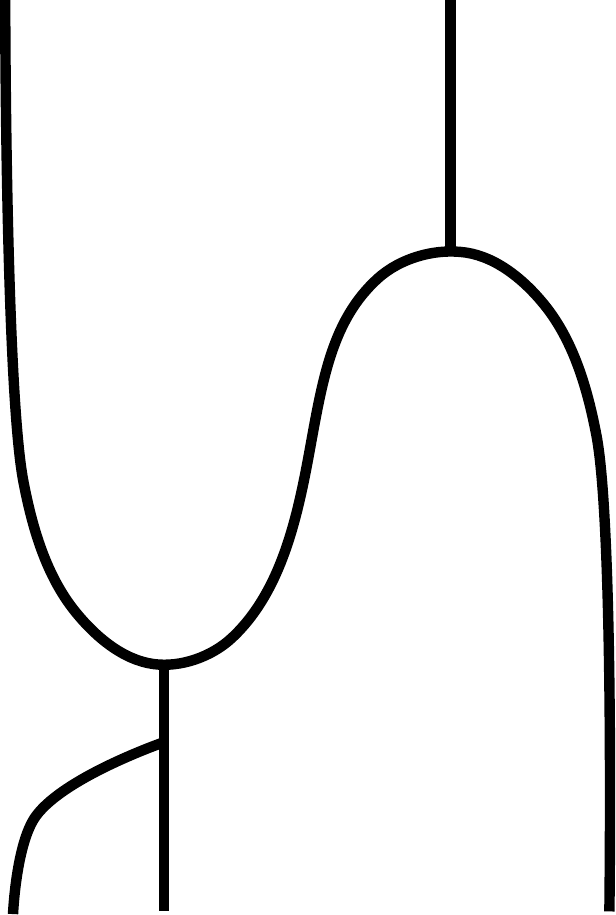}{14ex} \put(-54,-40){\ms{\alpha\beta}} \put(-39,-40){\ms{\alpha\beta}} \put(-4,-40){\ms{_\varepsilon\beta}} \put(-50,42){\ms{\alpha}} \put(-15,42){\ms{\beta}}\ 
=\ \epsh{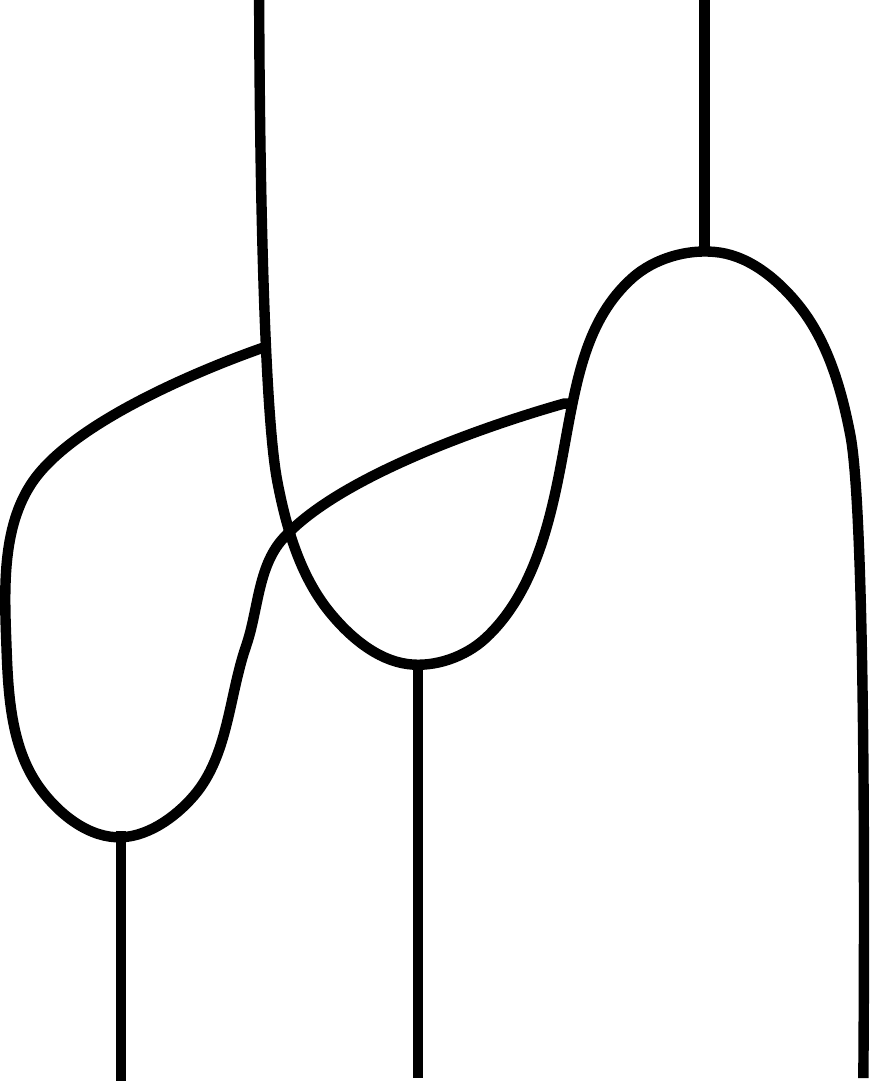}{14ex} \put(-55,-40){\ms{\alpha\beta}} \put(-35,-40){\ms{\alpha\beta}} \put(-4,-40){\ms{_\varepsilon\beta}} \put(-44,42){\ms{\alpha}} \put(-14,42){\ms{\beta}}\ 
=\ \epsh{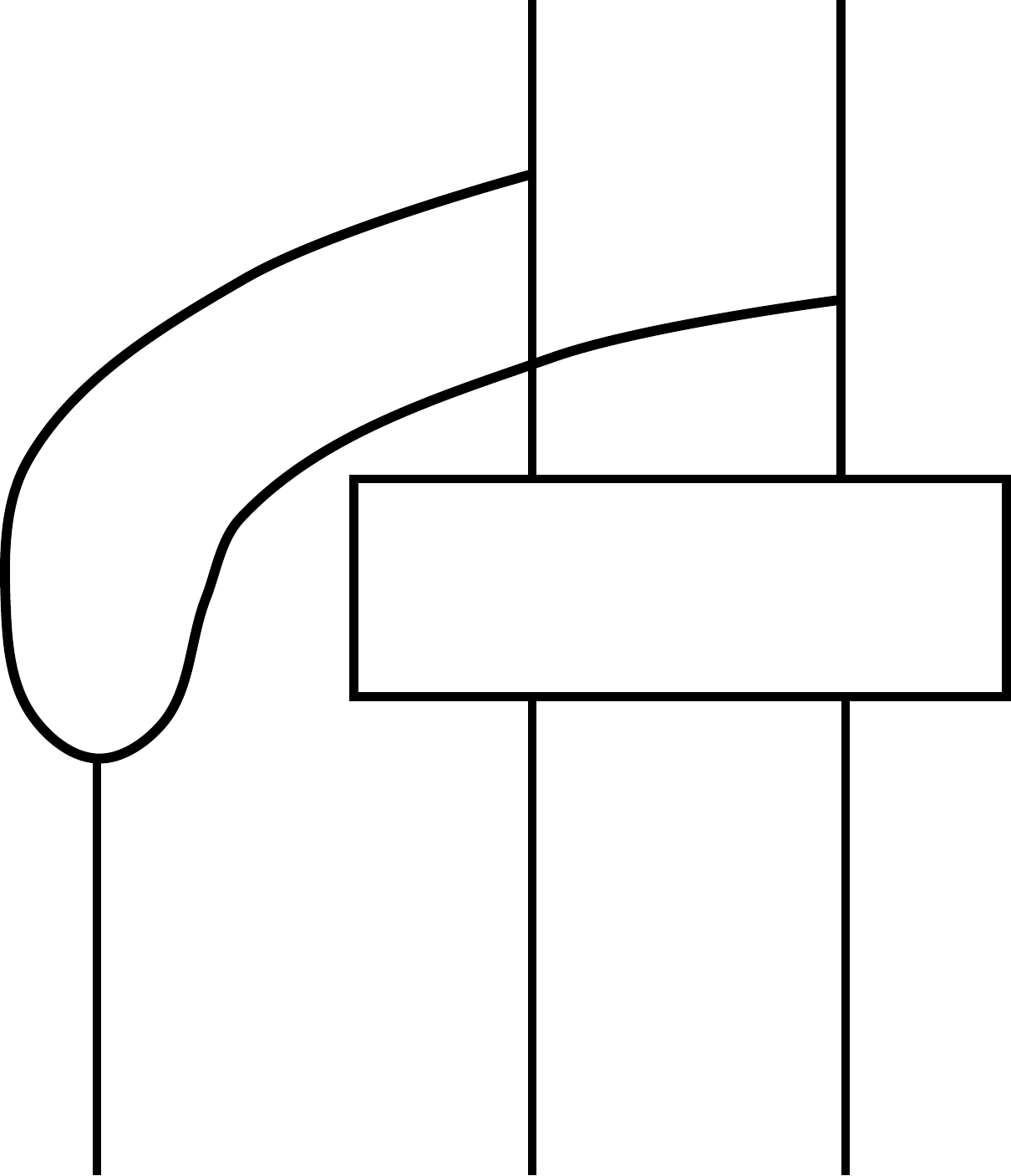}{14ex} \put(-60,-40){\ms{\alpha\beta}} \put(-34,-40){\ms{\alpha\beta}} \put(-14,-40){\ms{_\varepsilon\beta}} \put(-27,2){\ms{\phi_{\alpha,\beta}}} \put(-33,42){\ms{\alpha}} \put(-13,42){\ms{\beta}}
\end{equation*}
where we used the property of the algebra homomorphism $\Delta_{\alpha, \beta}$ in the second equality and the associativity of multiplication in the third equality. The map $\psi_{\alpha, \beta}$ is also $H_{\alpha\beta}$-linear by:
\begin{equation*}
\epsh{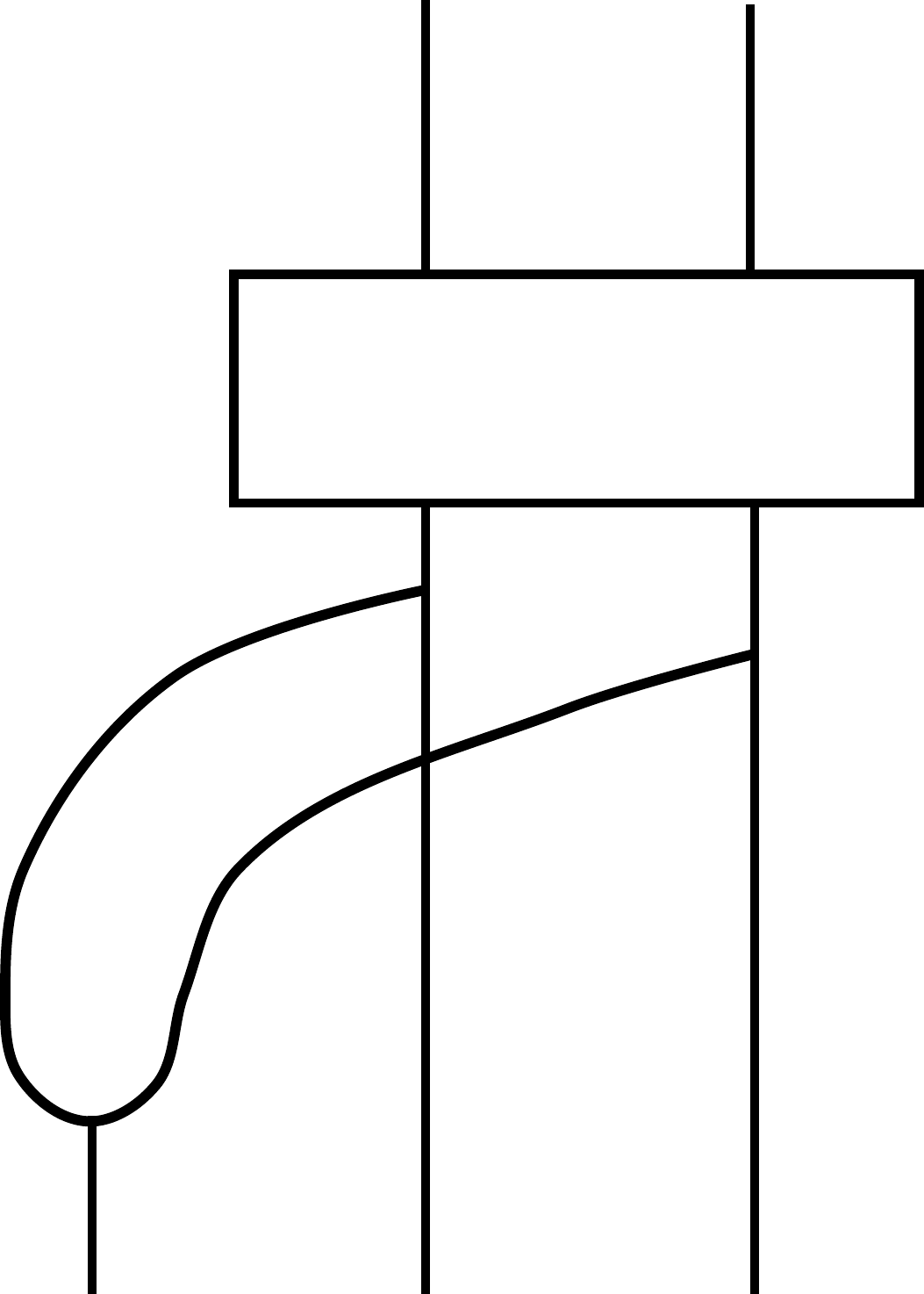}{14ex} \put(-27,17){\ms{\psi_{\alpha,\beta}}} \put(-52,-40){\ms{\alpha\beta}} \put(-30,-40){\ms{\alpha}} \put(-12,-40){\ms{\beta}} \put(-32,42){\ms{\alpha\beta}} \put(-12,42){\ms{_\varepsilon\beta}}
\ =\ \epsh{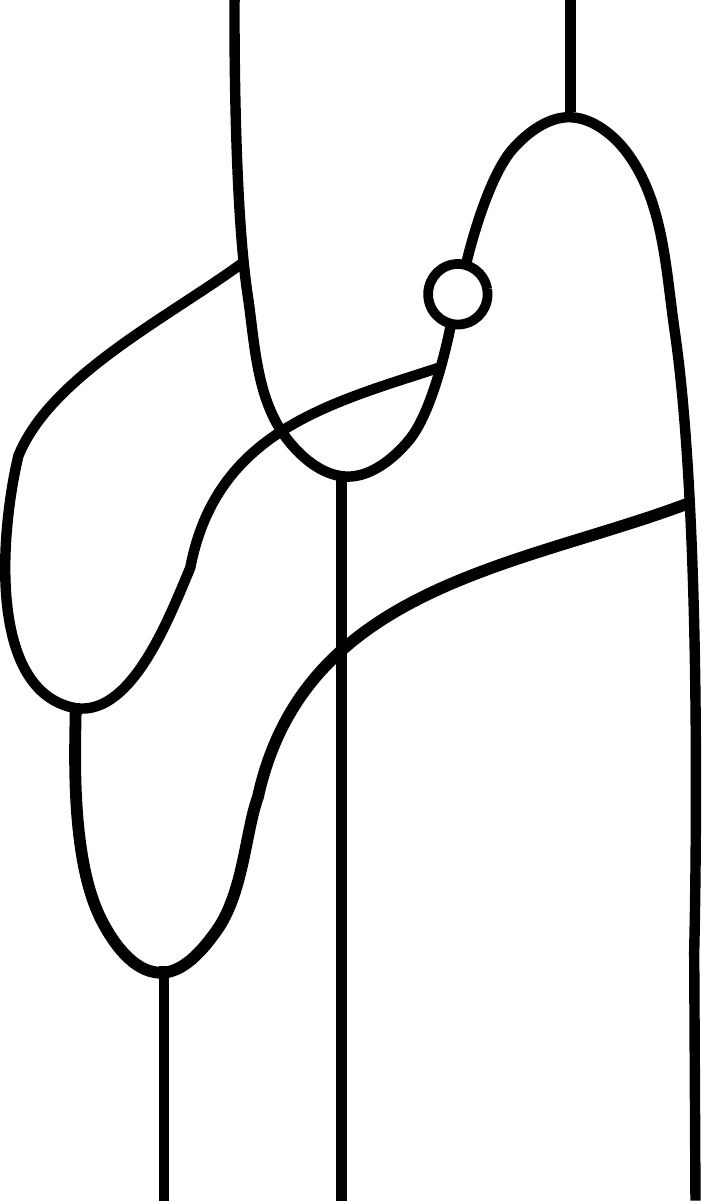}{14ex} \put(-40,-40){\ms{\alpha\beta}} \put(-22,-40){\ms{\alpha}} \put(0,-40){\ms{\beta}} \put(-32,42){\ms{\alpha\beta}} \put(-12,42){\ms{_\varepsilon\beta}}
\ =\ \epsh{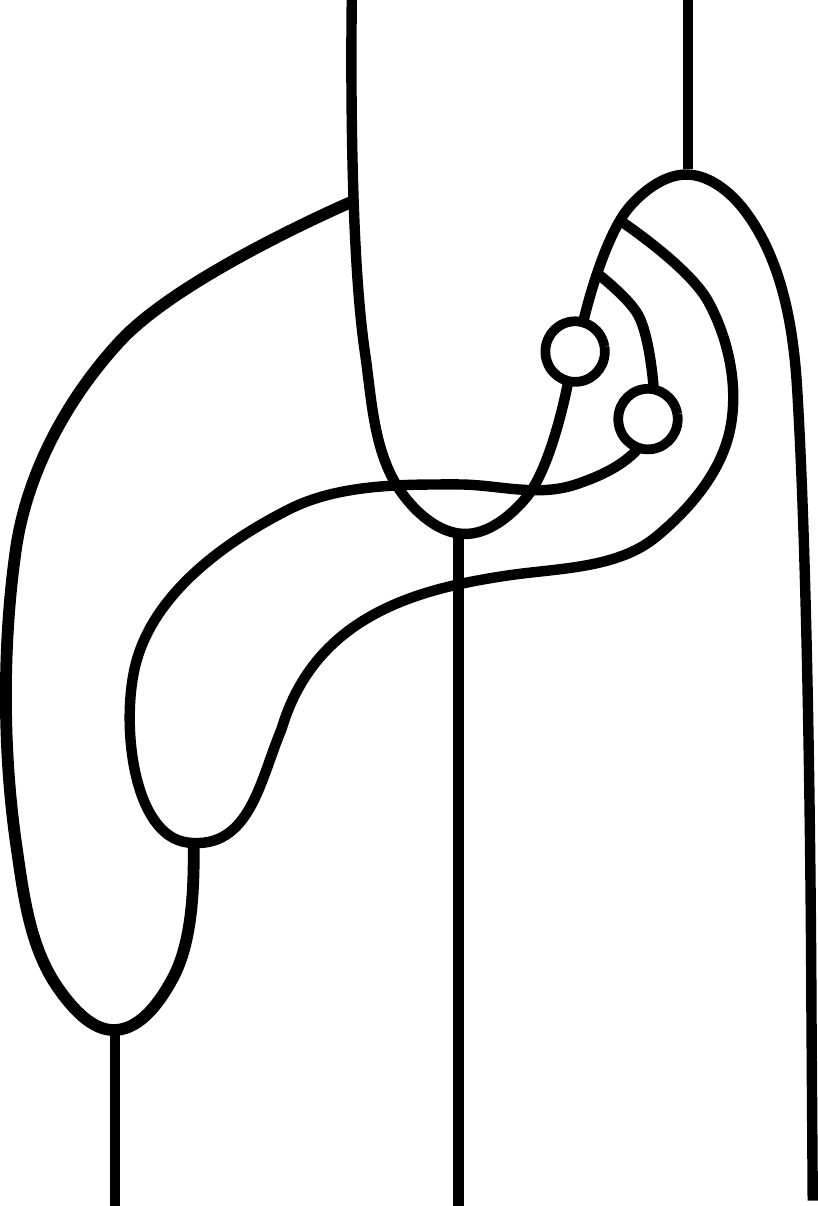}{14ex} \put(-48,-40){\ms{\alpha\beta}} \put(-22,-40){\ms{\alpha}} \put(0,-40){\ms{\beta}} \put(-32,42){\ms{\alpha\beta}} \put(-12,42){\ms{_\varepsilon\beta}}
\ =\ \epsh{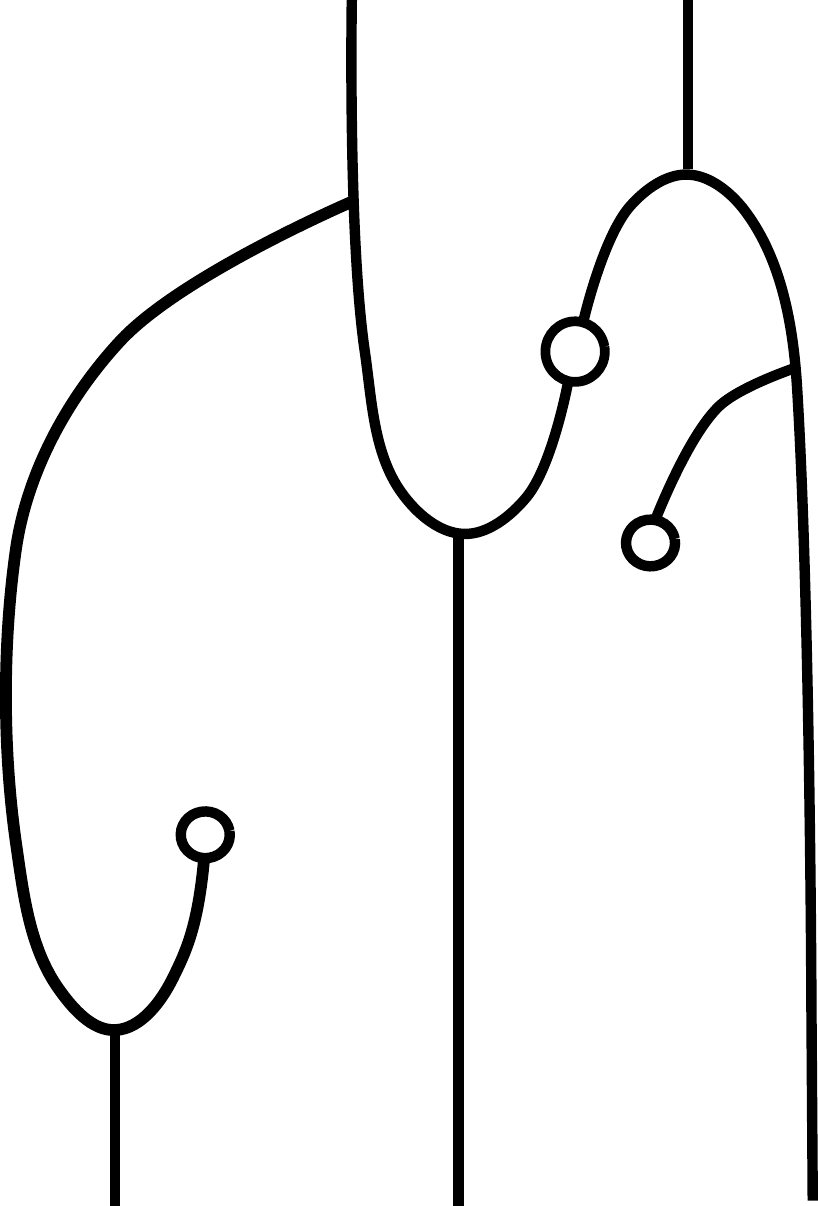}{14ex} \put(-48,-40){\ms{\alpha\beta}} \put(-22,-40){\ms{\alpha}} \put(0,-40){\ms{\beta}} \put(-32,42){\ms{\alpha\beta}} \put(-12,42){\ms{_\varepsilon\beta}}
\ =\ \epsh{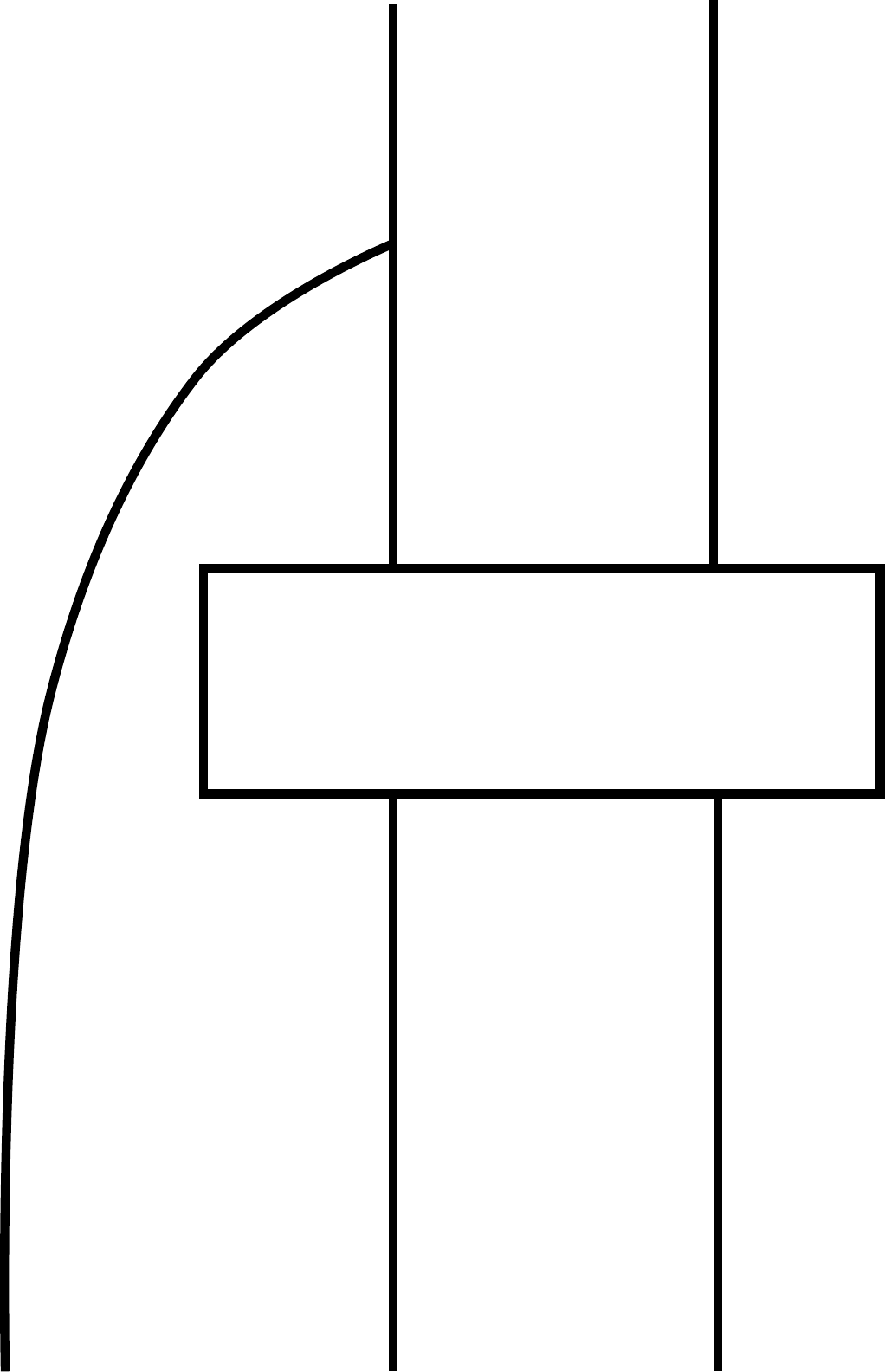}{14ex} \put(-52,-40){\ms{\alpha\beta}} \put(-28,-40){\ms{\alpha}} \put(-10,-40){\ms{\beta}} \put(-32,42){\ms{\alpha\beta}} \put(-12,42){\ms{_\varepsilon\beta}} \put(-27,2){\ms{\psi_{\alpha,\beta}}}
\end{equation*}
where we used the property of the algebra homomorphism $\Delta_{\alpha, \beta}$ in the first equality, the coassociativity of coproduct and the antipode properties are used in the second equality, the associativity of multiplication and the antipode properties are used in the third equality, and we used the antipode properties in the last equality.

The proof of the part (2) is similar way.
\end{proof}
\begin{Pro}\label{rmq to prove main theorem}
Let $H=(H_{\alpha})_{\alpha\in G}$ be a finite type pivotal Hopf $G$-coalgebra. Then we have the equalities of linear maps:
\begin{enumerate}
\item[(1)] $\phi_{\alpha, \beta}(1_{\alpha\beta}\otimes m)=1_\alpha\otimes m$ for $m\in \ _\varepsilon H_\beta$,
\item[(2)] $(\widetilde{\mu}_{\alpha\beta}\otimes \Id'_{\beta})\circ \psi_{\alpha, \beta}=\widetilde{\mu}_{\alpha}\otimes g_{\beta}\Id'_{\beta}$ where $\Id'_{\beta}:\ H_{\beta}\rightarrow\ _\varepsilon H_{\beta} $ is the identity map in $\Vect_{\Bbbk}$.
\end{enumerate}
\end{Pro}
\begin{proof}
The equality (1) holds by the definition of the map $\phi_{\alpha, \beta}$. Part (2) follows from the diagrammic calculus in $\Vect_{\Bbbk}$:
\begin{equation}\label{equi sym right int}
\epsh{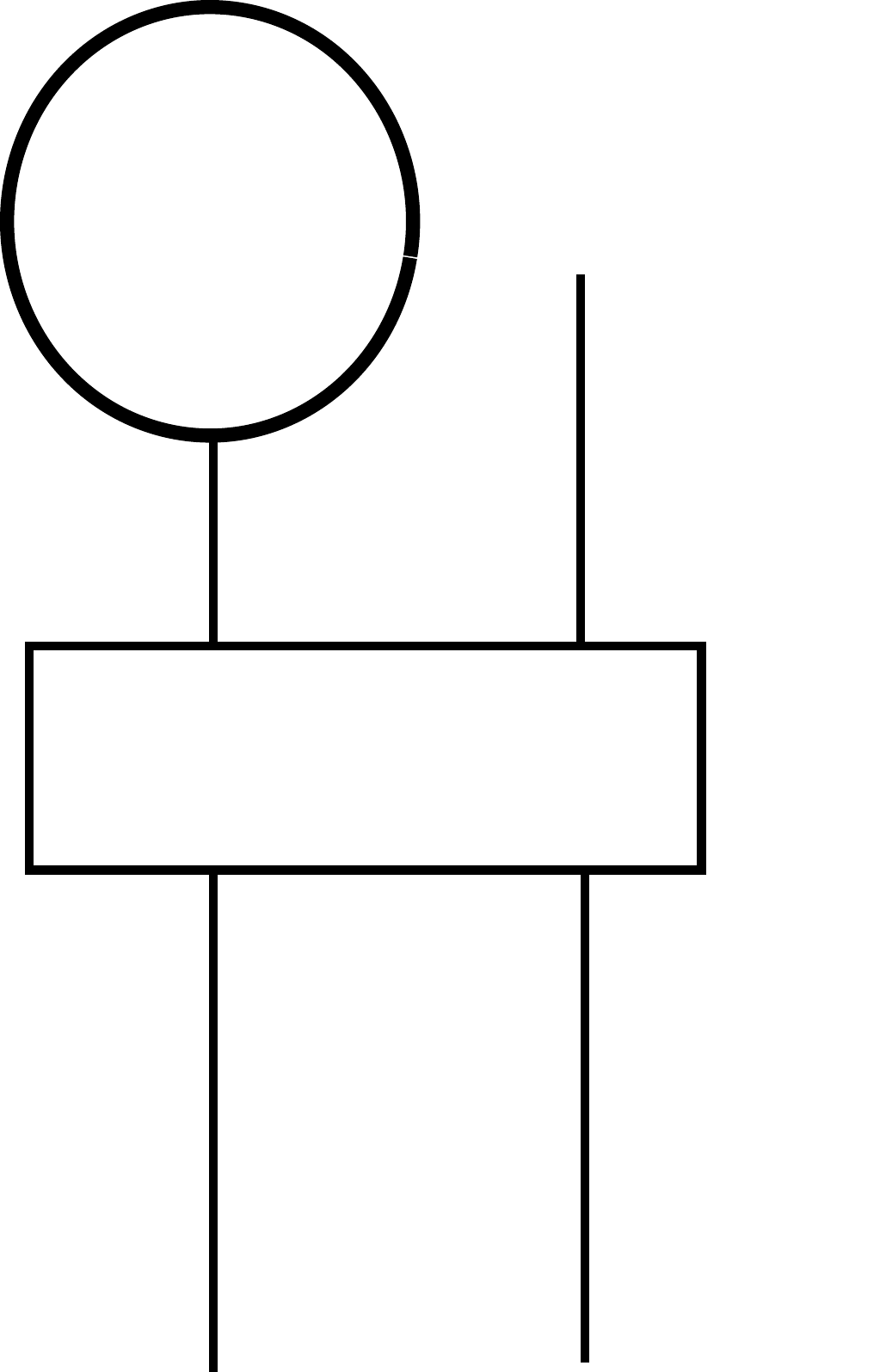}{14ex}\put(-34,-2){\ms{\psi_{\alpha, \beta
}}} \put(-42, 26){\ms{\widetilde{\mu}_{\alpha\beta}}} \put(-38, -40){\ms{\alpha}} \put(-18, -40){\ms{\beta}} \put(-18, 30){\ms{_\varepsilon\beta}}
\ =\ \epsh{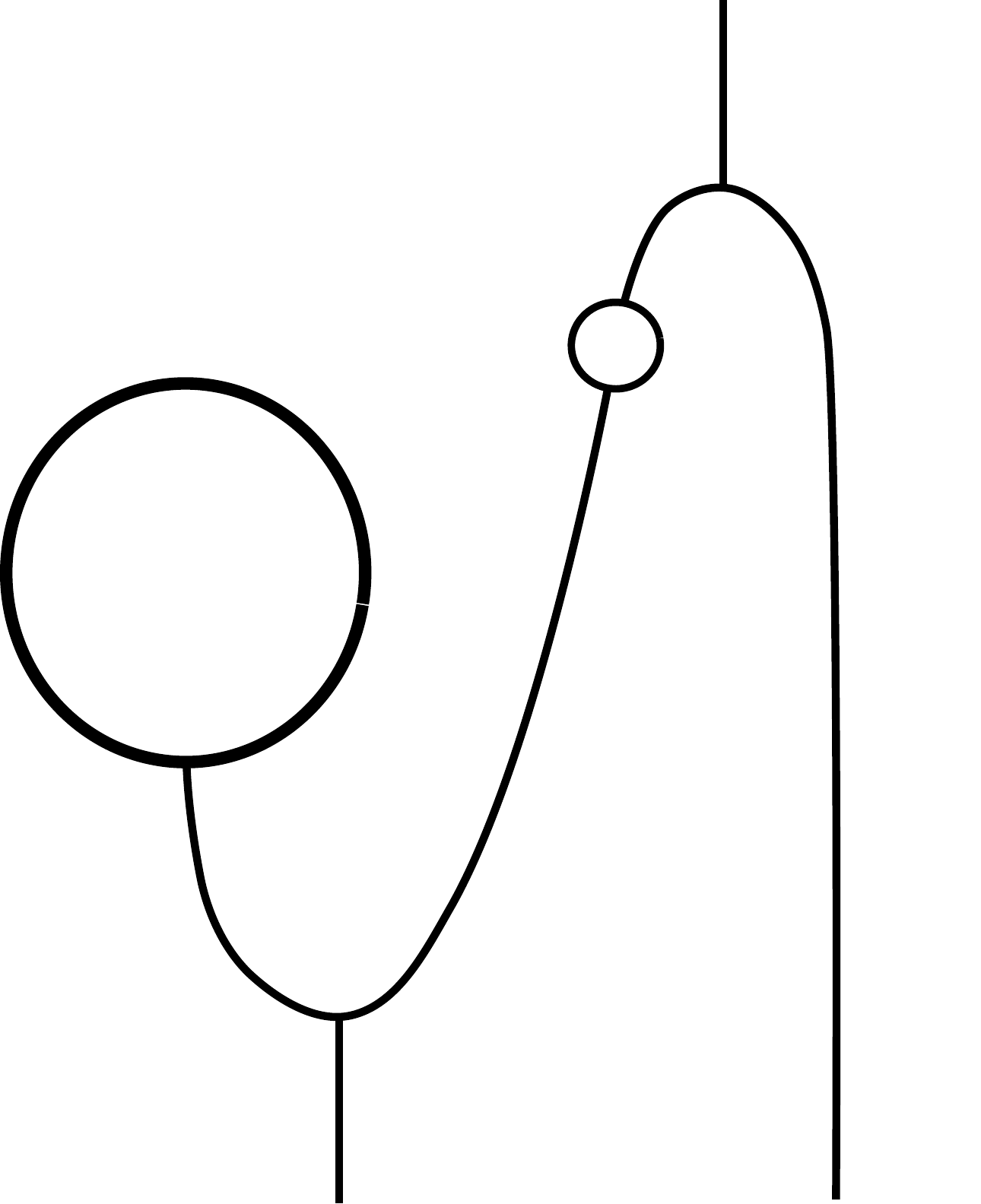}{14ex}\put(-54, 2){\ms{\widetilde{\mu}_{\alpha\beta}}}
\ =\ \epsh{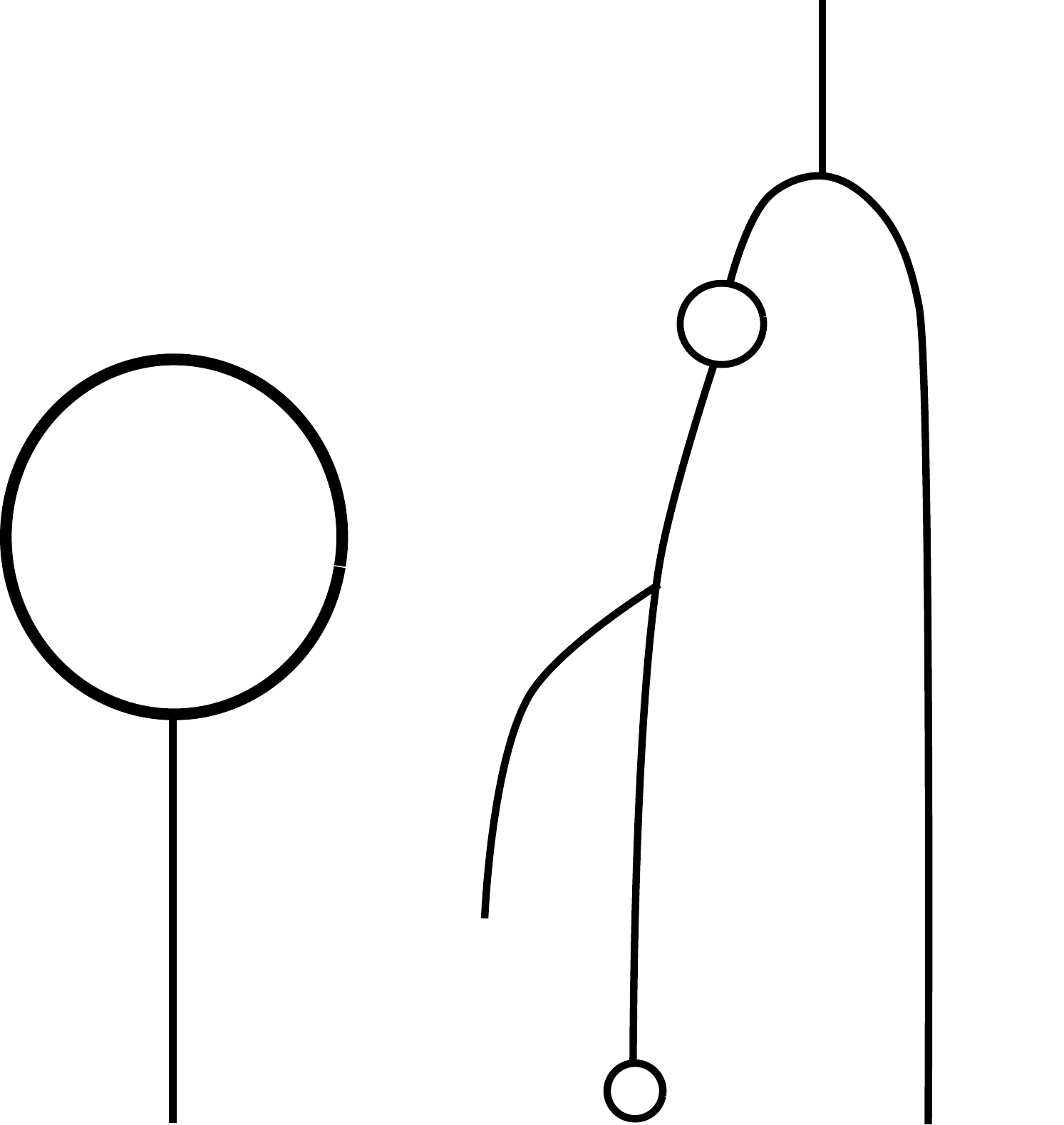}{14ex}\put(-60, 2){\ms{\widetilde{\mu}_{\alpha}}}\put(-48,-29){\ms{g_{\beta^{-1}}^{-1}}}
\ =\ \epsh{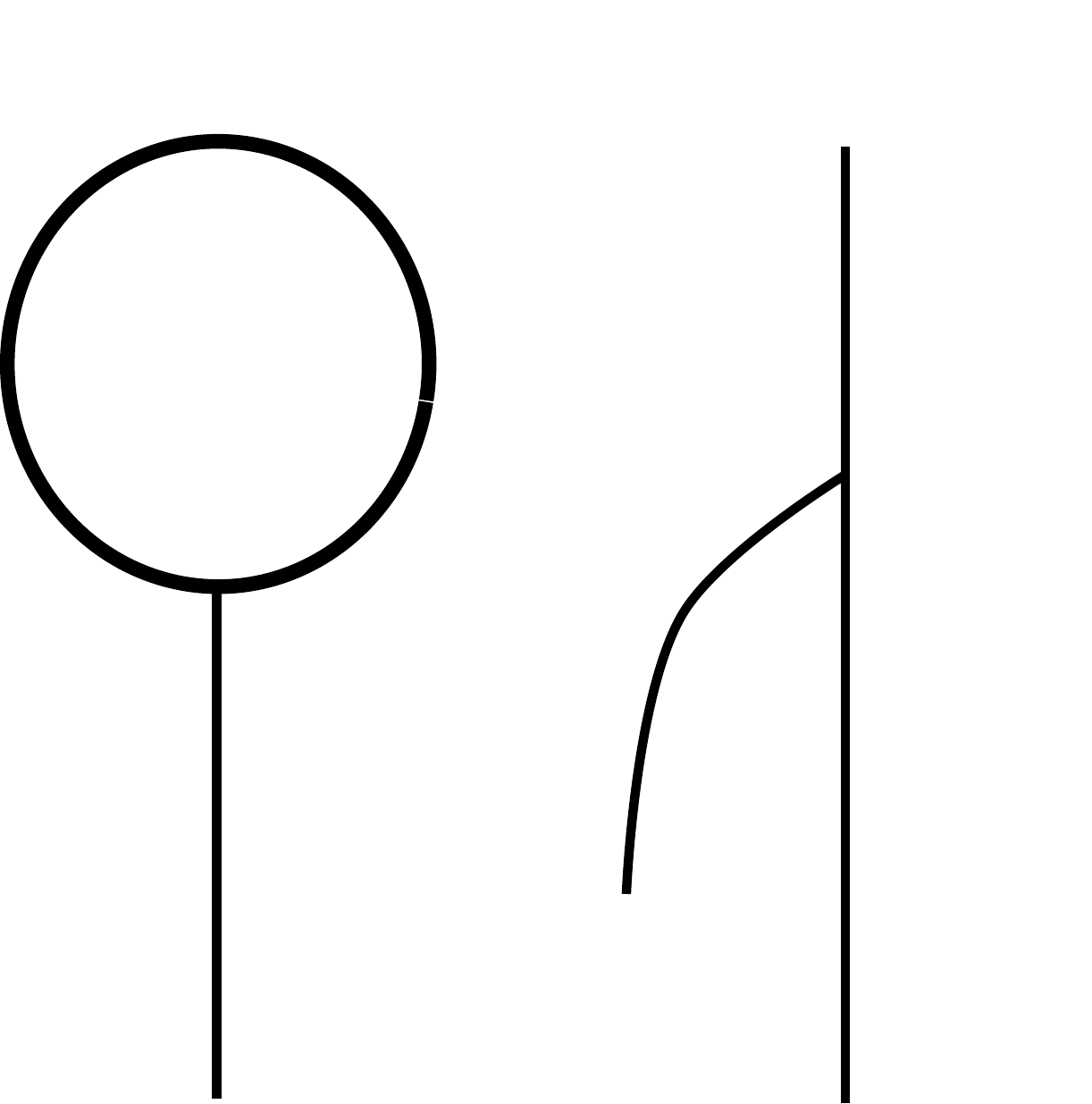}{14ex}\put(-60, 12){\ms{\widetilde{\mu}_{\alpha}}}\put(-36,-27){\ms{g_{\beta}}}
\end{equation}
where in the second equality of \eqref{equi sym right int} we used the relation of the right symmetrised $G$-integral in Figure \ref{presentation of right sym G-int}. 
\end{proof}
\subsection{Proof of Theorem \ref{main theorem}}
Let $(H,g)$ be a finite type unimodular pivotal Hopf $G$-coalgebra, $\mathcal{C}$ be the pivotal $G$-graded category of $H$-modules. The existence of modified trace on $\textbf{Proj}(\mathcal{C})$ follows from: 1) the existence of non-zero integral on $H_1$ 2) the existence of modified trace in $\mathcal{C}_1$ by applying the results of \cite{BBG17} for $H_1$ and 3) the existence of the extension of ambidextrous trace in \cite[Theorem 3.6]{GPV13}. Nevertheless we choose to give a direct proof of this fact following the lines of \cite{BBG17}. Furthermore, Theorem \ref{main theorem} also gives an explicit formula to compute the modified trace $\mtr$ from the integral and conversely.
\begin{proof}[Proof of Theorem \ref{main theorem}]
First, we show that a right symmetrised $G$-integral provides a modified trace. Suppose that $\widetilde{\mu}=(\widetilde{\mu}_{\alpha})_{\alpha\in G}$ is the right symmetrised $G$-integral for $H$.
By Proposition \ref{pro xd cyclic trace tu sym linear} the family of the symmetric forms associated with $\widetilde{\mu}$ induces the family of cyclic traces $\mtr=(\mtr^{\alpha})_{\alpha\in G}$ of $H$-pmod. Here $\mtr^{\alpha}=\{\mtr_P^{\alpha}:\ \End_{H_{\alpha}}(P)\rightarrow \Bbbk\}_{P\in H_{\alpha}\text{-pmod}}$
is determined by 
\begin{equation}\label{fomula determines trace}
\mtr_{H_{\alpha}}^{\alpha}(f)=\widetilde{\mu}_{\alpha}(f(1_\alpha))\ \text{for}\ f\in \End_{H_{\alpha}}(H_{\alpha}).
\end{equation}
To show $\mtr$ is a modified trace, it is enough to check 
\begin{equation}\label{X}
\mtr_{H_{\alpha}\otimes H_{\beta}}^{\alpha\beta}(f)=\mtr_{H_{\alpha}}^{\alpha}(\tr_{H_{\beta}}^{r}(f))\ \text{for any}\ f\in \End_{H_{\alpha\beta}}(H_{\alpha}\otimes H_{\beta})
\end{equation}
thanks to Reduction Lemma \ref{hq chinh}. The value of $\mtr_{H_{\alpha}\otimes H_{\beta}}^{\alpha\beta}(f)$ in Equality \eqref{X} is calculated
{\allowdisplaybreaks
\begin{align*}
\mtr_{H_{\alpha}\otimes H_{\beta}}^{\alpha\beta}\left(\epsh{fig425phi}{14ex}\right)\put(-35,1){\ms{f}} \put(-45, 41){\ms{\alpha}} \put(-45, - 40){\ms{\alpha}} \put(-25, 42){\ms{\beta}} \put(-25, - 41){\ms{\beta}}
\ &=\ \mtr_{H_{\alpha}\otimes H_{\beta}}^{\alpha\beta}\left(\epsh{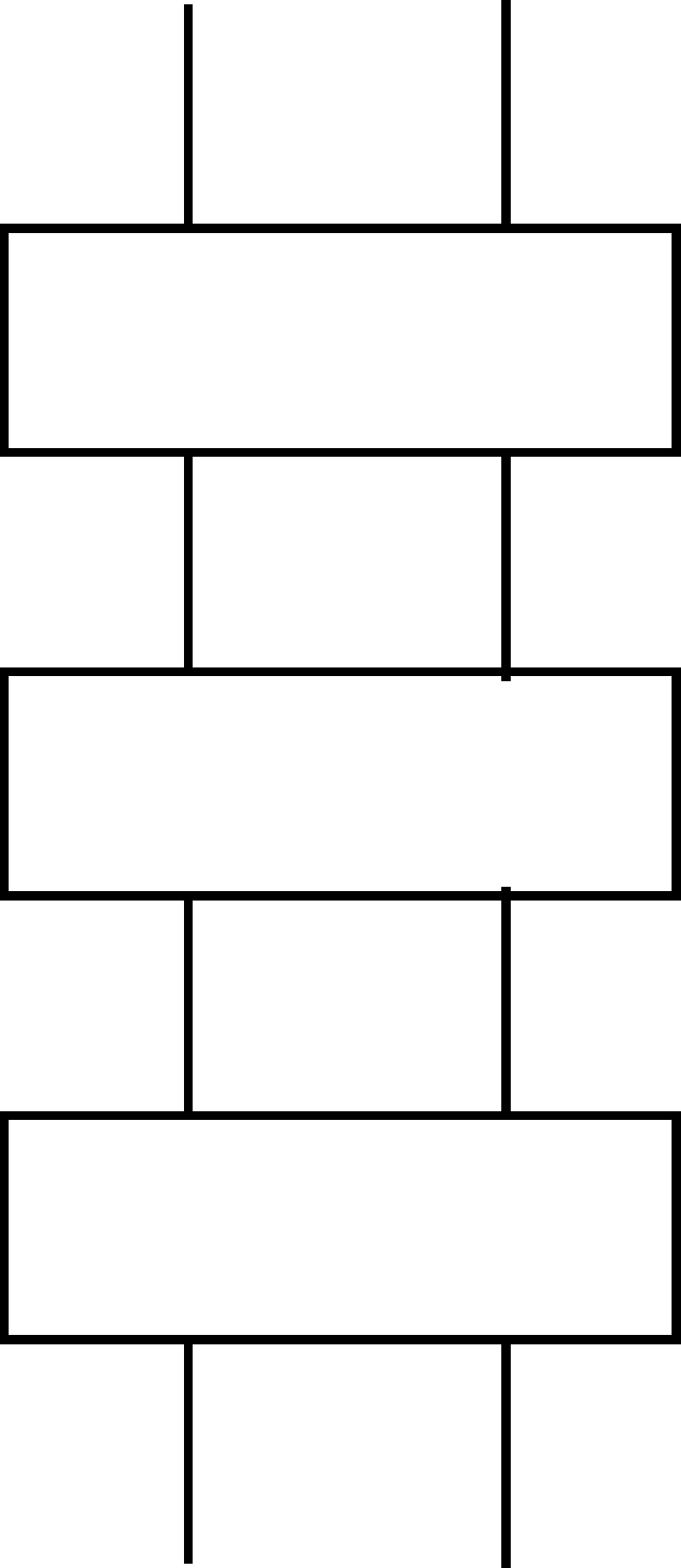}{18ex}\right) \put(-40,2){\ms{\phi_{\alpha,\beta}}} \put(-40,-25){\ms{\psi_{\alpha,\beta}}} \put(-35,27){\ms{f}} \put(-45, 51){\ms{\alpha}} \put(-45, - 50){\ms{\alpha}} \put(-25, 52){\ms{\beta}} \put(-25, - 51){\ms{\beta}}
\ =\ \mtr_{H_{\alpha\beta}\otimes\ _\varepsilon H_{\beta}}^{\alpha\beta}\left(\epsh{fig466}{18ex}\right) \put(-40,-25){\ms{\phi_{\alpha,\beta}}} \put(-40,28){\ms{\psi_{\alpha,\beta}}} \put(-35,1){\ms{f}} 
\put(-48, 52){\ms{\alpha\beta}} \put(-48, - 50){\ms{\alpha\beta}} \put(-25, 52){\ms{_\varepsilon\beta}} \put(-25, - 51){\ms{_\varepsilon\beta}} \\
&=\ \mtr_{H_{\alpha\beta}}^{\alpha\beta}\left(\epsh{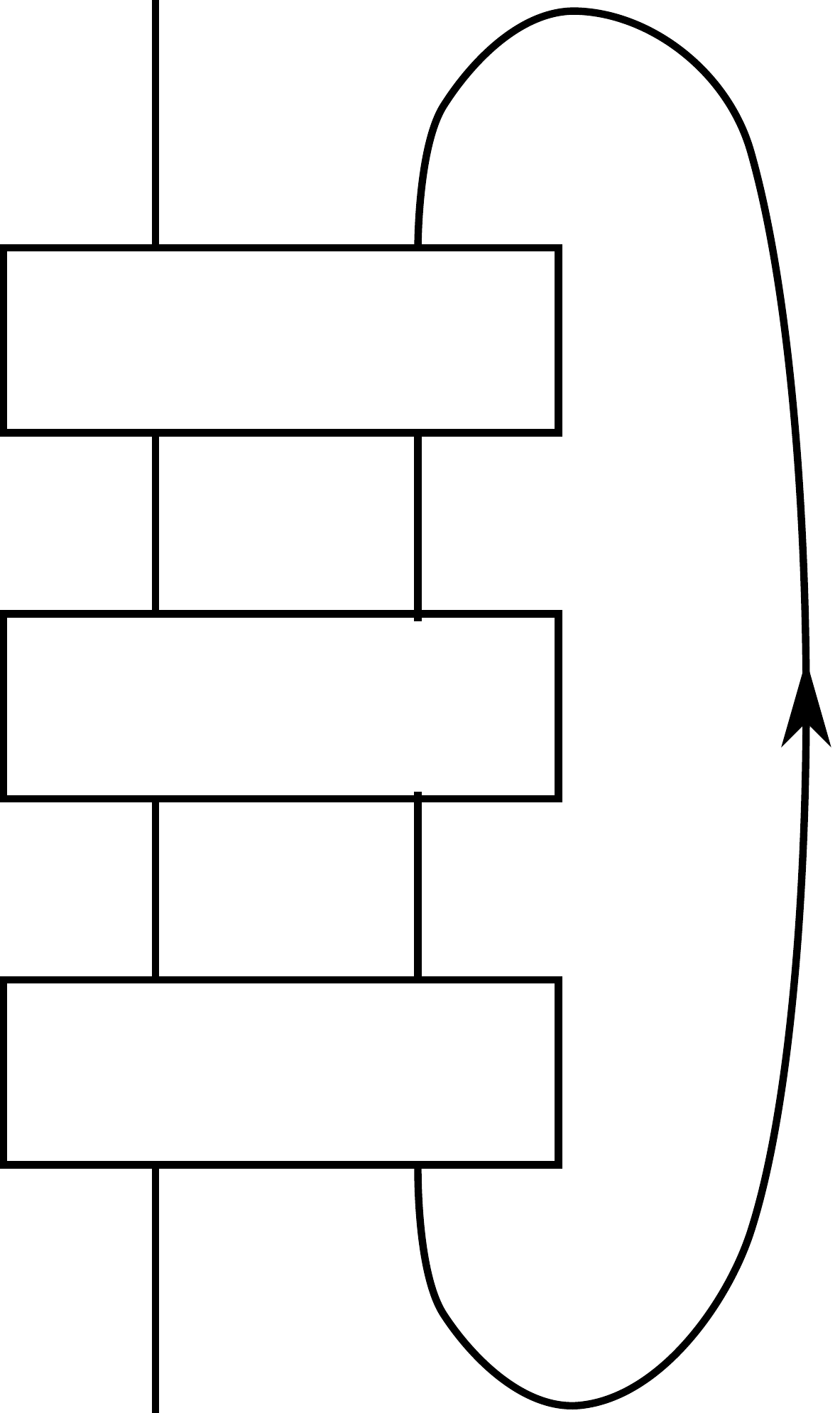}{20ex} \right) \put(-60,-25){\ms{\phi_{\alpha,\beta}}} \put(-60,28){\ms{\psi_{\alpha,\beta}}} \put(-55,1){\ms{f}} \put(-68, 58){\ms{\alpha\beta}} \put(-67, - 56){\ms{\alpha\beta}} \put(-25, - 50){\ms{_\varepsilon\beta}}
\ =\ \epsh{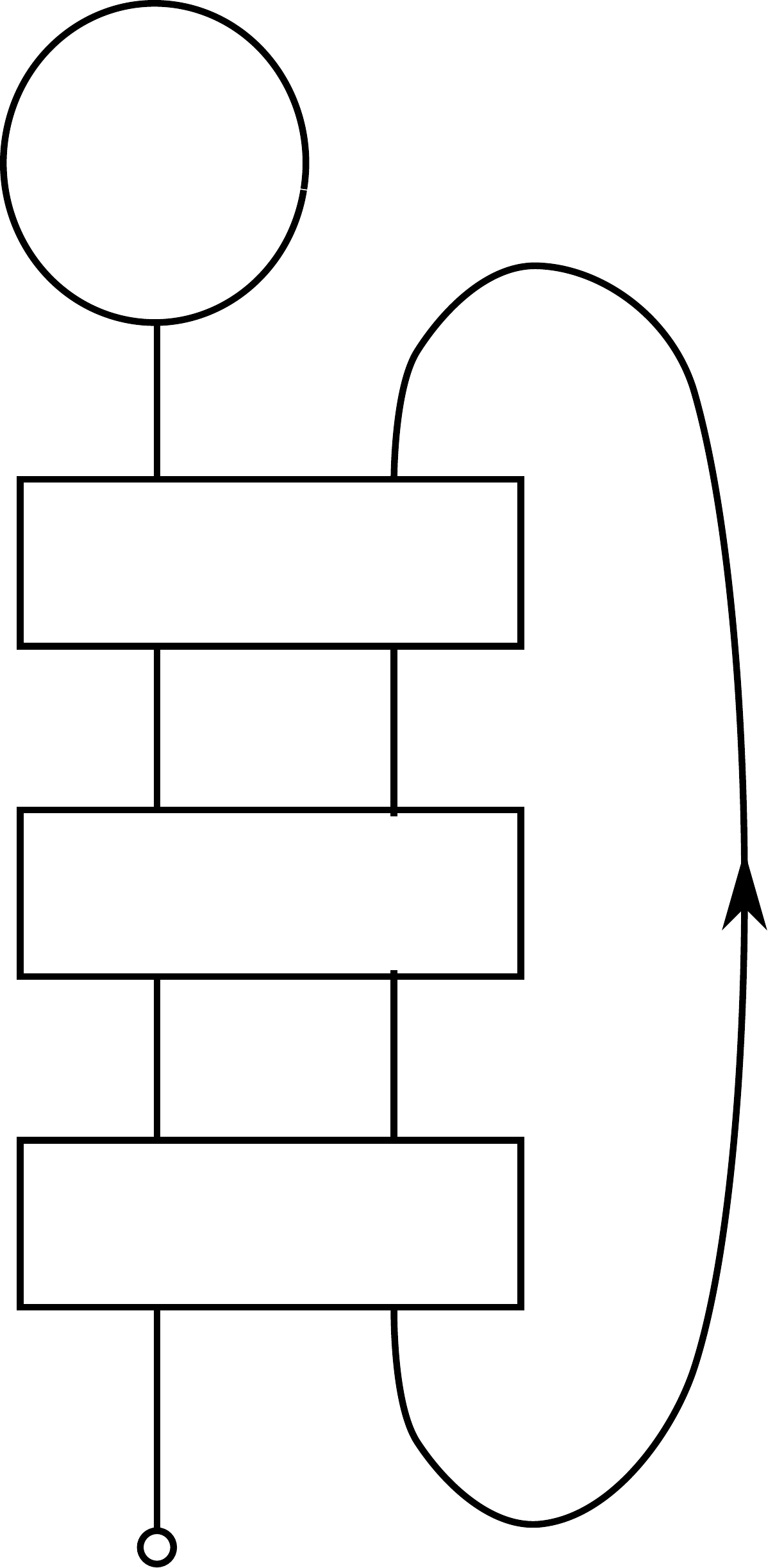}{22ex} \put(-43,-29){\ms{\phi_{\alpha,\beta}}} \put(-45,17){\ms{\psi_{\alpha,\beta}}} \put(-38,-7){\ms{f}} \put(-50, 46){\ms{\widetilde{\mu}_{\alpha\beta}}} \put(-56, -46){\ms{\alpha\beta}} \put(-10, - 50){\ms{_\varepsilon\beta}}\\
&= \ \epsh{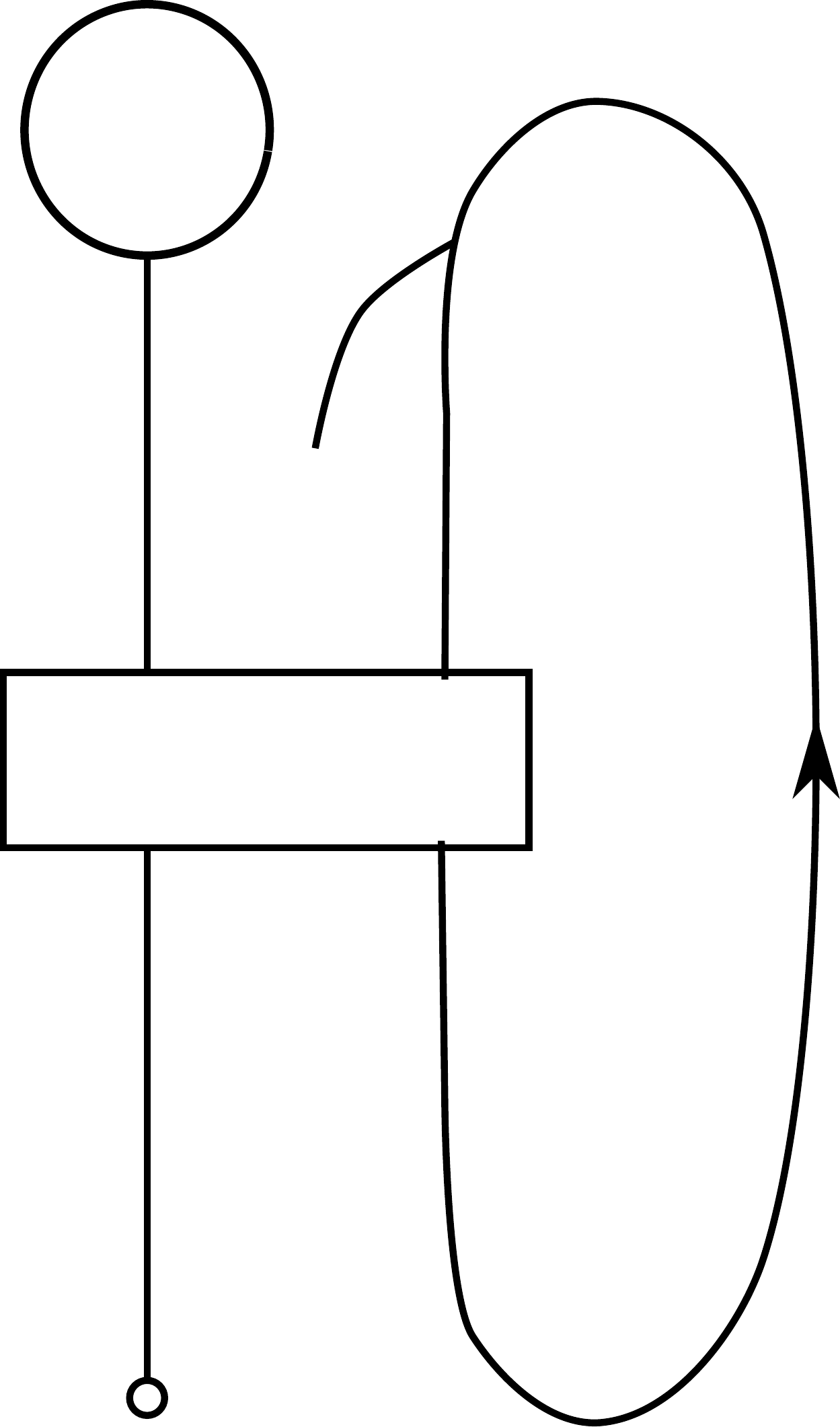}{22ex} \put(-46,-3){\ms{f}} \put(-59, 47){\ms{\widetilde{\mu}_{\alpha}}} \put(-63, -45){\ms{\alpha}} \put(-46,19){\ms{g_\beta}} \put(-6, - 50){\ms{_\varepsilon\beta}}\quad = \quad \mtr_{H_{\alpha}}^{\alpha}(\tr_{H_{\beta}}^{r}(f))\ .
\end{align*}}
In the above calculation, we use Theorem \ref{dl phi xi} in the first equality; the cyclicity property of trace in the second equality; Lemma \ref{lemma partial trace of cyclic trace} in the third equality; Equation \eqref{fomula determines trace} in the fourth equality and in the fifth equality we use the two equalities in Proposition \ref{rmq to prove main theorem}.

Second, assume that we have a right modified trace, and hence the symmetric form $\mtr_{P}^{\alpha}$ on $\End_{H_{\alpha}}(P)$ for any projective module $P$ and any $\alpha\in G$. In particular for any $\alpha, \beta \in G$ the symmetric forms $\mtr_{H_{\alpha}}^{\alpha}$ on $\End_{H_{\alpha}}(H_{\alpha})$ and $\mtr_{H_{\alpha}\otimes H_{\beta}}^{\alpha\beta}$ on $\End_{H_{\alpha\beta}}(H_{\alpha}\otimes H_{\beta})$ satisfy 
\begin{equation}\label{Y}
\mtr_{H_{\alpha}\otimes H_{\beta}}^{\alpha\beta}(f)=\mtr_{H_{\alpha}}^{\alpha}(\tr_{H_{\beta}}^{r}(f))\ \text{for any}\ f\in \End_{H_{\alpha\beta}}(H_{\alpha}\otimes H_{\beta}).
\end{equation}
Let $\widetilde{\nu}_{\alpha}(h)=\mtr_{H_{\alpha}}^{\alpha}(R_h)$ for $R_h\in \End_{H_\alpha}(H_\alpha)$ with $h\in H_{\alpha}$. Then $\widetilde{\nu}_{\alpha}(f(1_\alpha))=\mtr_{H_{\alpha}}^{\alpha}(f)$ for $f\in \End_{H_\alpha}(H_\alpha)$ (see Lemma \ref{ds-pt}). We prove that the family $\widetilde{\nu}=(\widetilde{\nu}_{\alpha})_{\alpha\in G}$ satisfies the relation of the right symmetrised $G$-integral.\\
Consider the maps $k=\Delta_{\alpha, \beta}\circ (R_h\otimes \varphi):\ H_{\alpha\beta} \otimes\ _{\varepsilon}H_{\beta} \rightarrow H_{\alpha}\otimes H_{\beta}$ for $h\in H_{\alpha\beta}$ and $\varphi \in \ _{\varepsilon}H_{\beta}^*$. Then $k$ is a morphism of $H_{\alpha\beta}$-modules. The graphical representation of the map $k$ is given in Figure \ref{bd map k}.
\begin{figure}
$$\epsh{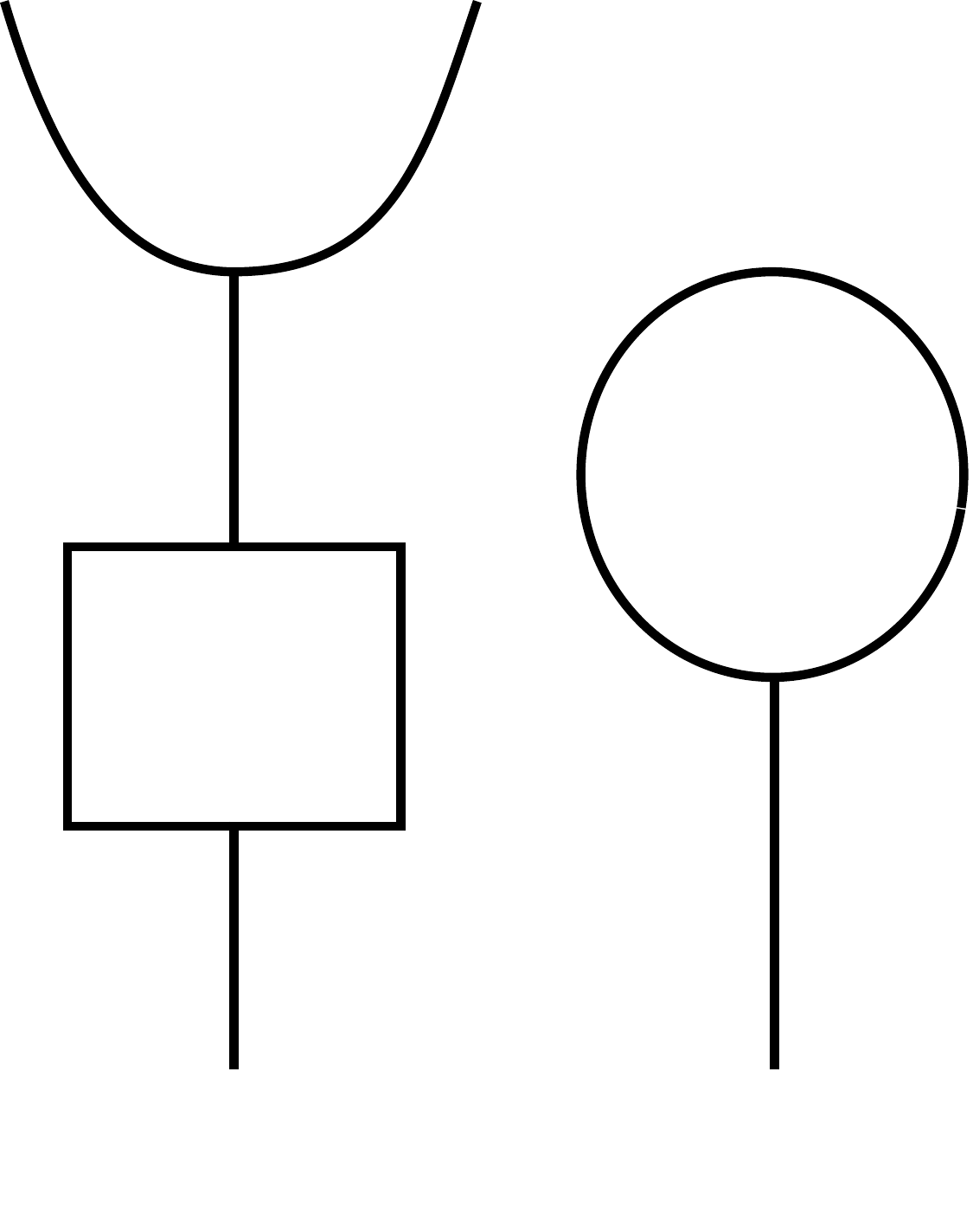}{12ex} \put(-42,-3){\ms{R_h}} \put(-12,8){\ms{\varphi}} \put(-14,-28){\ms{_\varepsilon\beta}} \put(-42,-28){\ms{\alpha\beta}} \put(-52,38){\ms{\alpha}} \put(-27,38){\ms{\beta}}$$
\caption{The graphical representation of the map $k$}\label{bd map k}
\end{figure}
Let $\widetilde{f}=k\circ \psi_{\alpha, \beta}:\ H_{\alpha}\otimes H_{\beta} \rightarrow H_{\alpha}\otimes H_{\beta}$ then $\widetilde{f}\in \End_{H_{\alpha\beta}}(H_{\alpha}\otimes H_{\beta})$. We now calculate the values of the modified trace for $\widetilde{f}\in \End_{H_{\alpha\beta}}(H_{\alpha}\otimes H_{\beta})$ and $\tr_{H_\beta}^{r}(\widetilde{f})\in \End_{H_\alpha}(H_\alpha)$.\\
Firstly, we have
{\allowdisplaybreaks
\begin{align*}
\mtr_{H_{\alpha}\otimes H_{\beta}}^{\alpha\beta}(\widetilde{f})&=\mtr_{H_{\alpha}\otimes H_{\beta}}^{\alpha\beta}(k\circ \psi_{\alpha, \beta})=\mtr_{H_{\alpha\beta}\otimes\ _\varepsilon H_{\beta}}^{\alpha\beta}(\psi_{\alpha, \beta}\circ k)\\
&=\mtr_{H_{\alpha\beta}\otimes\ _\varepsilon H_{\beta}}^{\alpha\beta}\left(\epsh{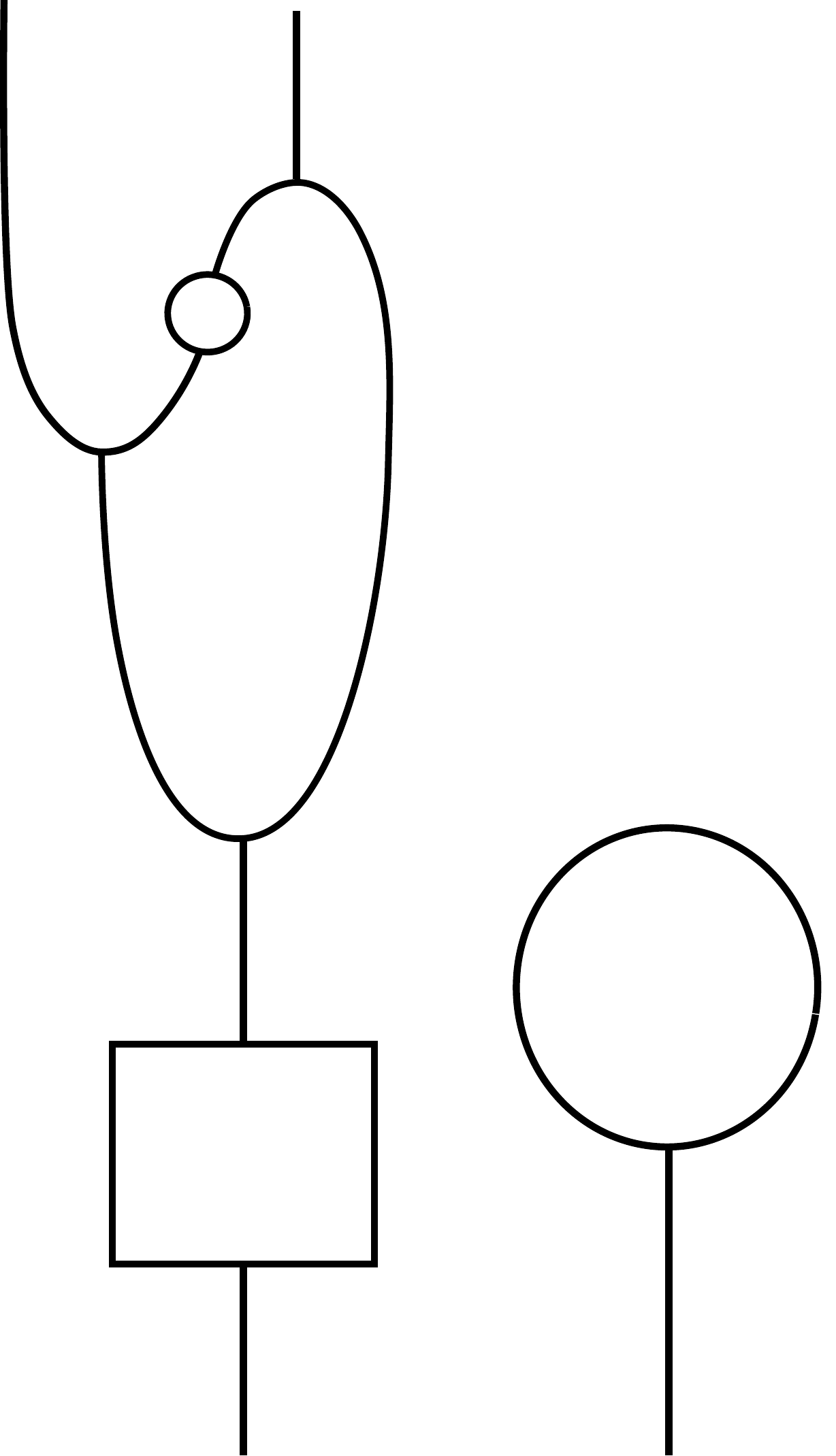}{20ex}\right) \put(-59,-30){\ms{R_h}} \put(-59,-56){\ms{\alpha\beta}} \put(-28,-56){\ms{_\varepsilon\beta}} \put(-26,-17){\ms{\varphi}} \put(-76,58){\ms{\alpha\beta}} \put(-56,58){\ms{_\varepsilon\beta}}
=\mtr_{H_{\alpha\beta}\otimes\ _\varepsilon H_{\beta}}^{\alpha\beta}\left(\epsh{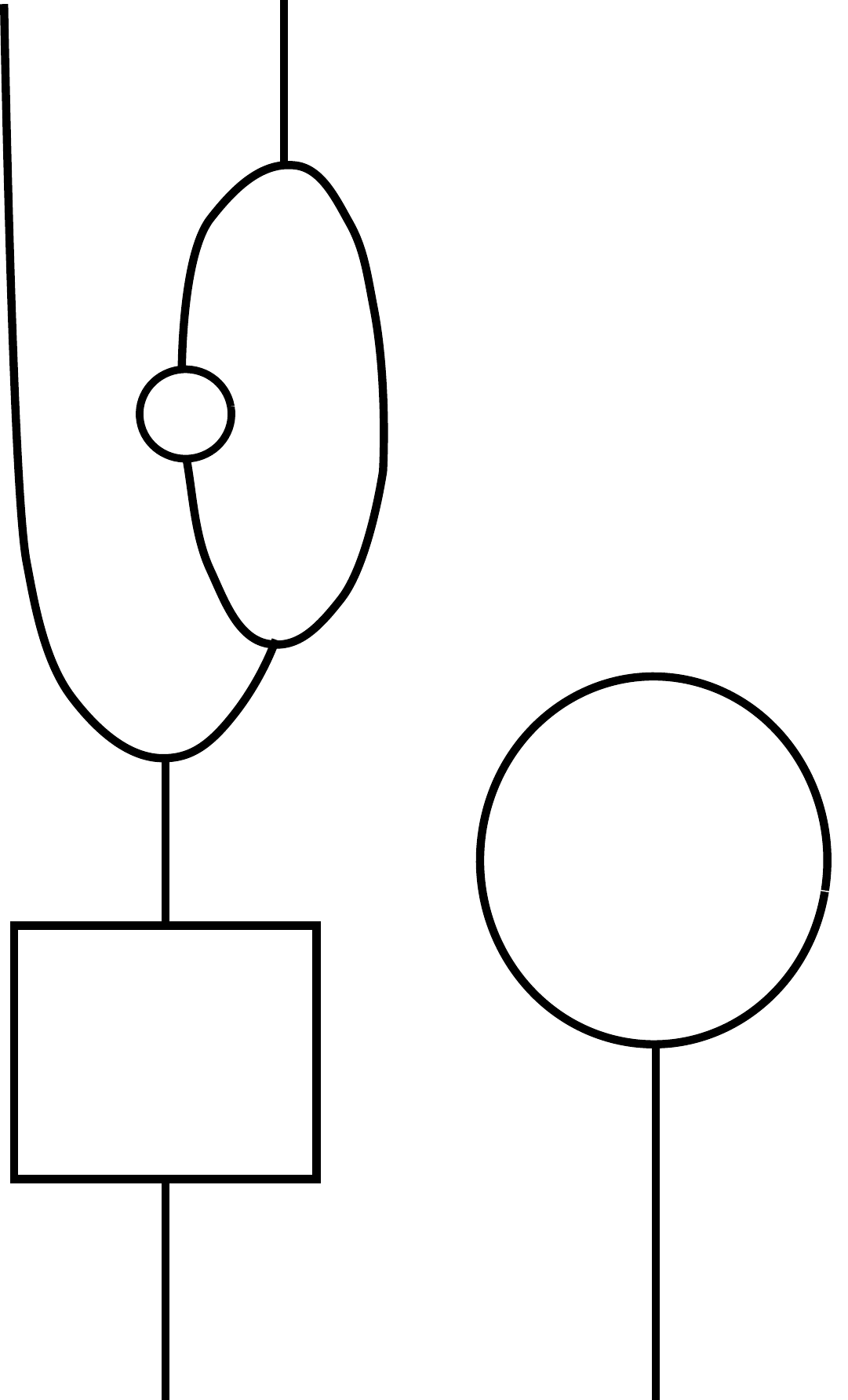}{20ex}\right) \put(-67,-26){\ms{R_h}} \put(-28,-11){\ms{\varphi}}    \\
&=\mtr_{H_{\alpha\beta}\otimes\ _\varepsilon H_{\beta}}^{\alpha\beta}\left(\epsh{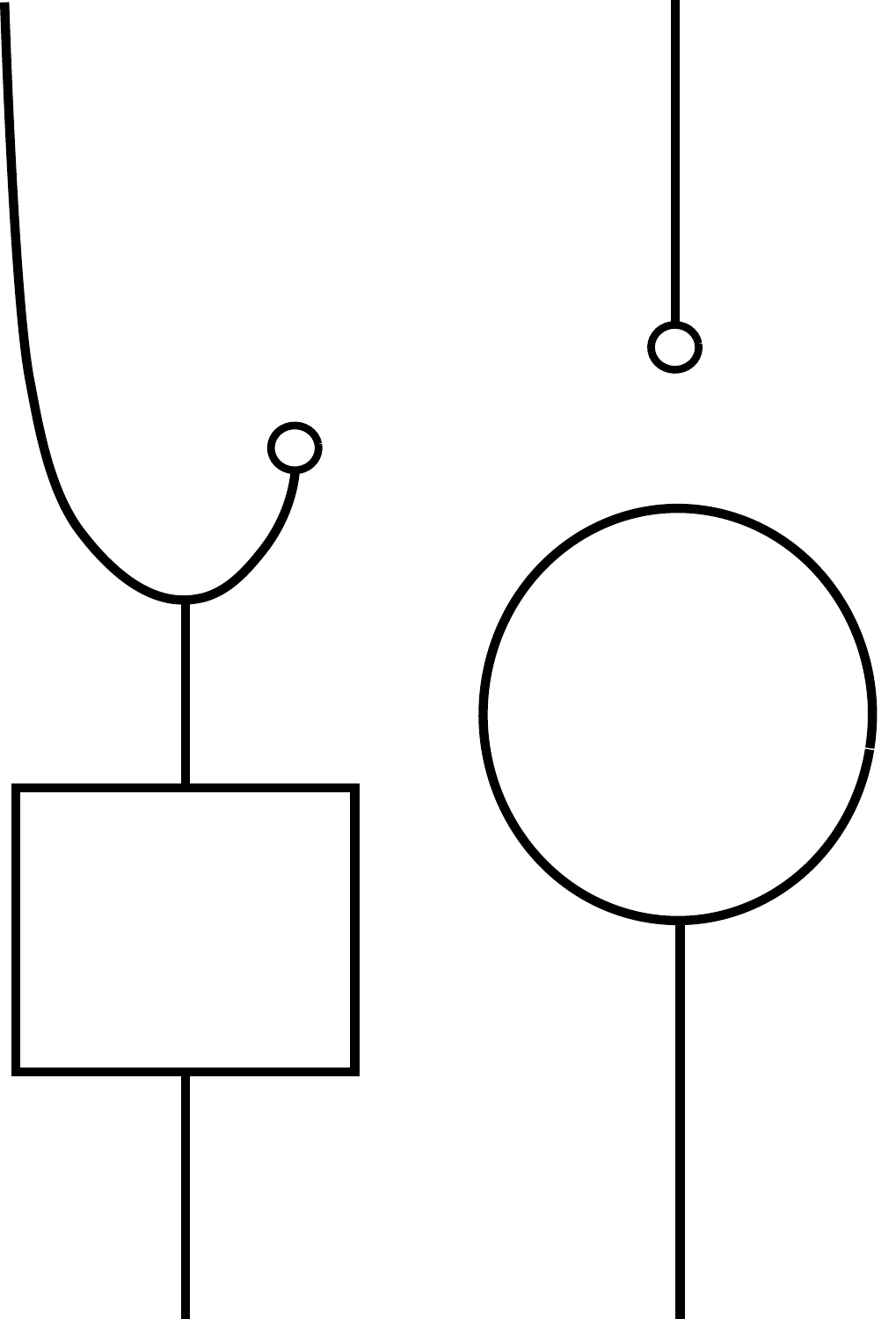}{16ex}\right) \put(-61,-16){\ms{R_h}} \put(-28,-3){\ms{\varphi}}
=\mtr_{H_{\alpha\beta}\otimes\ _\varepsilon H_{\beta}}^{\alpha\beta}\left(\epsh{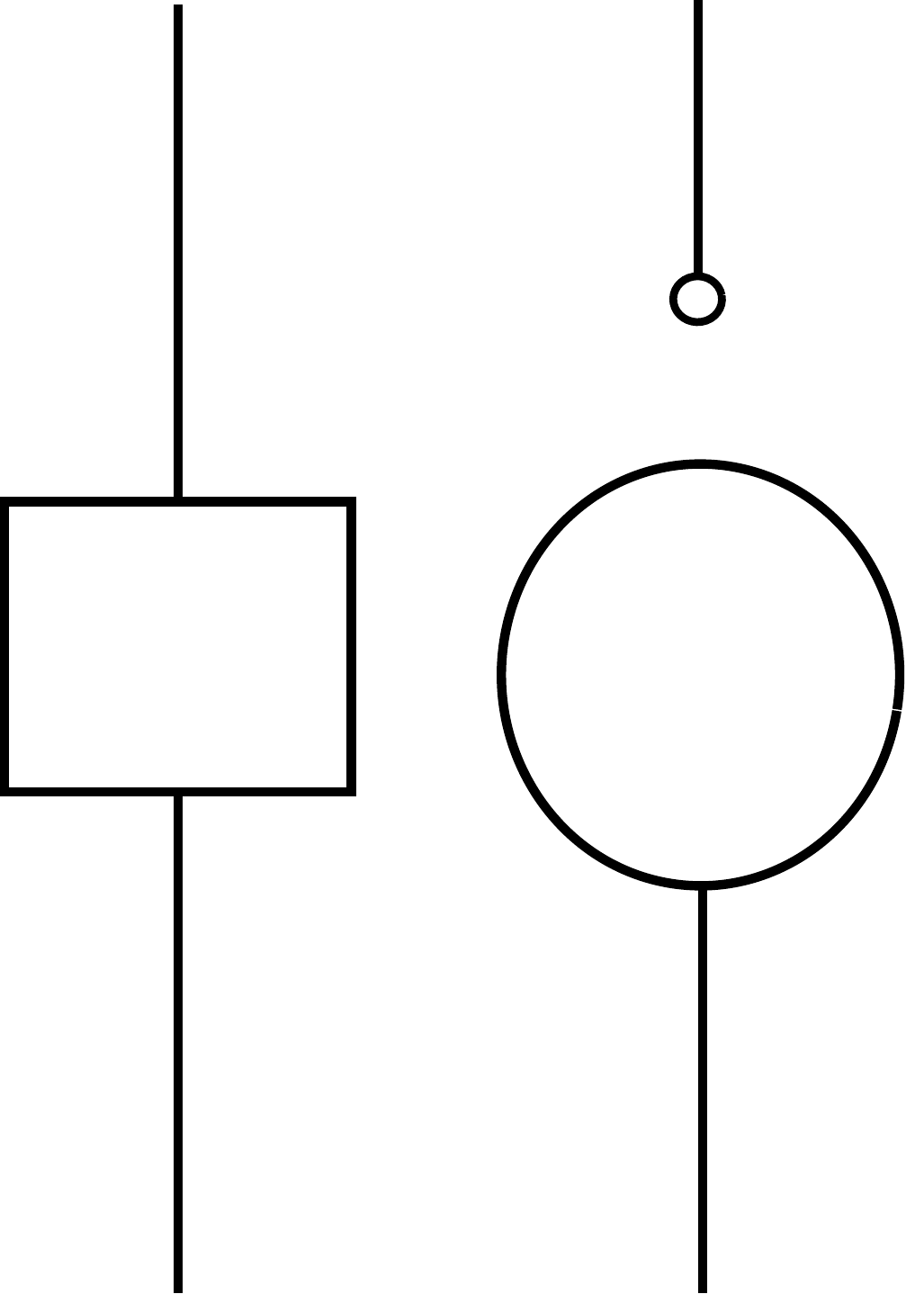}{16ex}\right)\put(-64,0){\ms{R_h}} \put(-28,-1){\ms{\varphi}}   \\
&=\mtr_{H_{\alpha\beta}}^{\alpha\beta}\left(\epsh{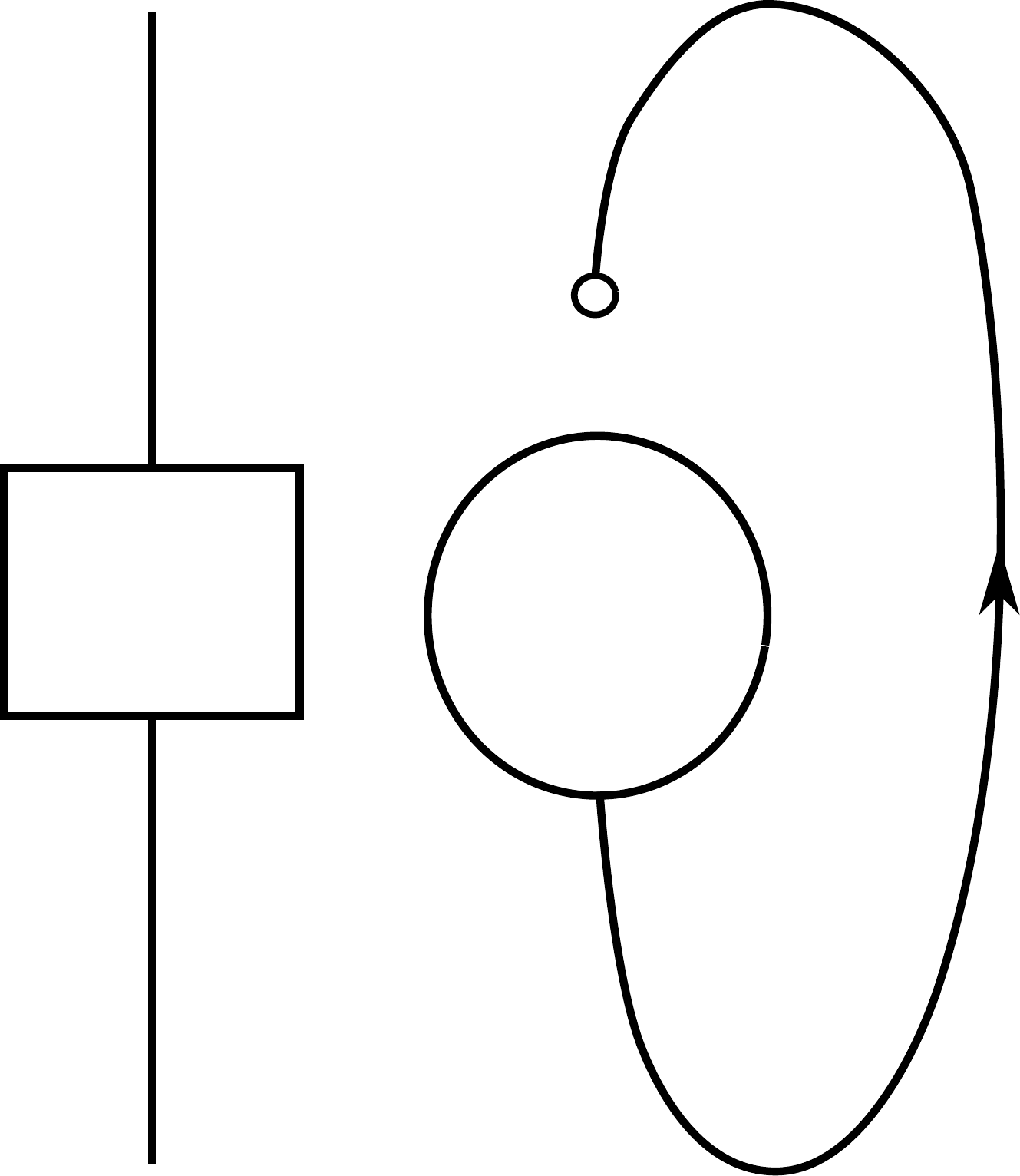}{16ex}\right) \put(-77,0){\ms{R_h}} \put(-43,-1){\ms{\varphi}}
=\widetilde{\nu}_{\alpha\beta}(h)\varphi(1_\beta).
\end{align*}}
In the above calculations, we use the cyclicity property in the second equality; the coassociativity of the coproduct in the fourth equality; the antipode properties in the fifth equality and finally we use the partial trace property.\\
Secondly, we have
{\allowdisplaybreaks
\begin{align*}
\mtr_{H_\alpha}^{\alpha}(\tr_{H_{\beta}}^{r}(\widetilde{f}))&=\mtr_{H_\alpha}^{\alpha}(\tr_{H_{\beta}}^{r}(k\circ\psi_{\alpha, \beta}))\\
&=\mtr_{H_\alpha}^{\alpha}\left(\epsh{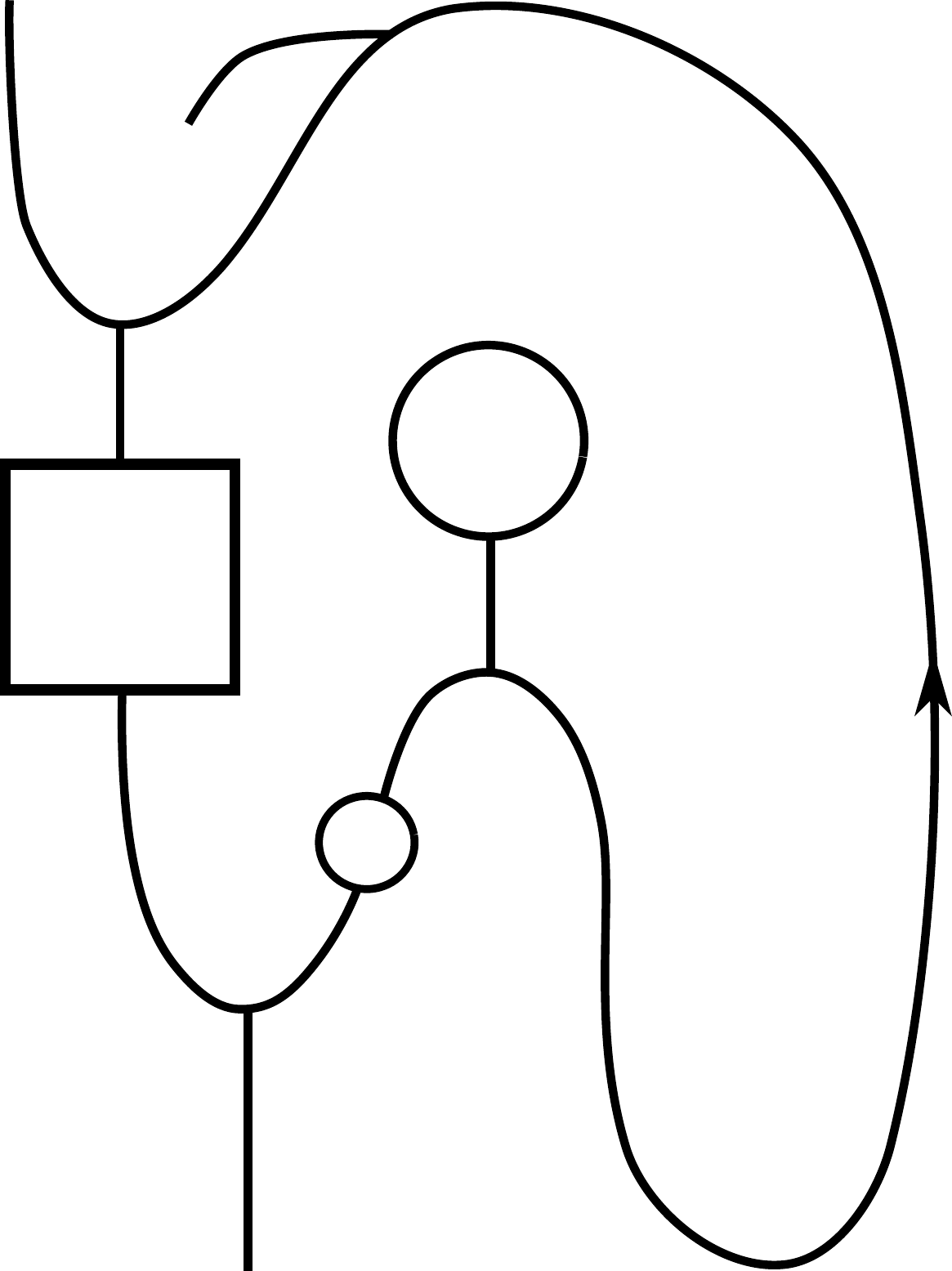}{20ex} \right) \put(-85,6){\ms{R_h}} \put(-52,17){\ms{\varphi}} \put(-72,-56){\ms{\alpha}} \put(-32,-56){\ms{\beta}} \put(-92, 57){\ms{\alpha}} \put(-80, 39){\ms{g_\beta}}
 =\epsh{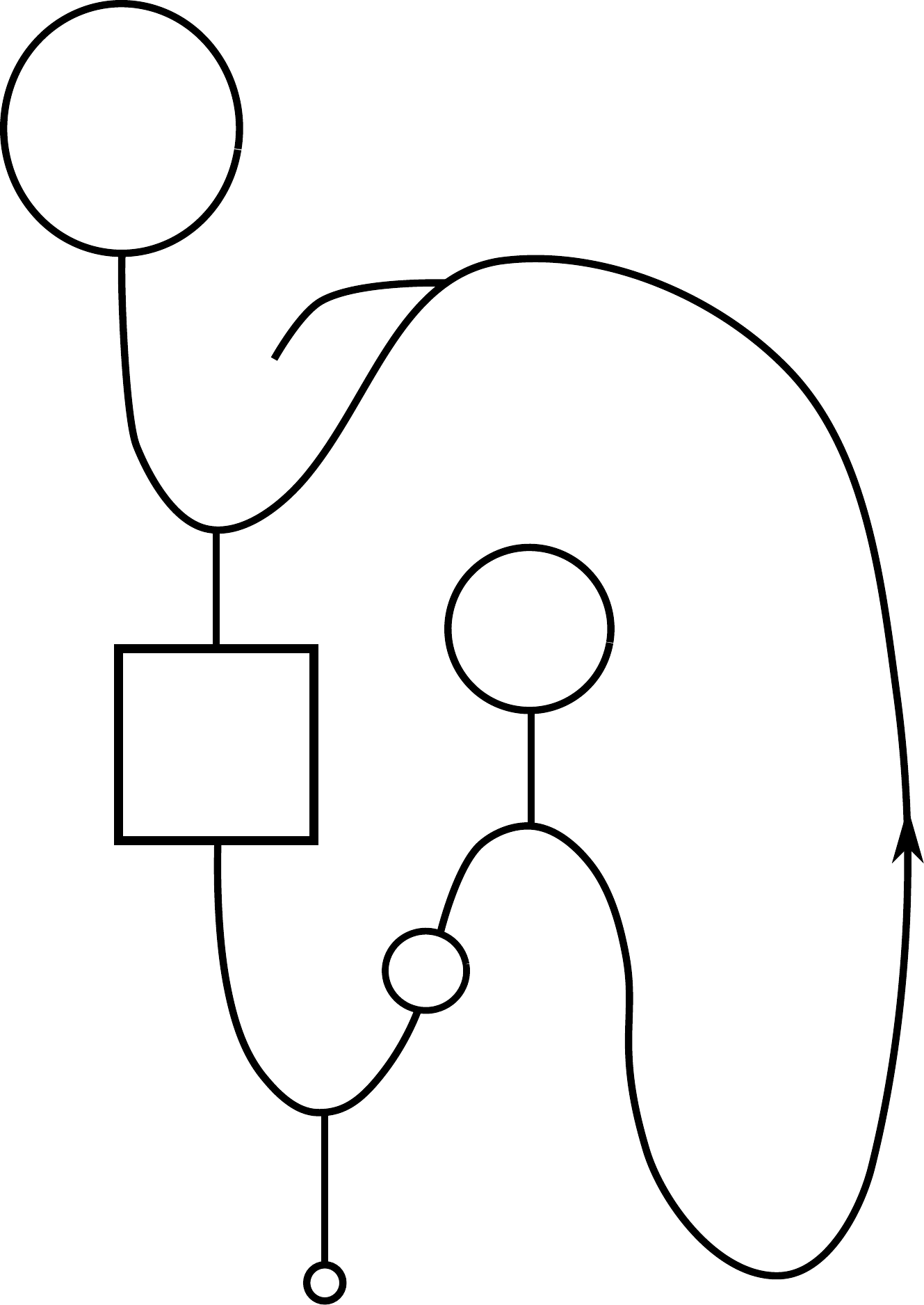}{20ex} \put(-61,-5){\ms{R_h}} \put(-33,4){\ms{\varphi}} \put(-66,41){\ms{\widetilde{\nu}_{\alpha}}} \put(-56, 22){\ms{g_\beta}} \put(-55,-43){\ms{\alpha}} \\
&=\epsh{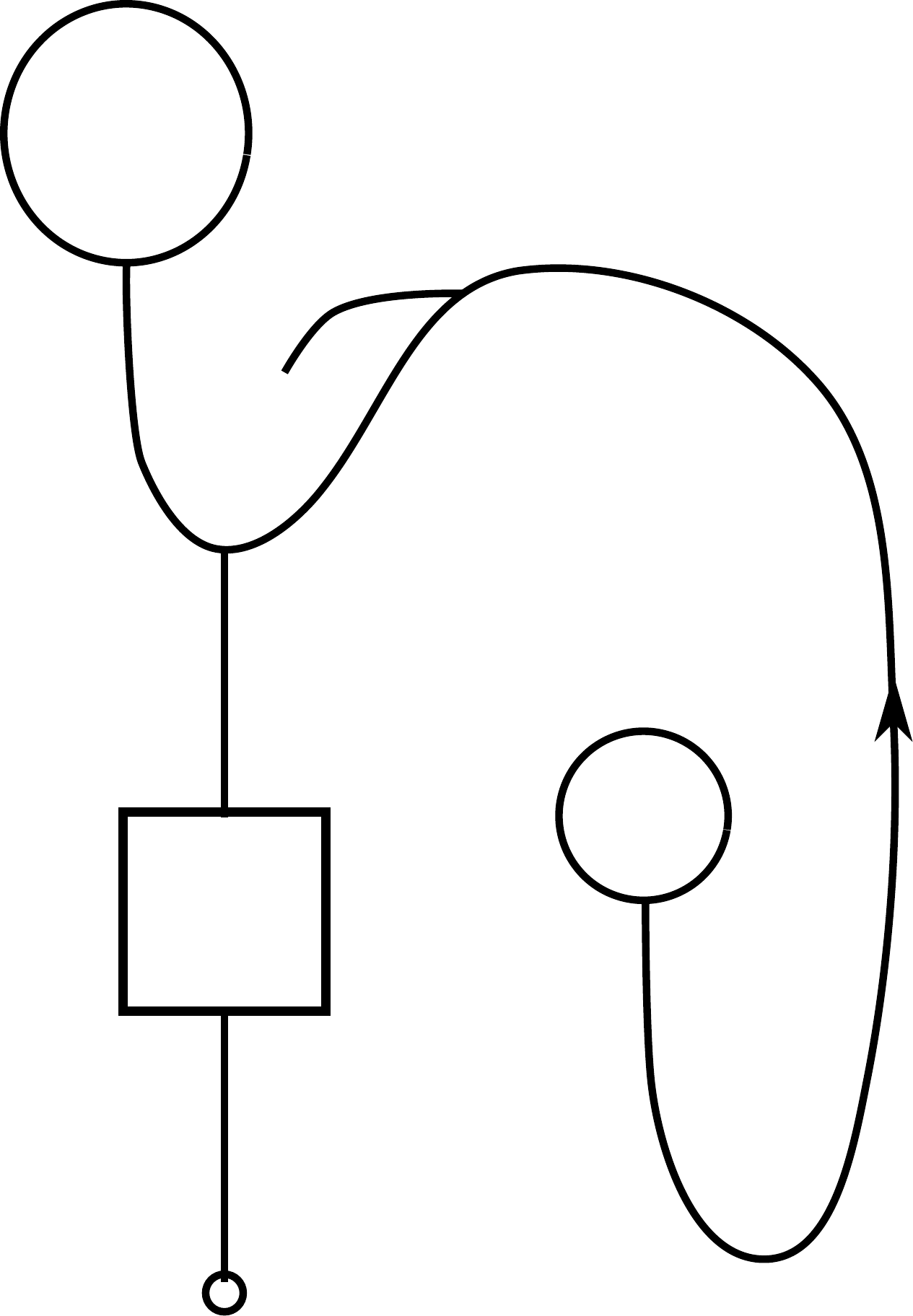}{20ex} \put(-65,42){\ms{\widetilde{\nu}_{\alpha}}} \put(-56, 20){\ms{g_\beta}} \put(-23,-10){\ms{\varphi}} \put(-67,-37){\ms{\alpha\beta}} \put(-59, -19){\ms{R_h}}
=\widetilde{\nu}_{\alpha}(h_{(1)})\varphi(g_{\beta}h_{(2)})
\end{align*}}
where we use the left evaluation $\tev$ with the pivot $g_\beta$ and the right coevaluation $\coev$ in the second equality and $\Delta_{\alpha, \beta}(h)=h_{(1)}\otimes h_{(2)}$.\\
By Equality \eqref{Y} one has $\mtr_{H_{\alpha}\otimes H_{\beta}}^{\alpha\beta}(\widetilde{f})=\mtr_{H_\alpha}^{\alpha}(\tr_{H_{\beta}}^{r}(\widetilde{f}))$. This equality means that
\begin{equation*}
\widetilde{\nu}_{\alpha\beta}(h)\varphi(1_\beta)=\widetilde{\nu}_{\alpha}(h_{(1)})\varphi(g_{\beta}h_{(2)})\ \text{for any}\ \varphi\in \ _\varepsilon H_{\beta}^*,\ h\in H_{\alpha\beta}.
\end{equation*}
This equality holds for any $\varphi\in \ _\varepsilon H_{\beta}^*$ implies that $ \widetilde{\nu}_{\alpha\beta}(h)1_\beta=\widetilde{\nu}_{\alpha}(h_{(1)})g_{\beta}h_{(2)}$, i.e. $(\widetilde{\nu}_{\alpha}\otimes g_{\beta})\Delta_{\alpha,\beta}(h)=\widetilde{\nu}_{\alpha\beta}(h)1_\beta$ for any $h\in H_{\alpha\beta}$. Therefore the family $\widetilde{\nu}=(\widetilde{\nu}_{\alpha})_{\alpha\in G}$ is the right symmetrised $G$-integral for $H$.

For the case of the left modified trace the proof is similar.
\end{proof}
\section{Modified trace for the $G$-graded quantum $\mathfrak{sl}(2)$}
In this section we present the symmetrised $G$-integral for the quantization of $\mathfrak{sl}(2)$ and the modified trace on ideal of projective modules of category of the weight modules over $\overline{\mathcal{U}}_{q}\mathfrak{sl}(2)$. It explains clearly the relation between the symmetrised $G$-integral for a pivotal Hopf $G$-coalgebra and the modified trace in associated category $\overline{\mathcal{U}}_{q}\mathfrak{sl}(2)$-mod.
\subsection{Unrestricted quantum $\overline{\mathcal{U}}_{q}\mathfrak{sl}(2)$}
Let $\mathcal{U}_{q}\mathfrak{sl}(2)$ be the $\mathbb{C}$-algebra given by generators $E, F, K, K^{-1}$ 
and relations:
\begin{align*}
  KK^{-1}&=K^{-1}K=1, & KEK^{-1}&=q^2E, & KFK^{-1}&=q^{-2}F, &
  [E,F]&=\frac{K-K^{-1}}{q-q^{-1}}
\end{align*}
where $q=e^{\frac{i\pi}{r}}$ is a $2r^{th}$-root of unity.
The algebra $\mathcal{U}_{q}\mathfrak{sl}(2)$ is a Hopf algebra where the coproduct, counit and
antipode are defined by
\begin{align*}
  \Delta(E)&= 1\otimes E + E\otimes K, 
  &\varepsilon(E)&= 0, 
  &S(E)&=-EK^{-1}, 
  \\
  \Delta(F)&=K^{-1} \otimes F + F\otimes 1,  
  &\varepsilon(F)&=0,& S(F)&=-KF,
    \\
  \Delta(K)&=K\otimes K
  &\varepsilon(K)&=1,
  & S(K)&=K^{-1}
.
\end{align*}
Let $\overline{\mathcal{U}}:=\overline{\mathcal{U}}_{q}\mathfrak{sl}(2)$ be the algebra $\mathcal{U}_{q}\mathfrak{sl}(2)$ modulo the relations
$E^r=F^r=0$ and $C=\mathbb{C}[K^{\pm r}]$ be the commutative Hopf subalgebra in the center of $\overline{\mathcal{U}}_{q}\mathfrak{sl}(2)$. The algebra $\overline{\mathcal{U}}$ is a pivotal Hopf algebra with the pivot $g=K^{1-r}$.
Let $G=(\mathbb{C}/2\mathbb{Z}, +) \xrightarrow{\sim} \Hom_{Alg}(C,\mathbb{C}),\ \overline{\alpha}\mapsto \left(K^r\mapsto q^{r\alpha}:=e^{\frac{i\pi\alpha}{r}}\right)$ and let $\mathcal{U}_{\overline{\alpha}}$ be the algebra $\overline{\mathcal{U}}_{q}\mathfrak{sl}(2)$ modulo the
relations $K^{r}=q^{r\alpha}$ for $\overline{\alpha}\in G$. 
By applying Example \ref{example 1} it follows that $\mathcal{U}=\{\mathcal{U}_{\overline{\alpha}}\}_{{\overline{\alpha}}\in G}$ is the Hopf $G$-coalgebra with the coproduct and the antipode are  determined by the commutative diagrams:
\begin{equation*}
\begin{diagram}
\overline{\mathcal{U}} &\rTo^{\Delta} &\overline{\mathcal{U}} \otimes \overline{\mathcal{U}}\\
\dTo_{p_{\overline{\alpha}+\overline{\beta}}} & &\dTo_{p_{\overline{\alpha}}\otimes p_{\overline{\beta}}}\\
\mathcal{U}_{\overline{\alpha}+\overline{\beta}} &\rTo^{\Delta_{\overline{\alpha},\overline{\beta}}} &\mathcal{U}_{\overline{\alpha}} \otimes \mathcal{U}_{\overline{\beta}}
\end{diagram} \qquad \qquad \qquad
\begin{diagram}
\overline{\mathcal{U}} &\rTo^{S} &\overline{\mathcal{U}} \\
\dTo_{p_{\overline{\alpha}}} & &\dTo_{p_{-\overline{\alpha}}}\\
\mathcal{U}_{\overline{\alpha}} &\rTo^{S_{\overline{\alpha}}} &\mathcal{U}_{-\overline{\alpha}}
\end{diagram}
\end{equation*}
where $p_{\overline{\alpha}}: \ \overline{\mathcal{U}} \rightarrow \mathcal{U}_{\overline{\alpha}}$ is the projective morphism from $\overline{\mathcal{U}}$ to $\mathcal{U}_{\overline{\alpha}}$.
The Hopf $G$-coalgebra $\mathcal{U}=\{\mathcal{U}_{\overline{\alpha}}\}_{{\overline{\alpha}}\in G}$ has the pivotal structure given by $g_{\overline{\alpha}}=q^{-r\alpha}K$.\\
For $\overline{\alpha}=\overline{0}$ the Hopf algebra $\mathcal{U}_{\overline{0}}$ is called the restricted quantum $\mathfrak{sl}(2)$, i.e. the algebra $\mathcal{U}_{q}\mathfrak{sl}(2)$ modulo the relations $E^r=F^r=0$ and $K^r=1$. The right $\overline{0}$-integral is the usual right integral given by 
\begin{equation*}
\mu_{\overline{0}}(E^m F^n K^l)=\eta\delta_{m, r-1}\delta_{n, r-1}\delta_{l,1}
\end{equation*}
where $\eta$ is a constant (see e.g \cite{BBG17}).
By definition of right $G$-integral \eqref{right int} we get
\begin{equation*}
\mu_{\overline{\alpha}}(E^m F^n K^l)=q^{r\alpha}\eta\delta_{m, r-1}\delta_{n, r-1}\delta_{l,1}.
\end{equation*}
One can show that the Hopf $G$-coalgebra $\{\mathcal{U}_{\overline{\alpha}}\}_{{\overline{\alpha}}\in G}$ is $G$-unibalanced.\\
The symmetrised right $G$-integral for $\{\mathcal{U}_{\overline{\alpha}}\}_{{\overline{\alpha}}\in G}$ is determined by
\begin{equation}\label{ct xd sym int}
\widetilde{\mu}_{\overline{\alpha}}(E^m F^n K^l)=\eta\delta_{m, r-1}\delta_{n, r-1}\delta_{l,0}.
\end{equation}
\subsection{Modified trace}
The category $\mathcal{C}=\mathcal{U}\text{-mod}$ is equal to the $G$-graded category of finite dimensional weight modules over $\overline{\mathcal{U}}_{q}\mathfrak{sl}(2)$ (module in which $K$ has a diagonalizable action).
For $\alpha\in \mathbb{C}$ let $V_{\alpha}$ be a $r$-dimensional highest weight module of highest weight $\alpha+r-1$ in $\mathcal{C}$ (see \cite{FrNaBe}).
Recall the modified dimension $d(V_{\alpha})$ of $V_{\alpha}$ for $\alpha\in (\mathbb{C}\setminus\mathbb{Z})\cup r\mathbb{Z}$ was computed:
\begin{equation}\label{modified dimension}
d(V_{\alpha})=\mtr_{V_{\alpha}}(\Id_{V_{\alpha}})=d_0\prod_{k=1}^{r-1}\frac{\{k\}}{\{\alpha+r-k\}}=d_0 \frac{r\{\alpha\}}{\{r\alpha\}}
\end{equation} 
where $\mtr$ is the modified trace on ideal $\textbf{Proj}(\mathcal{C})$ of projective modules and $d_0$ is a non-zero complex number. In \cite{FrNaBe} for the analogous unrolled category, it is normalized by  $d_0=(-1)^{r-1}$.
We now present the way to compute the modified dimension of $V_{\alpha}$ using the symmetrised $G$-integral.\\
By density theorem we have the isomorphism of algebras $$\mathcal{U}_{\overline{\alpha}}\xrightarrow{\sim} \bigoplus_{k\in H_{r}}\End(V_{\alpha+2k})$$
where $H_r=\{ -(r-1), -(r-3), ..., r-1\}$.
Hence we have the isomorphism of left $\mathcal{U}_{\overline{\alpha}}$-modules:  $$\mathcal{U}_{\overline{\alpha}}\xrightarrow{\sim} \bigoplus_{k\in H_{r}}\End(V_{\alpha+2k}) \xrightarrow{\sim} \bigoplus_{k\in H_{r}} V_{\alpha+2k} \otimes \ _\varepsilon V_{\alpha+2k}^{*}\ .$$ 
Consider the quantum Casimir element of $\overline{\mathcal{U}}$ defined by
\begin{equation*} 
  \Omega=FE+\dfrac{Kq+K^{-1}q^{-1}}{\{1\}^2}=
  EF+\dfrac{Kq^{-1}+K^{-1}q}{\{1\}^2}.
\end{equation*} 
For $k\in \mathbb{N}$, by induction one gets
\begin{equation}\label{luy thua pt Casimir}
\prod_{i=0}^{k-1}\left(\Omega - \dfrac{q^{-2i-1}K+q^{2i+1}K^{-1}}{\{1\}^2}\right)=E^kF^k.
\end{equation}
\begin{Lem}\label{luy thua casimir1}
For $k\in \mathbb{N}$ then
\begin{equation*}
\Omega^k-E^kF^k \in \text{Span}_{\mathbb{C}}\{E^jF^jK^i \ | \ j<k,\ i\in \mathbb{Z}\}.
\end{equation*}
\end{Lem}
\begin{proof}
The proof is by induction on $k$. Indeed, by \eqref{luy thua pt Casimir} $\Omega^k-E^kF^k \in \text{Span}_{\mathbb{C}}\{\Omega^jK^i \ | \ j<k,\ i\in \mathbb{Z}\}$ which by the induction hypothesis is contained in $\text{Span}_{\mathbb{C}}\{E^jF^jK^i \ | \ j<k,\ i\in \mathbb{Z}\}$.
\end{proof}
Following \eqref{ct xd sym int} we have the corollary.
\begin{HQ}\label{bd tinh sym sl2}
For all $k\in \{0, ..., r-2\}$ we have $\widetilde{\mu}_{\overline{\alpha}}\left(\Omega^k\right)=0$. For $k=r-1$ then $\widetilde{\mu}_{\overline{\alpha}}\left(\Omega^{r-1}\right)=\eta$.
\end{HQ}
\begin{proof}
It follows from \eqref{ct xd sym int} that $\text{Span}_{\mathbb{C}}\{E^jF^jK^i \ | \ j<k,\ i\in \mathbb{Z}\}$ is contained in the kernel of $\widetilde{\mu}_{\overline{\alpha}}$ for $k\in \{0, ..., r-2\}$.
\end{proof}
For $\alpha\in \mathbb{C}\setminus \mathbb{Z}$, $\Omega$ acts on $V_{\alpha}$ by the scalar $w_{\alpha}$ which is calculated as follows:
Let $v$ be a highest weight vector of $V_{\alpha}$. The action of $K$ on $v$ defined by $Kv=q^{\alpha+r-1}v$. This implies that $\Omega v=\frac{q^{\alpha+r}+q^{-\alpha-r}}{\{1\}^2} v$, i.e. $w_{\alpha}=\frac{q^{\alpha+r}+q^{-\alpha-r}}{\{1\}^2}$.
The elements $w_{\alpha+2k}, \ 0\leq k< r-1$ are distinct as $w_{\alpha+2i}-w_{\alpha+2j}=\frac{\{i-j\}\{\alpha+r+i+j\}}{\{1\}^2} \neq 0$ for $i\neq j$. \\
We consider in $\mathcal{U}_{\overline{\alpha}}$ the element 
$$L_{\alpha}(\Omega)=\frac{\prod_{k=1}^{r-1}(\Omega-w_{\alpha+2k})}{\prod_{k=1}^{r-1}(w_{\alpha}-w_{\alpha+2k})}.$$
This element is the projector on ${V_{\alpha} \otimes \ _\varepsilon V_{\alpha}^{*}}\simeq \bigoplus_{k=1}^{r}V_{\alpha}$ as $L_{\alpha}(w_{\alpha+2k})=\delta_{0,k}$. The value of symmetrised right $G$-integral on $L_{\alpha}(\Omega)$ is
\begin{align*}
\widetilde{\mu}_{\overline{\alpha}}\left(L_{\alpha}(\Omega)\right)=\dfrac{1}{\prod_{k=1}^{r-1}(w_{\alpha}-w_{\alpha+2k})}\widetilde{\mu}_{\overline{\alpha}}\left(\prod_{k=1}^{r-1}(\Omega-w_{\alpha+2k})\right).
\end{align*}
Corollary \ref{bd tinh sym sl2} implies that $$\widetilde{\mu}_{\overline{\alpha}}\left(\prod_{k=1}^{r-1}(\Omega-w_{\alpha+2k})\right)=\widetilde{\mu}_{\overline{\alpha}}\left(\Omega^{r-1}\right)=\eta.$$
The equality $\prod_{k=1}^{r-1}(w_{\alpha}-w_{\alpha+2k})=(-1)^{r-1}\prod_{k=1}^{r-1}\frac{\{k\}\{\alpha+k\}}{\{1\}^2}$ gives 
\begin{align*}
\widetilde{\mu}_{\overline{\alpha}}\left(L_{\alpha}(\Omega)\right)&=(-1)^{r-1}\eta\prod_{k=1}^{r-1}\frac{\{1\}^2}{\{k\}\{\alpha+k\}}\\
&=\eta\prod_{k=1}^{r-1}\frac{\{1\}^2}{\{k\}^2} (-1)^{r-1}\prod_{k=1}^{r-1}\frac{\{k\}}{\{\alpha+r-k\}}
= \frac{\{1\}^{2r-2}\eta}{r^3 d_0} r d(V_{\alpha})
\end{align*}
where we used the identity $\prod_{k=1}^{r-1}\{k\}^2=(-1)^{r-1}r^2$ in the last equality.\\
It is clear that the coefficient $ \frac{\{1\}^{2r-2}\eta}{r^3 d_0}$ does not depend on $\alpha$. This proves that $\widetilde{\mu}_{\overline{\alpha}}\left(L_{\alpha}(\Omega)\right)=rd(V_{\alpha})$ with the choice $d_0= \frac{\{1\}^{2r-2}\eta}{r^3}$ where $\eta=\widetilde{\mu}_{\overline{\alpha}}\left(E^{r-1}F^{r-1}\right)$.
\bibliographystyle{plain} 
\bibliography{Reference}

\begin{thebibliography}{10}

\bibitem{BBG17}
A.~Beliakova, C.~Blanchet, and A.~M. Gainutdinov.
\newblock Modified trace is a symmetrised integral.
\newblock {\em arXiv:1801.00321}, 2017.

\bibitem{BBGe17}
A.~Beliakova, C.~Blanchet, and N.~Geer.
\newblock Logarithmic hennings invariants for restricted quantum
  $\mathfrak{sl}(2)$.
\newblock {\em arXiv:1705.03083}, 2017.

\bibitem{KaDe92}
C.~De Concini, V.G. Kac, and C.~Procesi.
\newblock Quantum coadjoint action.
\newblock {\em J. Amer. Math. Soc.}, 5, No. 1:151--189, 1992.

\bibitem{FcNgBp14}
F.~Costantino, N.~Geer, and B.~Patureau-Mirand.
\newblock Quantum invariants of 3-manifolds via link surgery presentations and
  non-semi-simple categories.
\newblock {\em Journal of Topology}, pages 1005--1053, 2014.

\bibitem{FrNaBe}
F.~Costantino, N.~Geer, and B.~Patureau-Mirand.
\newblock Some remarks on the unrolled quantum group of $\mathfrak{sl}(2)$.
\newblock {\em J. Pure Appl. Algebra}, pages 3238--3262, 2015.

\bibitem{GR17}
A.~M. Gainutdinov and I.~Runkel.
\newblock Projective objects and the modified trace in factorisable finite
  tensor categories.
\newblock {\em arXiv:1703.00150}, 2017.

\bibitem{NgJkBp13}
N.~Geer, J.~Kujawa, and B.~Patureau-Mirand.
\newblock Ambidextrous objects and trace fuctions for nonsemisimple categories.
\newblock {\em Proceedings of the American Mathematical Society 141}, 2013.

\bibitem{GPV13}
N.~Geer, B.~Patureau-Mirand, and A.~Virelizier.
\newblock Traces on ideals in pivotals categories.
\newblock {\em Quantum Topology}, 4, No. 1:91--124, 2013.

\bibitem{Ha16}
N.~P. Ha.
\newblock Topological invariants from quantum group $\mathcal {U}_{\xi
  }\mathfrak {sl}(2|1)$ at roots of unity.
\newblock {\em Abhandlungen aus dem Mathematischen Seminar der Universität
  Hamburg}, pages 1--26, 2017.

\bibitem{Henning96}
M.~Hennings.
\newblock Invariants of links and $3$-manifolds obtained from hopf algebras.
\newblock {\em Journal of the London Mathematical Society}, 54:594--624, 1996.

\bibitem{ChKa95}
C.~Kassel.
\newblock {\em Quantum Groups}.
\newblock Springer-Verlag, 1995.

\bibitem{LaSweedler69}
R.~G. Larson and M.~E. Sweedler.
\newblock An associative orthogonal bilinear form for hopf algebras.
\newblock {\em American Journal of Mathematics}, Vol. 91, No. 1 (Jan.):75--94,
  1969.

\bibitem{BePM12}
B.~Patureau-Mirand.
\newblock {\em Invariants topologiques quantiques non semi-simples}.
\newblock Habilitation \`a diriger des recherches, Universite de Bretagne Sud,
  2012.

\bibitem{Ra12}
David~E Radford.
\newblock {\em Hopf Algebras}.
\newblock World Scientific, 2012.

\bibitem{Turaev10}
V.~Turaev.
\newblock {\em Homotopy Quantum Field Theory}.
\newblock European Mathematical Society, 2010.

\bibitem{Vire02}
A.~Virelizier.
\newblock Hopf group-coalgebra.
\newblock {\em Journal of pure and applied algebra}, 171:75--122, 2002.

\end{thebibliography}

\end{document}